\documentclass[12pt]{amsart}

\usepackage{amssymb,latexsym}
\usepackage{amsfonts}
\usepackage{amsmath}
\usepackage{enumerate}

\usepackage{booktabs}
\usepackage{enumerate}
\usepackage{rotating}
\usepackage{lscape}

\usepackage[usenames, dvipsnames]{xcolor}

\usepackage{subcaption}

\usepackage[colorlinks,linkcolor=blue,anchorcolor=blue,citecolor=blue]{hyperref}
\usepackage{tikz}
\usepackage{tikz-qtree}

\usepackage{cleveref}

\newcommand\C{{\mathbb C}}
\newcommand\Q{{\mathbb Q}}
\newcommand\R{{\mathbb R}}
\newcommand\N{{\mathbb N}}
\newcommand\Z{{\mathbb Z}}

\newcommand\DD{{\mathcal D}}
\newcommand\LL{{\mathcal L}}
\newcommand\NN{{\mathcal{N}}}

\newcommand\ord[1]{\mathrm{ord}(#1)}

\DeclareMathOperator{\RotOp}{Rot}
\newcommand\Rot[2]{\RotOp^{#2}\left[#1\right]}

\newcommand\al{\alpha}
\newcommand\be{\beta}
\newcommand\ga{\gamma}

\newcommand\eps{\varepsilon}
\numberwithin{equation}{section}
\numberwithin{figure}{section}
\numberwithin{table}{section}

\newcommand\variance[1]{\text{var}\left(#1\right)}

\newtheorem{theorem}{Theorem}[section]
\newtheorem{lemma}[theorem]{Lemma}

\newtheorem{corollary}[theorem]{Corollary}
\newtheorem{proposition}[theorem]{Proposition}
\newtheorem{conjecture}[theorem]{Conjecture}

\theoremstyle{definition}
\newtheorem{definition}[theorem]{Definition}
\newtheorem{example}[theorem]{Example}

\theoremstyle{remark}

\numberwithin{equation}{section}

\newcommand{\abs}[1]{\lvert#1\rvert}

\begin{document}

\title[Zeros of $\{0, 1\}$  and $\{-1, 1\}$ polynomials inside disk]{On Newman and Littlewood polynomials with prescribed number of zeros inside the unit disk}

\author{Kevin G. Hare}
\address{Department of Pure Mathematics, University of Waterloo, Waterloo, Ontario, Canada N2L 3G1}
\email{kghare@uwaterloo.ca}
\thanks{The research of Kevin G. Hare is supported in part by NSERC grant 2019-03930}

\author{Jonas Jankauskas}
\address{Mathematik und Statistik, Montanuniversität Leoben, Franz Josef Stra{\ss}e 11, 8700 Leoben, Austria}
\email{jonas.jankauskas@gmail.com}

\thanks{}

\thanks{The post-doctoral position of Jonas Jankauskas is supported by the Austrian Science Fund (FWF) project M2259 Digit Systems, Spectra and Rational Tiles under the Lise Meitner Program.}



\begin{abstract}
We study $\{0, 1\}$ and $\{-1, 1\}$ polynomials $f(z)$, called Newman and Littlewood polynomials, that have a prescribed number $N(f)$ of zeros in the open unit disk $\DD = \{z \in \C: \abs{z} < 1\}$. For every pair $(k, n) \in \N^2$, where $n \geq 7$ and $k \in [3, n-3]$, we prove that it is possible to find a $\{0, 1\}$--polynomial $f(z)$ of degree $\deg{f}=n$ with non--zero constant term $f(0) \ne 0$, such that $N(f)=k$ and $f(z) \ne 0$ on the unit circle $\partial\DD$.  On the way to this goal, we answer a question of D.~W.~Boyd from 1986 on the smallest degree Newman polynomial that satisfies $\abs{f(z)} > 2$ on the unit circle $\partial \DD$. This polynomial is of degree $38$ and we use this special polynomial in our constructions.  We also identify (without a proof) all exceptional $(k, n)$ with $k \in \{1, 2, 3, n-3, n-2, n-1\}$, for which no such $\{0, 1\}$--polynomial of degree $n$ exists: such pairs are related to regular (real and complex) Pisot numbers.

Similar, but less complete results for $\{-1, 1\}$ polynomials are established. We also look at the products of spaced Newman polynomials and consider the rotated large Littlewood polynomials. Lastly, based on our data, we formulate a natural conjecture about the statistical distribution of $N(f)$ in the set of Newman and Littlewood polynomials.
\end{abstract}

\maketitle

\section{Introduction}\label{intro}

Throughout the paper, we will write
\begin{equation}\label{genForm}
f(z) = a_0 + a_1 z + \dots + a_n z^n
\end{equation}
for a polynomial with coefficients $a_j \in \Z$, $\R$ or $\C$; we will assume that the highest order term $a_n \ne 0$, so that $\deg{f} = n$. We are interested in two important classes of polynomials $f(z)$ whose coefficients are restricted to very small sets containing only two integer numbers, namely, Newman and Littlewood polynomials of degree $n$:
\begin{equation}\label{defNL}
\begin{array}{rcl}
\NN_n &=& \{f(z) \in \Z[z]: a_0 = a_n = 1, a_j \in \{0, 1\}, 1 \leq j \leq n\},\\
           &  &\\
\LL_n  &=& \{f(z) \in \Z[z]: a_j \in \{-1,1\},  1 \leq j \leq n\}.
\end{array}
\end{equation}
The set of all possible Newman and Littlewood polynomials, named after \cite{New1, New2} and \cite{Littl1, Littl2, Littl3, Littl4, Littl5}, respectively, will be denoted
\[
\NN = \bigcup_{n=1}^{\infty}{\NN_n}, \qquad \LL = \bigcup_{n=1}^{\infty}{\LL_n}.
\]
As any polynomial $f(z) \in \C[z]$, Newman and Littlewood polynomials factor over $\C$ into
\begin{equation}\label{eqFactorizes}
f(z) = a_n(z-\al_1)\cdots (z-\al_n),
\end{equation}
where the complex zeros $\al_1, \al_2, \dots, \al_n$ of $f(z)$ are not necessarily distinct.

The sets of numbers that are the zeros of Newman and Littlewood polynomials have received a lot of attention in the contemporary mathematical literature. Despite the deceiving simplicity of $\NN$ and $\LL$, the zero sets exhibit complicated, fractal--like geometry \cite{BEL, BoPi, OdlPoo, Solo}. For number theorists, they are valuable sources of algebraic integers of small height \cite{BoDoMo, BoHaMo, MoPiVa} and algebraic integers with special properties, like real or complex Pisot and Salem numbers \cite{G1, G2, Mu1, Mu2, Mu3}. The set structure and the relationship between the zero sets of polynomials $\NN$ and $\LL$ is rather intricate, see \cite{DJ, DJS, DJKJ, Saund}. Due to pioneering work by Schur, Szeg\H{o}, Erd\H{o}s, Tur\'{a}n, Bloch, P\'{o}lya, Littlewood, Offord, subsequent contributions by Borwein, Boyd, Erd{\'e}lyi and many other authors, the key properties of these zero sets have been determined. For instance, the angular equidistribution property \cite{AmoMig, Erd3, ET, Gan, Mign, Sound} of zeros, the range of absolute values of zeros compared to their argument angle or zero multiplicity \cite{BBBP1, BBBP2, Solo}. Asymptotic bounds  for the number of real zeros \cite{BEK, BloPoy, Erd2, Sch, Sze}, their expected  values \cite{LitOff1, LitOff2, EO} are available; as well as the concentration inequalities  for the probability to have a zero at a fixed point \cite{LitOff3}; also asymptotic estimates on the number of zeros in circular and polygonal regions \cite{BE1, BEL, Erd1}; bounds \cite{Amor} for the maximal vanishing multiplicity at roots of unity (especially at point $z=1$, see \cite{BoMo1, B7, B8, Dub1, Dub2}), etc.

Let us denote the open unit disk of radius one, centered at the origin and its circular boundary by
\[
\DD = \{z \in \C, \abs{z} < 1\}, \qquad \text{ and } \qquad \partial\DD=\{z \in \C: \abs{z}=1\}, 
\] 
respectively. We will refer to complex numbers $z \in \partial\DD$ as \emph{unimodular}. Recall that  a \emph{reciprocal polynomial} of $f(z) \in \C[z]$ is defined by $f^*(z) := z^{\deg{f}}\overline{f}(z^{-1})$, or $f^*(z) = \overline{a}_n + \overline{a}_{n-1}z+\dots+\overline{a}_0z^n.$

When $f(z) \in \R[z]$, the reciprocation simply reflects the coefficients of $f(z)$ backwards.  Whenever $f^*(z)=f(z)$,  $f(z)$ is called \emph{self-reciprocal} (or \emph{symmetric}, if $f(z) \in \R[z]$).  We call a polynomial $f(z)$ \emph{negative self--reciprocal} (\emph{antisymmetric}, resp.), whenever  $f^*(z)=-f(z)$. Polynomials that satisfy $f^*(z)=\pm f(-z)$ are referred as \emph{skew--reciprocal} (resp. \emph{skew--symmetric}). For unimodular $z=e^{it}$, $t \in [0, 2\pi)$ and self--reciprocal $f(z)$ of even degree $n=2l$,
\[
f(e^{it})e^{-ilt/2} =a_{l}+2a_{l-1}\cos{t}+\dots+2a_0\cos{(lt)}.
\]
Similar trigonometric expressions exist for self--reciprocal and negative self--reciprocal polynomials of even and odd degrees. Due to these relations, the unimodular zeros of self--reciprocal polynomials are of  interest for harmonic analysts, especially in conjunction with topics of the positivity and $L^p$--norm estimation, see, for instance, \cite{Erd6, Littl2, Littl4, Littl5}. For this purpose, let us define the function
\[
U(f) := \# \{j, 1 \leq j \leq n: |\al_j| = 1\},
\] that counts the unimodular zeros of $f(z)$ in Eq. \eqref{eqFactorizes}. For simplicity, each zero will be counted the number of times it is repeated in equation \eqref{eqFactorizes}; although the subject of repeated zeros will occur through the paper briefly and only in a very special situation.

Over last $15$ years, many results have been established on $U(f)$ for polynomials $f \in \NN$ and $\LL$. The first essential fact, that $U(f) \geq 1$ holds for all self-reciprocal $\{0, 1\}$ and $\{-1, 1\}$ polynomials, was proved (using different methods) by Konyagin and Lev \cite{KonLev}, Kovalina and Matache \cite{KovMat}, Mercer \cite{Mer} and Erd\'elyi \cite{Erd2}. For self--reciprocal $f \in \NN \cup \LL$ of odd degree $\geq 3$,  Mukunda \cite{Mu1} raised the lower bound to $U(f) \geq 3$. Later, specifically for $f \in \LL_n$, Drungilas improved the lower bound to $U(f) \geq 4$ for $n=2l \geq 14$ and $U(f) \geq 5$ for $n =2l+1\geq 7$. Conrey, Granville, Poonen and Soundararajan \cite{CGPS} proved that, for Fekete polynomials $f_p \in \LL_p$ of prime degree $p$,  $U(f_p)=\kappa_0p$, where the constant factor $\kappa_0 \in [0.5006, 0.5008]$. In the opposite direction, Mercer proved \cite{Mer} that $U(f)=0$ for skew--symmetric $f \in \LL_n$. Smyth \cite{Smy1} and Boyd \cite{B8} proved the existence of  Newman polynomials with $U(f)=0$ for every $n > n_0$; Mercer  \cite{Mer2, Mer3} established that one can take $n_0 =2$ and calculated the proportion of Newman polynomials with $3$ and $4$ terms that satisfy $U(f) \ne 0$. Dubickas proved that certain $f \in \NN_n$ have $U(f)=0$ in \cite{Dub3}. By using the probabilistic method, Borwein, Erd\'elyi, Fergusson and Lockhart \cite{BEFL} proved that there exists infinitely many Newman polynomials with $U(f)$ as low as $O(n^{5/6}\log{n})$, settling a question posed by Littlewood in \cite{Littl5} negatively.

Very recently, new major results on $U(f)$ appeared.  Erd{\'e}lyi \cite{Erd4} proved that, for any sequence of self-reciprocal polynomials $\{f_j\}$ of increasing degree with coefficients from arbitrary finite subset $\mathcal{S} \subset \R$, one has $\lim_{j\to \infty}U(f_j)=+\infty$ whenever $\lim_{j \to \infty} \abs{f_j(1)} = +\infty$ (which is exactly the case for Newman polynomials with incresing number of non-zero terms, but not for Littlewood polynomials). In the earlier paper of Borwein and Erd\'elyi \cite{BE2} the same result was shown to hold, under additional restriction that the coefficients of $f_j(z)$ do not constitute an ultimately periodic sequence. Almost simultaneously when Erd\'elyi announced his result \cite{Erd4}, Sahasrabudhe \cite{Saha} achieved a breakthrough by establishing an effective inequality $U(f) \geq c \left(\log{\log{\log{\abs{f(1)}}}}\right)^{1/2 - \eps}-1$, for any $f$ with coefficients from a finite set $\mathcal{S} \subset \Z$ and effective constant $c=c(\mathcal{S})$. Erd\'elyi \cite{Erd5} quickly improved $1/2$ in the exponent of Sahasrabudhe's bound to $1$ by combining the strengths of different methods used in \cite{BE2, Erd2, Saha}. For a more in--depth account refer to \cite{Erd6}.

Let us now define another function
\[
N(f) := \# \{j, 1 \leq j \leq n: |\al_j| < 1\}
\] 
that counts (with multiplicities) the number of zeros of a polynomial $f(z) \in \C[z]$ in equation \eqref{genForm} in the open unit disk. Note that, in general, one always has $N(f)+N(f^*)+U(f) = n.$

In \cite{BCFJ}, Borwein, Choi, Fergusson and the second named author of the present paper initiated the study of the $U(f)$ and $N(f)$ for the class of Littlewood polynomials. In Problem 1.1  \cite{BCFJ}, they asked for which pair of $(k, n)$ with $1 \leq k \leq n-1$, one can find a polynomial $f(z) \in \LL$, such that $N(f)=k$ and $U(f)=0$? (As it is explained in \cite{BCFJ}, the cases $k=0$, $k=n$ are trivial.) Taking a similar approach as in \cite{MoPiVa}, authors of \cite{BCFJ} looked at the Littlewood polynomials that can be obtained from the simplest \emph{sum--of--geometric progression} polynomial by a simple perturbation (one sign change, or one term negation):
\[
1+z+\dots+z^k-z^{k+1}\dots -z^n, \qquad 1+\dots+z^k-z^{k+1}+z^{k+1}+\dots+z^n,
\] and calculated $N(f)$ and $U(f)$ for these perturbed Littlewood polynomials.

In the present paper, we will continue to study the problem posed in \cite{BCFJ} of constructing polynomials with restricted coefficients that have prescribed $N(f)$ and $U(f)$. In addition to $\LL$, we will also look at the class $\NN$  of Newman polynomials. It is convenient to introduce the following notation.

\begin{definition}\label{defAdm}
A pair $(k, n) \in \N^2$, where $1 \leq k \leq n-1$, will be called \emph{a valid pair}. A valid pair $(k, n)$ is \emph{Newman--admissible}, if there exist a polynomial $f(z) \in \NN_n$, not vanishing on the unit circle $\abs{z}=1$, such that $N(f)=k$. Likewise, a valid pair $(k, n) \in \N^2$ is \emph{Littlewood--admissible} if there exist a polynomial $f(z) \in \LL_n$, such that  $U(f)=0$, $N(f)=k$. We say that this polynomial $f(z)$ \emph{realizes} that particular pair $(k, n)$.
\end{definition}

Instead of perturbing cyclotomic polynomials like in \cite{BCFJ}, our main method will be to look for $\{0, 1\}$ and $\{-1, 1\}$ coefficient patterns that produce polynomials $f \in \NN$ or $\LL$ with predictable changes in $N(f)$ and $U(f)$ as the length of the pattern increases. In some cases, such patterns will be constructed explicitly from polynomials $f(z)$ that stay large on $\partial\DD$. We make a heavy use of computers to conduct large scale pattern searches and even to prove the $N(f)$ and $U(f)$ formulas for certain infinite sequences of polynomials $f(z)$.

We conclude this section with two comments to justify the normalizations we made in the admissibility problem under the consideration.

Our first remark is about the treatment of $z=0$ in the Newman case. Notice that, according to our definition of $\NN_n$ \eqref{defNL}, Newman polynomials  $f(z)$ have constant terms $a_0=1$.  Allowing $f(0)=0$ would make the problem of manufacturing $f \in \NN$ with prescribed $N(f)=k$ much easier. However, polynomials $f(z) \in \NN$ obtained this way would be of a little value in applications, like producing algebraic integers with a prescribed number of conjugates in $\DD$, or for the purpose of $L^p$ norm estimates: any non--zero $\{0, 1\}$ polynomial factors as $f(z)=z^mg(z)$, where $m \in \N$ and $g \in \NN$,  satisfy $g(0) \ne 0$. All the interesting information about $\abs{f(z)}$ comes from $g(z)$. Removing `a singularity' at $z=0$ from the set of zeros of $\{0, 1\}$ also has other benefits: it makes $\NN_n$ invariant under inversion and symmetrizes the distributions of $N(f)$, as described in Section \ref{sec_stats}.

Our second remark concerns the condition $U(f)=0$, (introduced in \cite{BCFJ}) that we continue to use in our Definition \ref{defAdm}. This condition arises  mainly from our current deficiency to handle zeros on $\partial\DD$. In the classical control theory, polynomials with unimodular roots are called singular: all root--counting methods, such as Schur-Cohn rule, Jury rule or Bistritz rule, singular polynomials require special, out--of--flow treatment \cite{Bis}. From the results of \cite{BE2, Erd2, Erd5, Erd6, Saha} that were discussed above, it is clear that presently there are no known simple method to control $U(f)$ in Newman and Littlewood polynomials with the same finesse as $N(f)$.

\newpage

\section{Main results}\label{sec_main}

\subsection{Newman polynomials}\label{sec_Newman}

First we take a careful look into the set of Newman polynomials.  Our principal result for the class $\NN$ is that \emph{most} valid pairs $(k, n) \in \N^2$ are Newman-admissible.

\begin{theorem}\label{mainNewmanThm}
Every valid pair $(k, n)$ with $n \geq 7$ and $k \in [3, n-3]$ is Newman--admissible.
\end{theorem}

On the way to prove Theorem \ref{mainNewmanThm}, we will need to find a polynomial $f(z) \in \NN$ -- preferably, of the smallest possible degree -- such that $\abs{f(z)} > 2$ for every $z \in \partial\DD$. Such polynomial is
\begin{align} \label{superNewman}
\begin{array}{ll}
f(z) &= 1 + z^{4} + z^{6} + z^{7} + z^{8} + z^{10} + z^{11} + z^{12} + z^{15} + z^{16} \\
     &+ z^{17} + z^{22} + z^{24} + z^{25} + z^{26} + z^{29} + z^{32} + z^{35} + z^{38}. 
\end{array}
\end{align}
The minima of $|f(z)|$ on the unit circle is $m=2.0181\dots$. We will use this particular $f(z)$ in the proof of Theorem \ref{mainNewmanThm} to cover \emph{most} of the admissible pairs. This answers a question of Boyd from 1986 paper \cite{B6}-- see Section \ref{sec_largeMin} for more details. 

From Theorem \ref{mainNewmanThm}, it follows that for a non-admissible pair, $k$ or $n-k$ must be very small. The cases $k=1$ and $k=n-1$ were settled completely by Mukunda \cite{Mu3}; here we re-state them for the sake of completeness only.

\begin{proposition}[Mukunda]\label{Prop_Mu_Newman}
The pairs $(1, 2m+1)$ and $(2m, 2m+1)$, $m \in \N$ are Newman-admissible, while pairs $(1, 2m)$ and $(2m-1, 2m)$, $m \in \N$ are not.
\end{proposition}

The case $k \in \{2, n-2\}$ is related to \emph{complex Pisot numbers} and their inverses. This will be covered in Section \ref{sec_complex_Pisot} in great detail. Meanwhile, we state the following:

\begin{theorem}\label{Thm_k2_Newman}
If 
\begin{enumerate}
\item  $3 \leq n\leq 20$, or
\label{k2_Newman_case_1}
\item  $n > 20$ and $n \not \equiv 1, 7, 15, 21, 25, 27 \pmod{30}$
\label{k2_Newman_case_2}
\end{enumerate}
then the pairs $(2, n)$ and $(n-2, n)$ are Newman--admissible.
\end{theorem}

The reason behind the admissibility in Case \ref{k2_Newman_case_1} of Theorem \ref{Thm_k2_Newman} is the existence of many sporadic examples  of small degree with $k=2$ and $k=n-2$. The admissibility in Case \ref{k2_Newman_case_2} follows from the existence of $10$ sequences of polynomials $f(z) \in \NN$ with $N(f)=2$ that are listed in Table \ref{patternsNewman2}. These are reciprocal to the minimal polynomials of regular complex--Pisot numbers, see Section \ref{sec_complex_Pisot}.

Next, we list all pairs $(k, n)$ of small degree that are not admissible in a non--trivial way. All those pairs were found with a computer.

\begin{theorem}\label{notNewmanAdmThm} Among pairs $(k, n)$ with $k \in \{3, n-3\}$, $n \geq 4$, only one pair $(3, 6)$ is not Newman--admissible. For $k \in \{2, n-2\}$, inadmissible pairs of degree  $n \leq 35$ are:
\[
(2, 21),  (2, 25), (2, 27),  (2, 31), (19, 21), (23, 25), (25, 27), (29, 31).
\]
Besides the exceptions listed above, every other pair $(k, n)$ with $ 2 \leq k \leq n-2$ and $n \leq 35$ is Newman--admissible.
\end{theorem}
Inadmissibility of $(3, 6)$ explains the lower bound $n \geq 7$ in Theorem \ref{mainNewmanThm}. The list of inadmissible pairs for $k=2$ (and $k=n-2$) in Theorem \ref{notNewmanAdmThm}, together with the list of exceptional congruences in Theorem \ref{Thm_k2_Newman} yield evidence towards the following conjecture:
\begin{conjecture}\label{conjNewman}
Pairs $(2, n)$ and $(n-2, n)$, where $n \in \N$ satisfies
\[
n \equiv 1, 7, 15, 21, 25, 27 \pmod{30}, 
\]
form the complete set elements of $(k, n)$ and $(n-k, n) \in \N^2$ with $k \not\in \{0, 1, n-1, n\}$ that are not Newman admissible for $n \geq 21$.
\end{conjecture}

Figure \ref{gridN} depicts the situation described above.

\subsection{Littlewood polynomials}\label{sec_Littlewood}

In the Littlewood case our results are less complete than for Newman polynomials: we are short of regular $\{-1, 1\}$--patterns that produce polynomials $f(z) \in \LL_n$ with regular, predictable behaviour of $N(f)$ and freedom of control in both the zero--number $k$ and the degree $n$.

As in the Newman case, the admissibility for $k \in \{1, n-1\}$ follows from the classification of Littlewood--Pisot polynomials due to Mukunda \cite{Mu1, Mu2}:

\begin{proposition}[Mukunda]\label{Prop_Mu_Littlewood}
Pairs $(1, n)$ and $(n-1, n)$, $n \geq 2$ are Littlewood--admissible.
\end{proposition}

For $k \in \{2, n-2\}$, we have a new result:

\begin{theorem}\label{Thm_k2_Littlewood}
If $n \geq 3$ and 
\begin{enumerate}
\item $n \leq 12$, or
\label{k2_Littlewood_case_1}
\item $n \not\equiv 1 \pmod{6}$
\label{k_2Littlewood_case_2}
\end{enumerate}
then the pairs $(2, n)$ and $(n-2, n)$ are Littlewood--admissible.
\end{theorem}

We also prove the admissibility of pairs $(k, n)$ and $(n-k, n)$ for other small $k$:

\begin{theorem}\label{Thm_k3to11_Littlewood}
Every pair $(k, n), (n-k, n) \in \N^2$ with $3 \leq k \leq 11$, $n \geq k+3$ is Littlewood--admissible.
\end{theorem}

A significant part of remaining pairs with $k \in [12, n-12]$ are already known to be admissible \cite{BCFJ}.  By Theorem 2.1 of \cite{BCFJ}, every valid pair $(k, n)$, where $n \geq 2k$ and $\gcd{(k, n+1)}=1$,  or $n \leq 2k$, and $\gcd{(k+1, n+1)}=1$ is Littlewood admissible: they are realized by $f(z) \in \LL_n$ with one sign change in the coefficient pattern. Also, points $(k, n) \in \N^2$ lying on each of the $7$ lines $n=2k$, $2k \pm 1$, $2k \pm2$, $2k  \pm 3$, correspond to Littlewood--admissible pairs. Points where $n=2k-4$, $k \not\equiv 0 \pmod{6}$ and points where $n=2k+4$, $k \not\equiv 2 \pmod{6}$ are admissible as well. These results can be extracted from Proposition 2.5 and Corollary 2.8 Part (ii)  of \cite{BCFJ}: such pairs are covered by $f(z) \in \LL_n$ with one negative coefficient near the middle term.

Our main new result regarding $f(z) \in \LL_n$ is

\begin{theorem}\label{mainLittlewoodThm}
Let $n \geq 6$ be even. Then, for every odd $k \in [3, n-3]$ and every even $k \in [n/3, 2n/3]$, the valid pair $(k, n)$ is Littlewood-admissible.

Likewise, suppose that $n >  6$ is odd. Then, for every odd $k \in [3, n/2]$ and every even $k \in [n/2, n-3]$, the valid pair $(k, n)$ is Littlewood-admissible.
\end{theorem}

This result will be derived by counting the zeros of the Littlewood polynomials that originate from a certain pattern:

\begin{theorem}\label{thmSPL}
Let $m$, $l \in \N$, $m \geq 1$, $l \geq 0$. Let $h(z) \in \LL_{l+3m}$ be defined by the pattern $(+-+)^m+(-)^l$:
\begin{equation}\label{eqSPL}
h(z) = (1-z+z^2)\cdot\frac{1-z^{3m}}{1-z^3} + z^{3m}-z^{3m+1} \cdot \frac{1-z^l}{1-z}.
\end{equation}
If $l \ne m+1$, then $U(h) =0$ and
\[
N(h) =	\begin{cases}
			2m,         & \text{for } l < m + 1\\
			2m+1,  & \text{for }l > m + 1
		\end{cases}.
\]
\end{theorem}

Formulas of Theorem \ref{thmSPL} yield an important corollary:

\begin{corollary}\label{colSPL}
Let $n \geq 2$. For every even $k \in [n/2, 2n/3]$ and for every odd integer $k \in [3, n/2]$, one of the polynomials $h(z) \in \LL_n$ described in Theorem \ref{thmSPL} satisfies $N(h)=k$, $U(h)=0$.
\end{corollary}

Now we turn our attention to non--admissible cases. Our computations show the following:

\begin{theorem}\label{thm_inadm_Littl} Pairs
\[
(2, 13),  (2, 19),  (2, 25),  (2, 31),  (11, 13), (17, 19), (23, 25), (29, 31)
\] are not Littlewood--admissible. Apart from these exceptions, for $n \leq 31$, every other valid pair is Littlewood--admissible.
\end{theorem}

We conjecture that

\begin{conjecture}\label{conjLittl} Pairs $(2, n)$ and $(n-2, n)$, where $n \equiv 1 \pmod{6}$ and $n \geq 13$,  are the only non--trivial Littlewood--inadmissible pairs.
\end{conjecture}

Figure \ref{gridL} depicts our new results in the Littlewood case in conjunction with earlier results from \cite{BCFJ}.

\begin{figure}
\centering
\subcaptionbox{$f \in \NN_n$ \label{gridN}}{\includegraphics[width=0.75\linewidth]{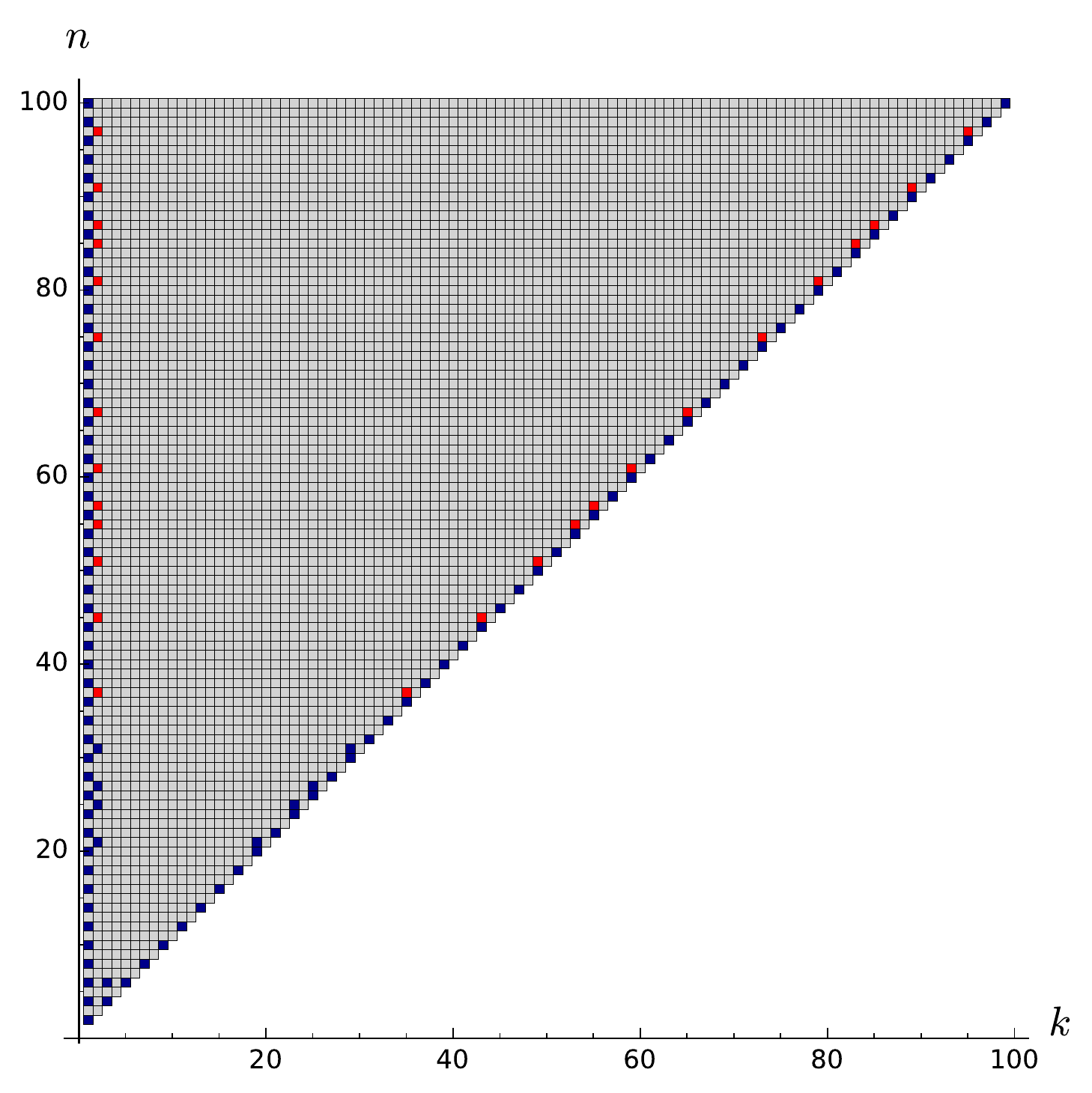}}
\subcaptionbox{$f \in \LL_n$ \label{gridL}}{\includegraphics[width=0.75\linewidth]{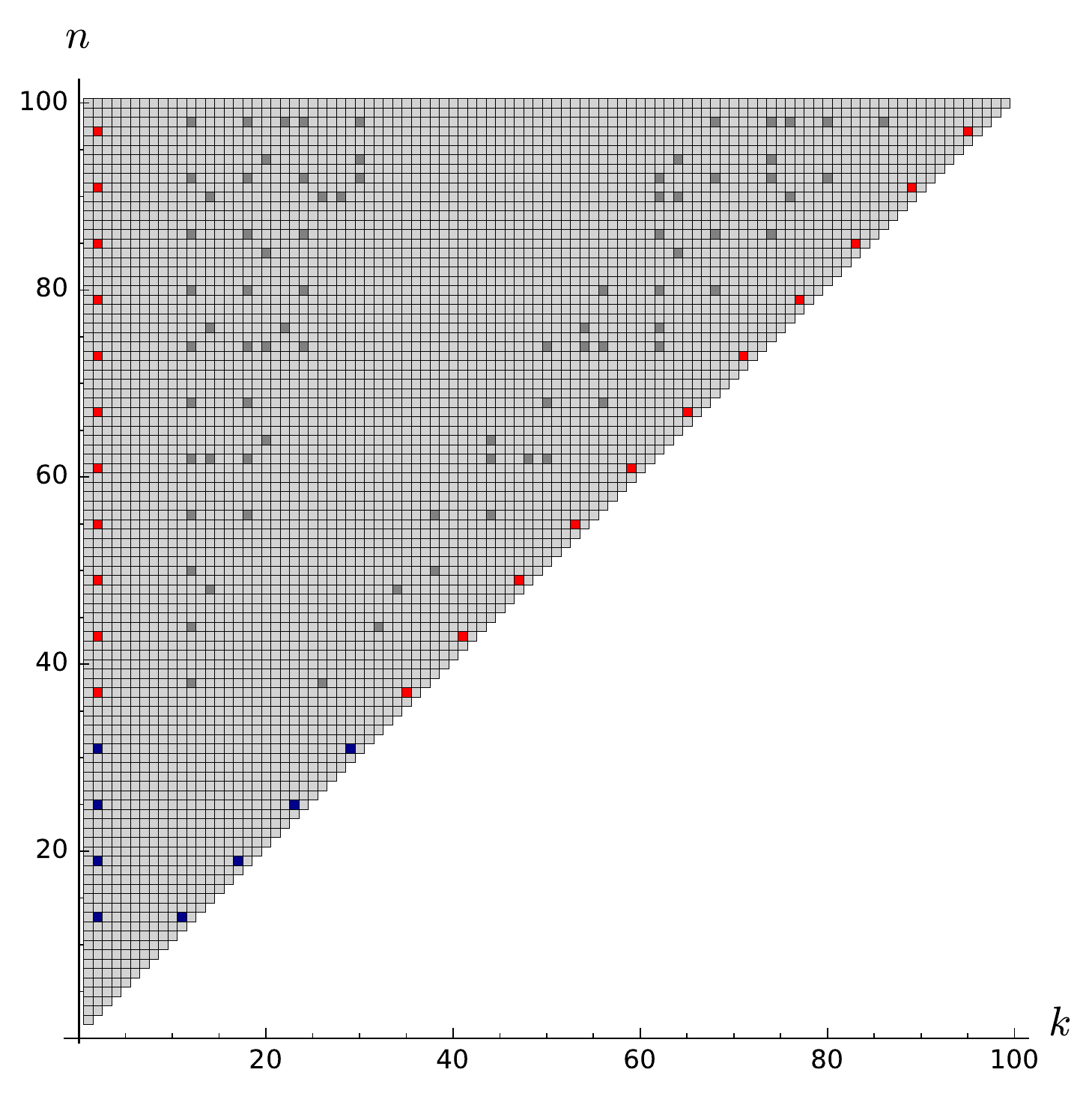}}
\vfill
\caption{Newman and Littlewood--admissibility of pairs $(k, n) \in \N^2$ for $n \leq 100$. Grey grid cells represent proved admissible pairs; dark grey cells -- conjectured admissible, dark blue cells -- proved inadmissible, red cells -- conjectured inadmissible cases.}\label{gridNewman}\label{gridNL}
\end{figure}

\section{Connection to complex Pisot numbers}\label{sec_complex_Pisot}

Recall that a real algebraic integer $\alpha > 1$ is called a \emph{Pisot number} \cite{PisotBook, DP2}, if all its conjugates $\al' \ne \al$ over $\Q$ satisfy $\abs{\al'}<1$. For every Pisot number $1 < \alpha < \frac{1+\sqrt{5}}{2}$ we have that $-\alpha$ is a root of some Newman polynomial that might be reducible in $\Z[z]$, see \cite{HM}. 

In the same way, an algebraic integer $\be \in \C \setminus \R$, $\abs{\be} > 1$ is called \emph{a complex Pisot number}, if its conjugates $\be' \not\in \{\be, \overline{\be}\}$ satisfy $\abs{\be'}<1$. They are less known than their real counterparts; it seems that they were first studied by Kelly \cite{Kelly} and later by Samet \cite{Sam}. Chamfy \cite{C1, C2} characterized smallest complex Pisot numbers, thereby extending the previous work of Pisot and Dufresnoy \cite{DP2}.  More recently, the list of small known complex Pisot numbers was expanded significantly by Garth \cite{G1, G2}. Additionally, Blumenstein, Lamarche and Saunders  computed some small complex Pisot numbers that appear among the roots of  $\{0, 1\}$ polynomials, \cite{BLS}; some results of theirs were partially published in master's thesis \cite{Saund}.

As discussed in Section \ref{sec_main}, for polynomials in $\NN$ and $\LL$, all inadmissible pairs $(k, n) \in \N^2$, apart from single exception $(3, 6)$ in the Newman case, come from $k \in \{1, 2, n-2, n-1\}$. It is easy to see that the case $k=n-1$ and $k=1$ corresponds to the irreducible (in $\Z[z]$) polynomials that are minimal polynomials of Pisot numbers and their inverses; their respective admissibility or inadmissibility was proved by Mukunda \cite{Mu1, Mu2, Mu3}.

Thus,  special attention was devoted to identify admissible and inadmissible cases  for $k=2$ and $k=n-2$. By using a pattern search that is outlined in Section \ref{sec_heuristic}, the candidate patterns to the infinite sequences of $f \in \NN_n$ and $f \in \LL_n$ with $N(f)=2$ and $U(f)=0$ were found. Polynomials $f(z)$ that belong to one of these sequences will be referred as \emph{regular} Newman and Littlewood polynomials; all other polynomials with $N(f)=2$, $U(f)=0$ that do not belong to any of these sequences will be called \emph{sporadic}. The values $N(f)=2$ and $U(f)=0$ for these regular polynomials were rigorously proved by carrying out the procedure described in Section \ref{sec_autoproof}.

An exhaustive computation of numbers $N(f)$ and $U(f)$ for $f \in \NN_n$, $n \leq 35$  and for $f \in \LL_n$ for $n \leq 31$ was performed to double check if heuristic searches did not miss any regular sequences, and also to find sporadic examples (as many as possible).  We expect that at this point we already have found all regular sequences, and, with somewhat lesser degree of confidence, all sporadic complex--Pisot polynomials in $\NN$ and $\LL$, although this, of course, remains to be proved.

\subsection{Complex Pisot polynomials in $\NN$}

\subsubsection{Regular sequences.}  Exactly $10$ infinite sequences of polynomials $f(z) \in \NN$, satisfying $N(f)=2$, $U(f)=0$,  all presented in Table \ref{patternsNewman2}, were found in $\NN$.  All of them are minimal polynomials of inverse complex Pisot numbers. Out of these $10$, only the sequence no.~2 seems to have been recorded previously  by Garth  in Table 1 of \cite{G2} (naturally, one has to take $h^*(z)$ in place of our $h(z)$). In $8$ cases the limit polynomials $f(z)$ from our Table \ref{patternsNewman2} were essentially known to Garth: they coincide with the limit polynomial $f^*(z)$ or $f^*(-z)$ in Table 1 of \cite{G2}. The remaining $2$ cases yield new limit points that previously were not registered on Garth's list.  One set of limit points arise (after the transformation $z \to z^{-1}$) from the roots of the limit polynomial $f(z)=1-z+2z^2-z^3$ of the sequence no.~4 from Table \ref{patternsNewman2}. These points are $\{\be^{-1/2}, -\be^{-1/2}\}$, where $\be$ stands for a single negative real zero of a limit polynomial from the entry no.~3  in Garth's Table 1. Another, completely new set of limit points, not associated to any polynomial from Garth's list \cite{G2}, arises from $f(z)=1+z+z^4-z^5$ in the entry no.~6 of our Table \ref{patternsNewman2}.

One should note that zeros of Newman polynomials from entries no.~8 and no.~9 from Table \ref{patternsNewman2} inside $\DD$ converge to the set of limit points $\{\ga, \overline{\ga}\}$, and $\{-\ga, -\overline{\ga}\}$, respectively; here $\ga$ is one of the zeros of limit polynomial $f(z)= 1 - z + z^2 + z^3 -z^4$ inside $\DD$. This is a non--trivial example of a non-zero complex number, such that both the number and its negative are limit points of roots of Newman polynomials.

\subsubsection{Sporadic instances.} Most of the sporadic Newman polynomials occur below $n \leq 22$, except for one remarkable finite family that originates from the pattern $11(101)^l(001)^m$. Eight choices of exponents $l, m$ in this pattern produce $f(z)$ with $N(f)=2$, $U(f)=0$:
\[
(l, m)=(0, 1), (1, 0), (1, 1), (2, 1), (3, 2), (4, 3), (5, 4), (6, 5).
\]
The longest of these exceptional patterns with $l=6$, $m=5$ correspond to $f(z) \in \NN_{34}$. We are inclined to believe that this peculiar polynomial is last possible sporadic example in $\NN$ and no more exists for $n \geq 35$.

\subsection{Complex Pisot polynomials in $\LL$}
\subsubsection{Regular sequences.} There are $16$ possible sequences of regular Littlewood polynomials with $N(f)=2$ and $U(f)=2$, each corresponding to one the $4$ normalized families from Table \ref{patternsLittlewood2} via transformations $f(z) \to \pm f(\pm z)$. Polynomials $f(z)$ from sequence no.~1 from Table \ref{patternsLittlewood2}  each has $k=2$ real roots inside $\DD$; the set of limit points for these roots is $\{1/\sqrt{2}, -1/\sqrt{2}\}$. The remaining regular Littlewood polynomials that correspond to sequences no.~2, 3, 4 in Table \ref{patternsLittlewood2} are minimal polynomials of inverse complex Pisot numbers. None of the entries in Table \ref{patternsLittlewood2} nor their limit polynomials appear in \cite{G2}.

\subsubsection{Sporadic instances.} The number of sporadic examples in $\LL$ is much smaller than in $\NN$: the last found sporadic $f(z)$ was in  $\LL_{16}$, they seem to run out for $n \geq 17$.

\section{Statistical distribution of $N(f)$ in  $\NN$ and $\LL$}\label{sec_stats}

Throughout this section, let $a_0$, $a_1$, $\dots$, $a_n$ be  $n+1$ independent, identically distributed discrete random variables.
If each $a_j$ is a Bernoulli random variable, taking values $0$ or $1$ with equal probability $p=1/2$, then the polynomial
\[
f(z) = 1 + a_1z + \dots a_{n-1}z^{n-1} + z^n
\]
will be called \emph{a random Newman polynomial} of degree $n$. Likewise, if $b_j$ are Rademacher random variables, that is, random variables taking values $-1$, $1$ with equal probability $p=1/2$, then the polynomial
\[
g(z) = b_0 + b_1z + b_2z +\dots + b_n z^n
\] will be called \emph{a random Littlewood polynomial}. In the Newman case $f(z)$ depends on $n-1$ random coefficients, whereas in the Littlewood case $g(z)$ depends on $n+1$) random coefficients. Note that $f(z)$ and $g(z)$ are uniformly distributed in $\NN_n$ and $\LL_n$, respectively, according to the definitions in equation \eqref{defNL}.

For such random polynomials $f(z)$ and $g(z)$, the unit disk root-counting functions $N(f)$ and $N(g)$ themselves become a random variables that take values $k \in \{0, 1, \dots, n-1\}$. Thus, one can speak of the probabilities
\[
p_k = P\left(N(f) = k\right) = \#\{f \in \NN_n: N(f)=k\} \big/ \#{\NN_n},
\]
or
\[
p_k = P\left(N(g) = k\right) = \#\{g \in \LL_n: N(g)=k\} \big/ \#{\LL_n}.
\]

Since we are mostly interested in polynomials with zeros inside $\DD$ and no zero on $\partial \DD$, it makes sense to condition the random polynomials so that $U(f)=U(g)=0$. Let us define:
\[
\widetilde{\NN}_n = \{f \in \NN_n: U(f)= 0\}, \qquad \widetilde{\LL}_n = \{g \in \LL_n: U(g)= 0\},
\] and the corresponding conditional probabilities:
\[
\widetilde{p}_k = P\left(N(f) = k \mid U(f)=0 \right) = \#\{f \in \widetilde{\NN}_n: N(f)=k\} \big/  \#\widetilde{\NN}_n,
\]
and
\[
\widetilde{p}_k = P\left(N(g) = k \mid U(g)=0 \right) = \#\{g \in \widetilde{\LL}_n: N(g)=k\} \big/  \#\widetilde{\LL}_n.
\] The zero-counting random variable $N(f)$ under the condition $U(f)=0$ will be denoted by $\widetilde{N}(f)$ (resp. $\widetilde{N}(g)$). Since $N(f) + N(f^*) \leq n$ with `$=$' only for $f(z)$ with $U(f)=0$, it is easily seen that, for $n \geq 3$, $E\widetilde{N}(f) = n/2$, $EN(f) < n/2$. The same holds true if one replaces $f(z)$ by $g(z)$.

The least--squares fit of the line $k=\gamma n$ to our data depicted in Figure \ref{meanvarN} yields
$EN(f) \approx 0.49050n$, for $f \in \NN_n, 15 \leq n \leq 35$  and $EN(g) \approx 0.49389n$, for $g \in \LL_n$, $15 \leq n \leq 31$. From Figure \ref{meanvarN}, least--squares fit of the line $k=\alpha n+\beta$ to the values of data points representing the variances yields crude estimates
\[
\variance{N(f)} \approx 0.15499n + 0.97761, \quad \variance{\widetilde{N}(f)} \approx 0.18087n + 0.00424
\] for random Newman polynomials, and
\[
\variance{N(g)} \approx 0.12812n + 0.45214, \quad \variance{\widetilde{N}(g)} \approx 0.14154n + 0.00966, 
\] for random Littlewood polynomials. In above estimates, we removed the data for degrees $n \leq 14$ because of larger initial oscillations in variance.  Even with this precaution, these estimates match the true values of $\al$, $\be$ and $\ga$ in no more than first or second decimal after `.' -- this is due to the relatively small values of $n$.

Comparing the distributions of $N(f)$ and $\widetilde{N}(f)$  in top and bottom histograms in Figure \ref{histNewman}, one sees that $N(f)$ is slightly skewed to the left, while $\widetilde{N}(f)$ is perfectly symmetric. Similar skew occurs also in the Littlewood case, but it is hardly noticeable for certain degrees, in particular, when $n+1$ is a prime number or a product of at most two primes (cf. $n=30$ and $n=31$ degree cases in \ref{histLittlewood}). The later phenomenon seems to be related to the restrictions of factoring a $\{-1, 1\}$ polynomial $\pmod{2}$, which leaves much less possibilities for a Littlewood polynomials to have a non-trivial divisor in $\Z[z]$ with unimodular zeros.

The distributions of $N(f)$ and $\widetilde{N}(f)$ in $\NN_n$ differ considerably for even and odd $n$. The shortage of Newman polynomials of even degree with an odd number $N(f)$ of zeros inside $\DD$ is evident in histograms \ref{histNewmanA}, \ref{histNewmanC}, \ref{histNewmanD} and \ref{histNewmanF}. For $f \in \NN_{2n}$, $N(f)$ is odd precisely when $f(-1) \leq 0$, which means that there should be at least as many, or more $z^{2j-1}$ terms than $z^{2j}$ terms in $f(z)$; this seems to severely reduce the number of choices and creates artifacts in histograms for even $n$. There appears to be no such artifacts in the odd degree Newman case, cf. \ref{histNewmanB}, \ref{histNewmanE}. In $\LL_n$, there are no visible differences between even and odd $n$ in Histogram \ref{histLittlewood}. 
\begin{figure}
\centering
\subcaptionbox{$N(f)$, $f \in \NN_n$, for $2 \leq n \leq 35$}{\includegraphics[width=0.495\linewidth]{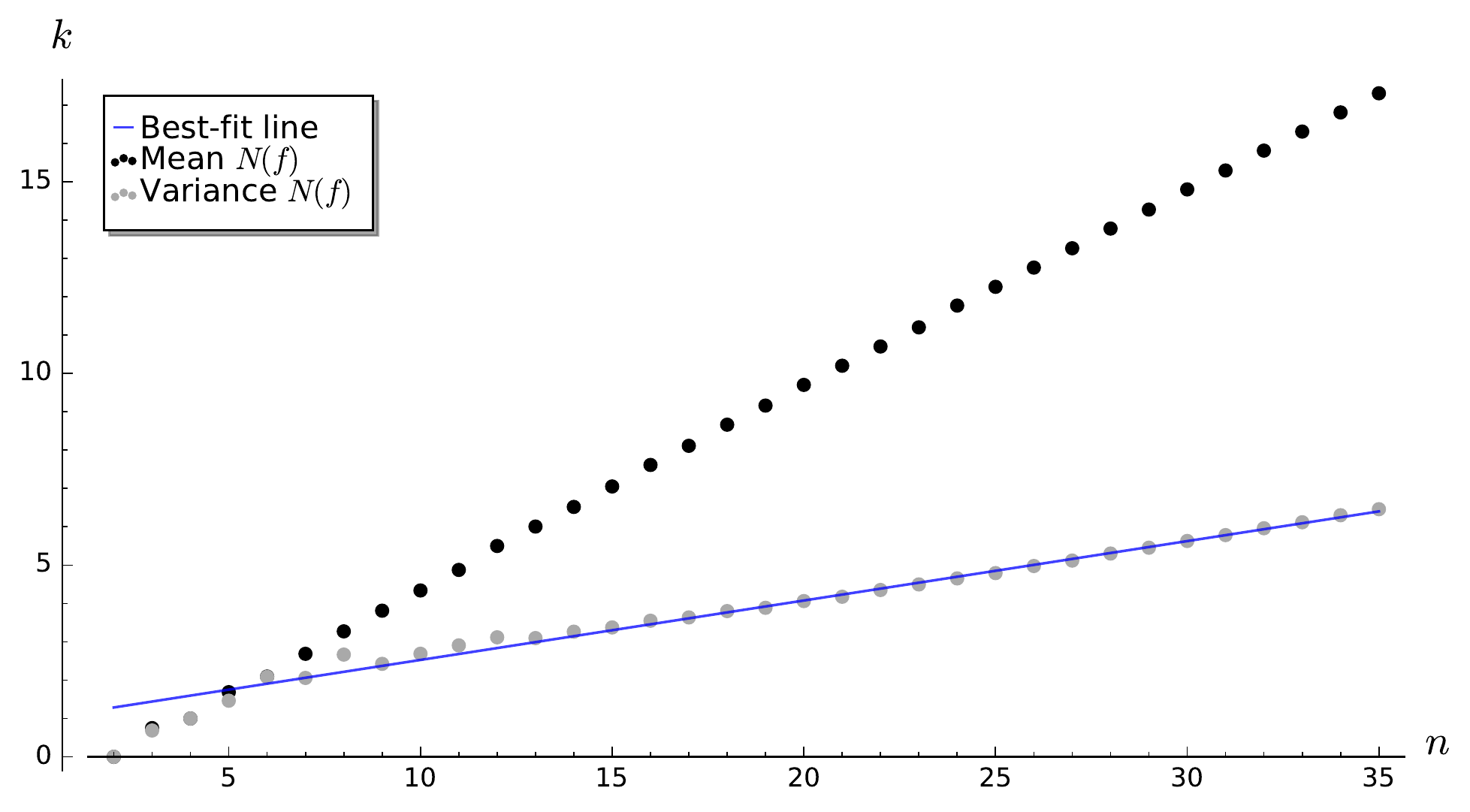}}
\hfill
\subcaptionbox{$\widetilde{N}(f)$, $f \in \widetilde{\NN}_n$, for $3 \leq n \leq 35$}{\includegraphics[width=0.495\linewidth]{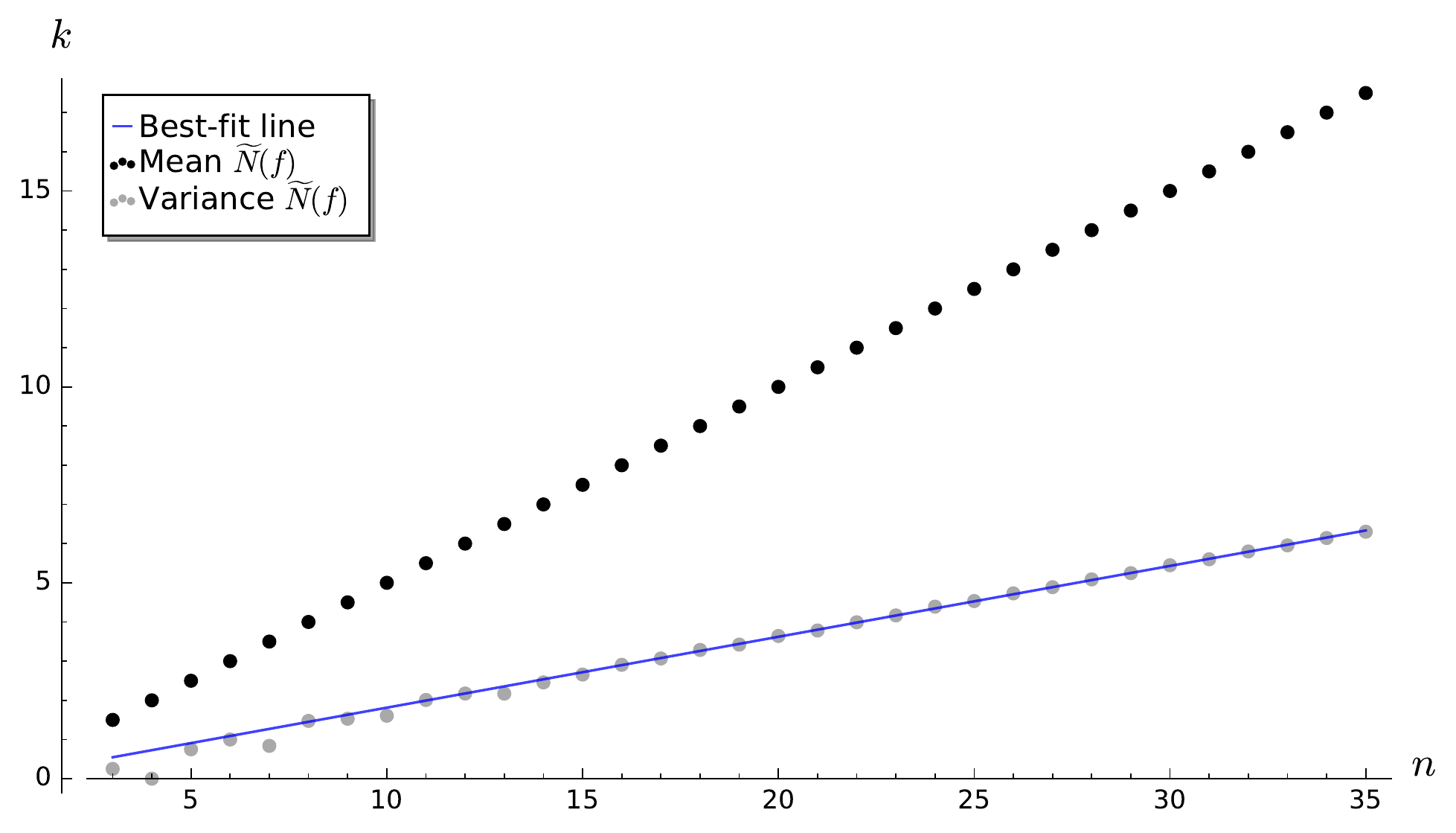}}
\vfill
\subcaptionbox{$N(g)$, $g \in \LL_n$, for $2 \leq n \leq 31$}{\includegraphics[width=0.495\linewidth]{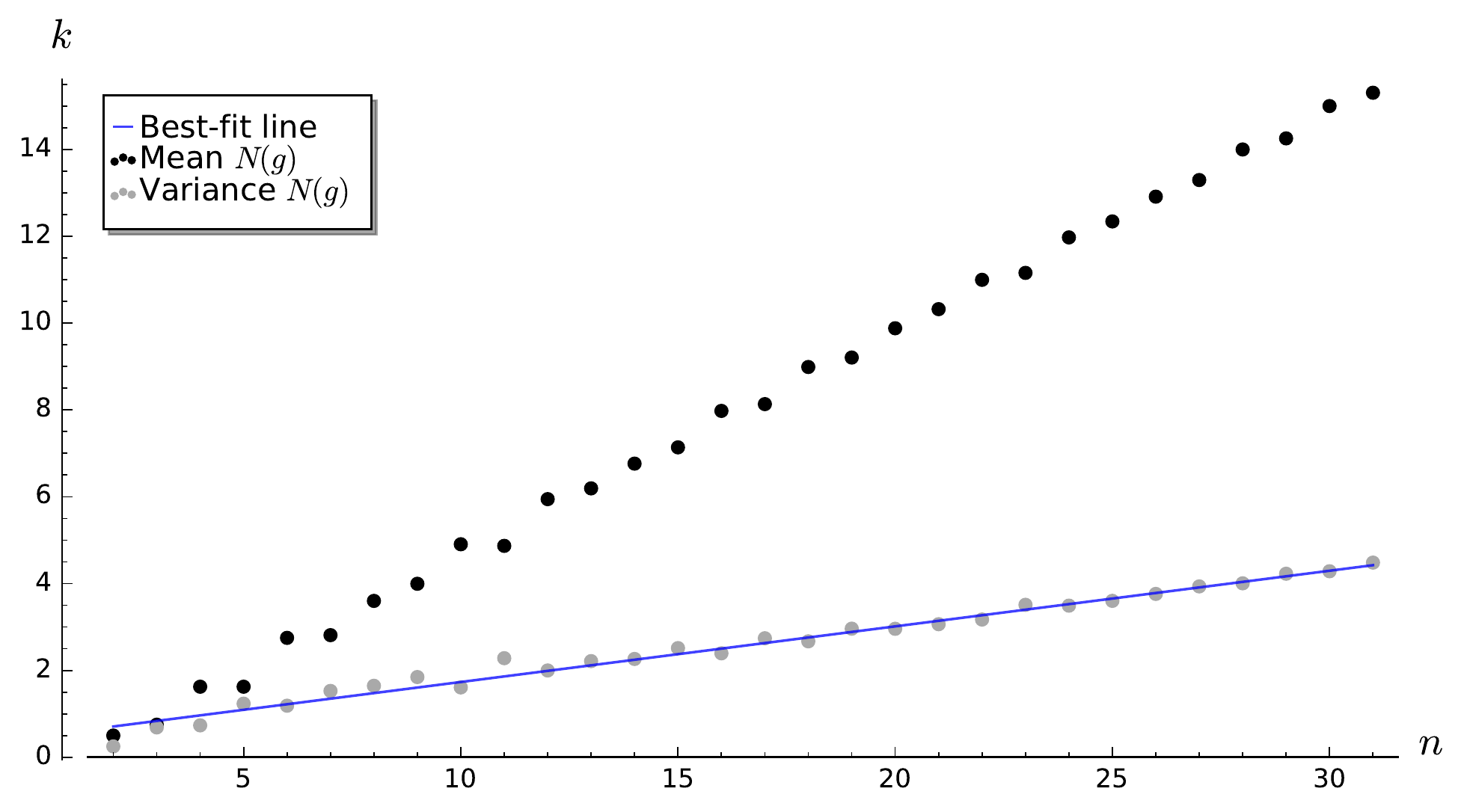}}
\hfill
\subcaptionbox{$\widetilde{N}(g)$, $g \in \widetilde{\LL}_n$, for $2 \leq n \leq 31$}{\includegraphics[width=0.495\linewidth]{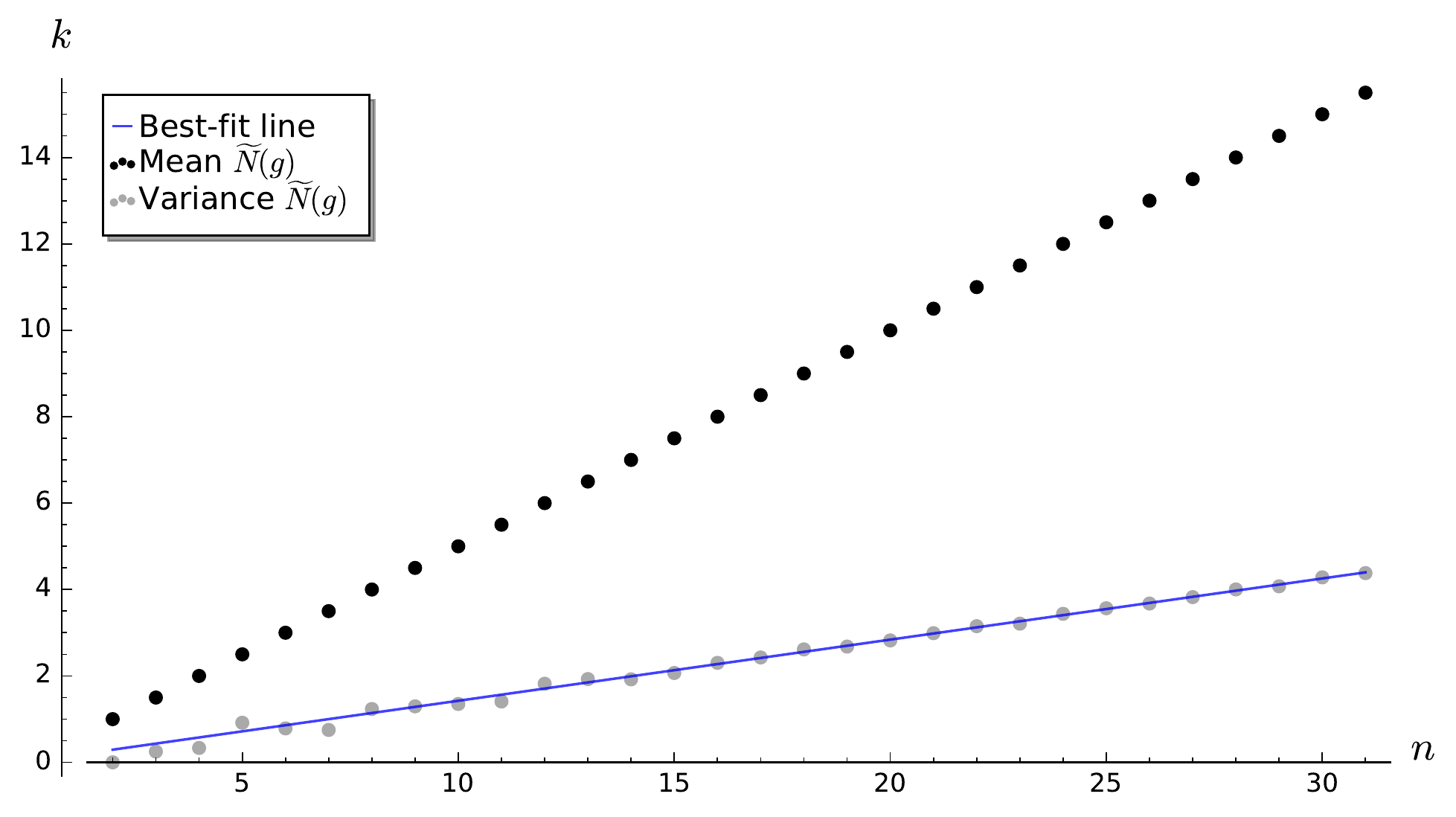}}
\caption{Means (black) and variances (gray)  of $N(f)$ and $N(g)$ for $f \in \NN_n$, $g \in \LL_n$, with/without unimodular zeros, together with lines of best-fit (blue) for variances.}\label{meanvarN}
\end{figure}
\begin{figure}
\centering
\subcaptionbox{$N(f)$, $f \in \NN_{30}$ \label{histNewmanA}}{\includegraphics[width=0.32\linewidth]{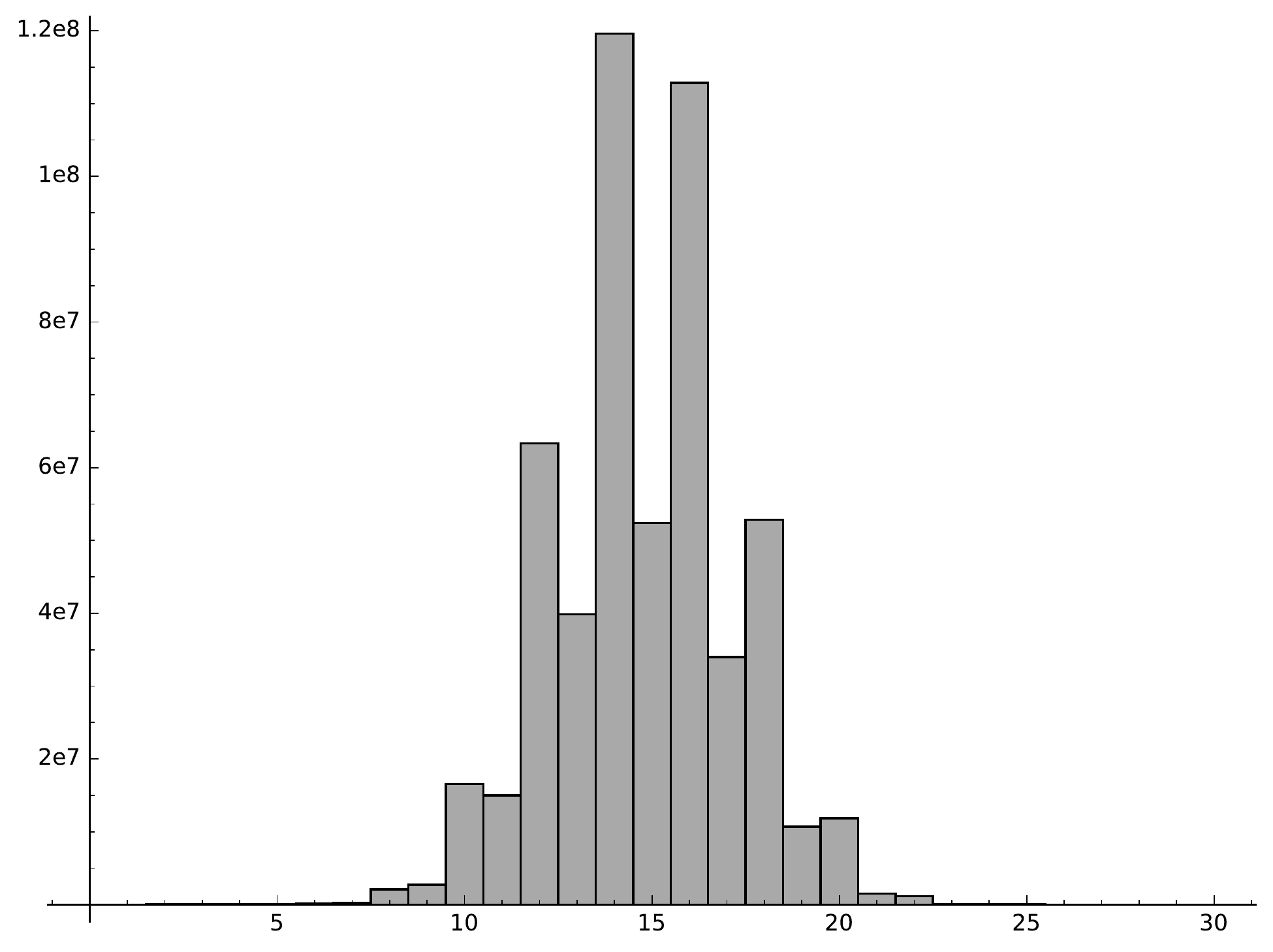}}
\subcaptionbox{$N(f)$, $f \in \NN_{31}$ \label{histNewmanB}}{\includegraphics[width=0.32\linewidth]{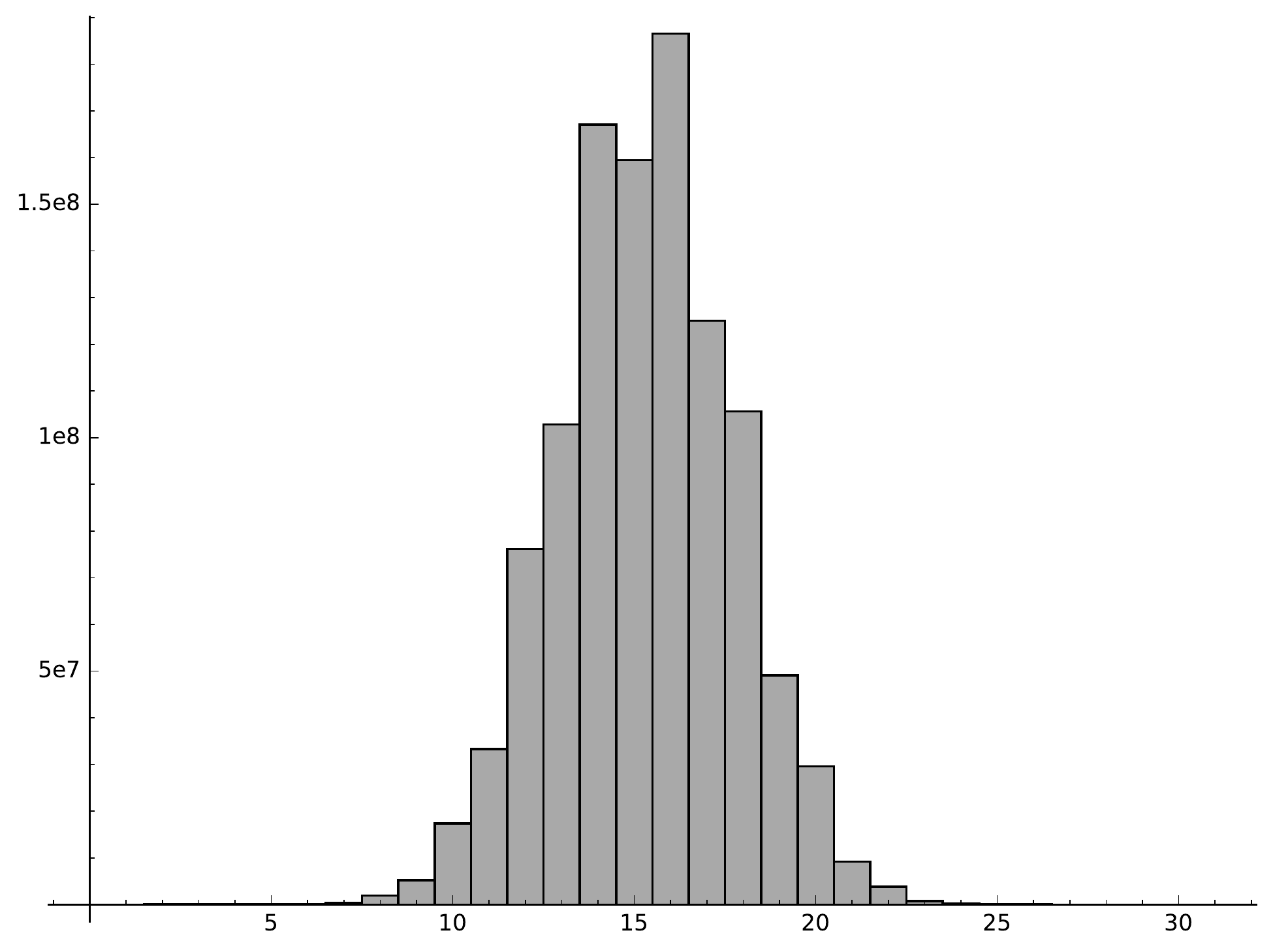}}
\subcaptionbox{$N(f)$, $f \in \NN_{32}$ \label{histNewmanC}}{\includegraphics[width=0.32\linewidth]{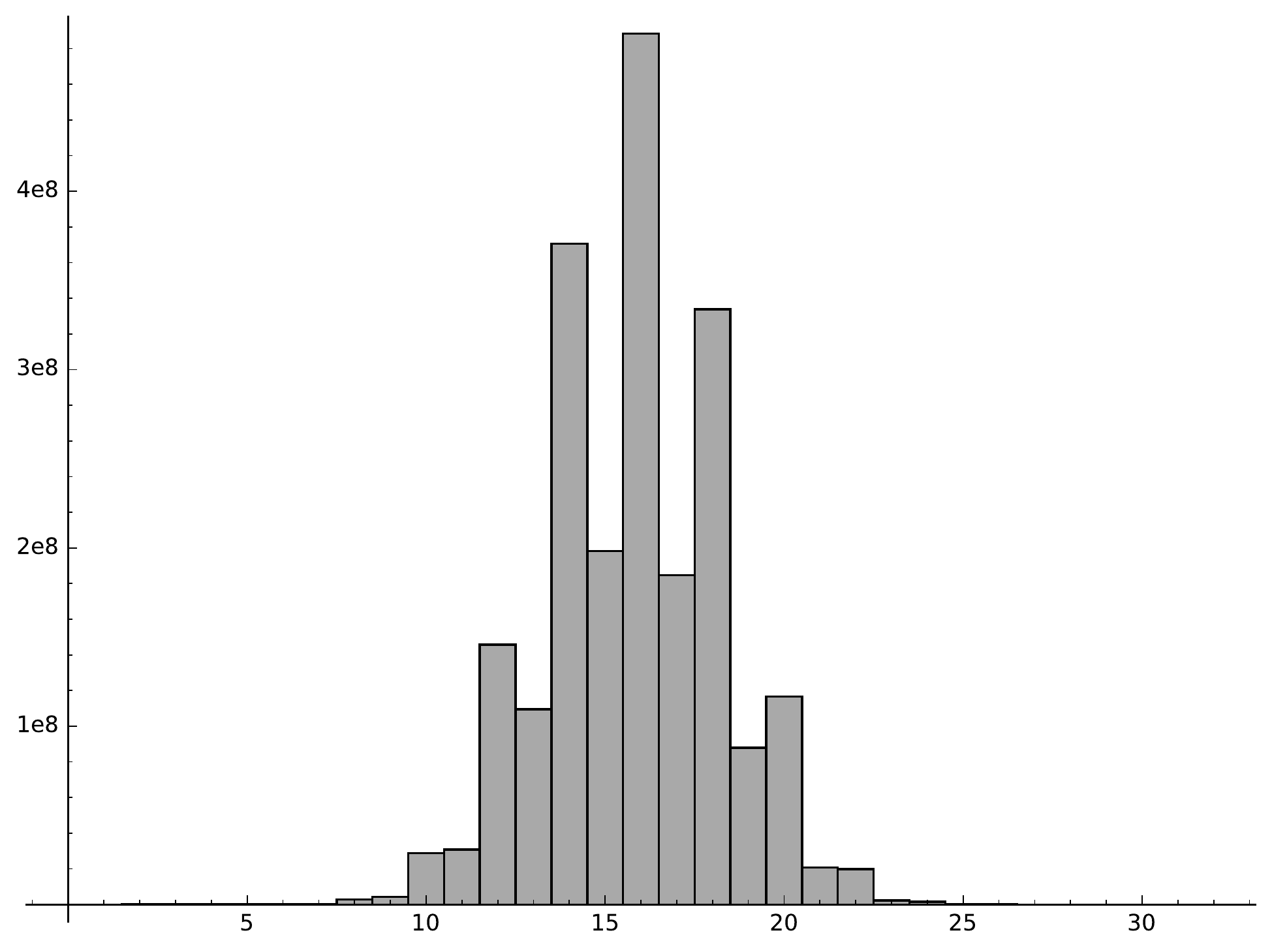}}
\vfill
\subcaptionbox{$\widetilde{N}(f)$, $f \in \widetilde{\NN}_{30}$ \label{histNewmanD}}{\includegraphics[width=0.32\linewidth]{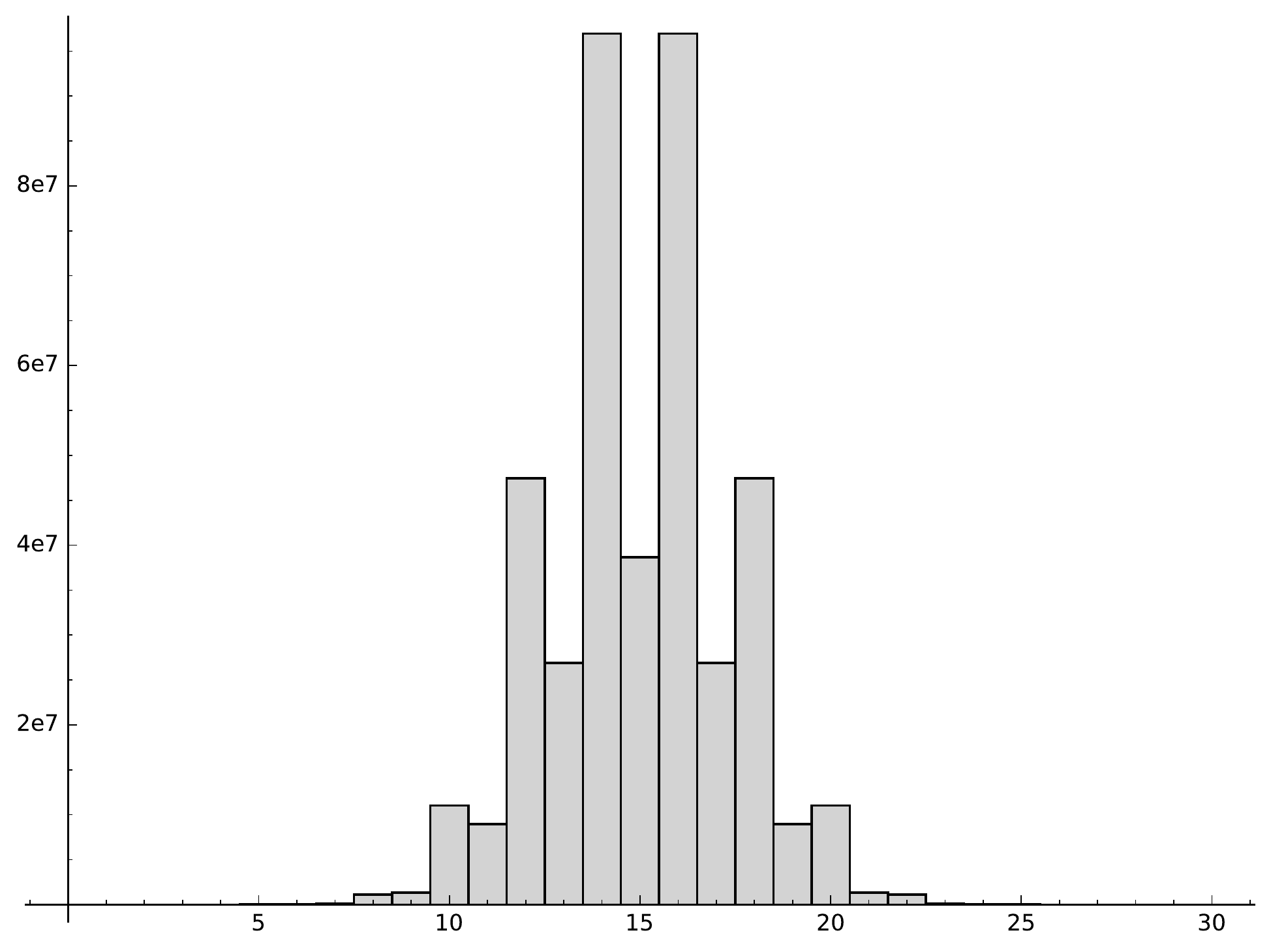}}
\subcaptionbox{$\widetilde{N}(f)$, $f \in \widetilde{\NN}_{31}$ \label{histNewmanE}}{\includegraphics[width=0.32\linewidth]{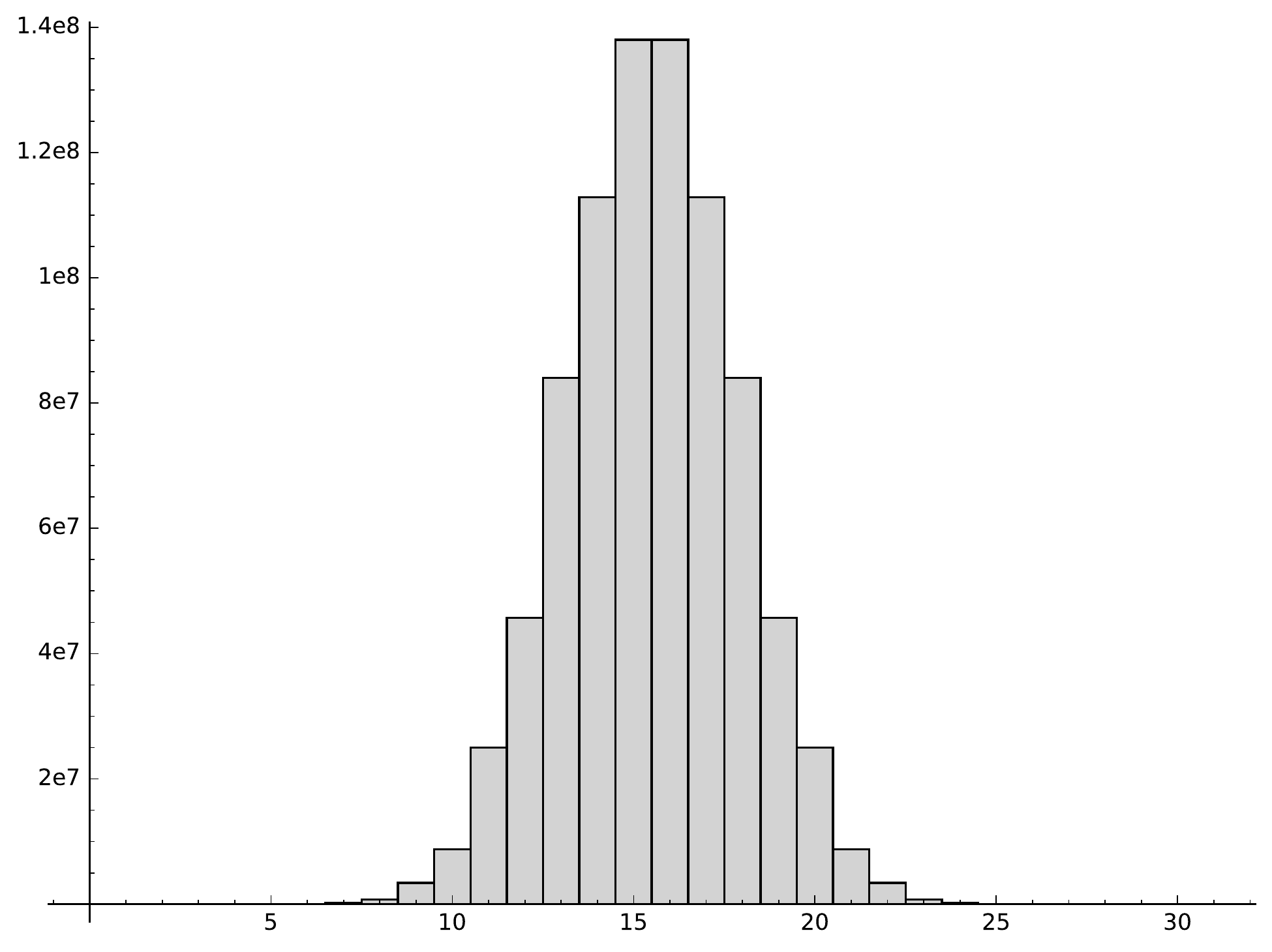}}
\subcaptionbox{$\widetilde{N}(f)$, $f \in \widetilde{\NN}_{32}$ \label{histNewmanF}}{\includegraphics[width=0.32\linewidth]{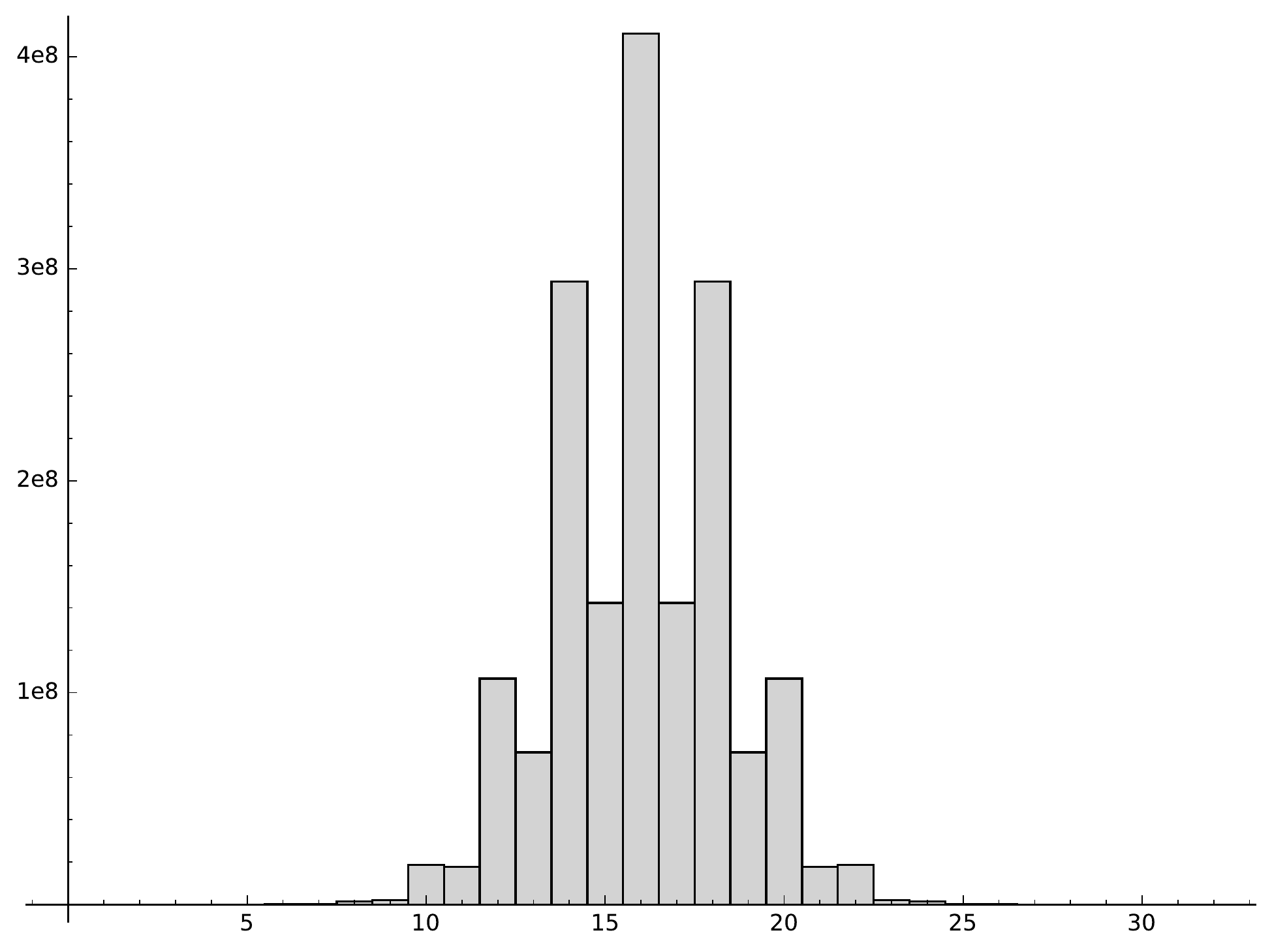}}
\caption{$N(f)$, $f \in \NN_n$ and $\widetilde{\NN}_n$, for $n=30, 31, 32$}\label{histNewman}
\end{figure}
\begin{figure}
\centering
\subcaptionbox{$N(g)$, $g \in \LL_{30}$ \label{histLittlewoodA}}{\includegraphics[width=0.24\linewidth]{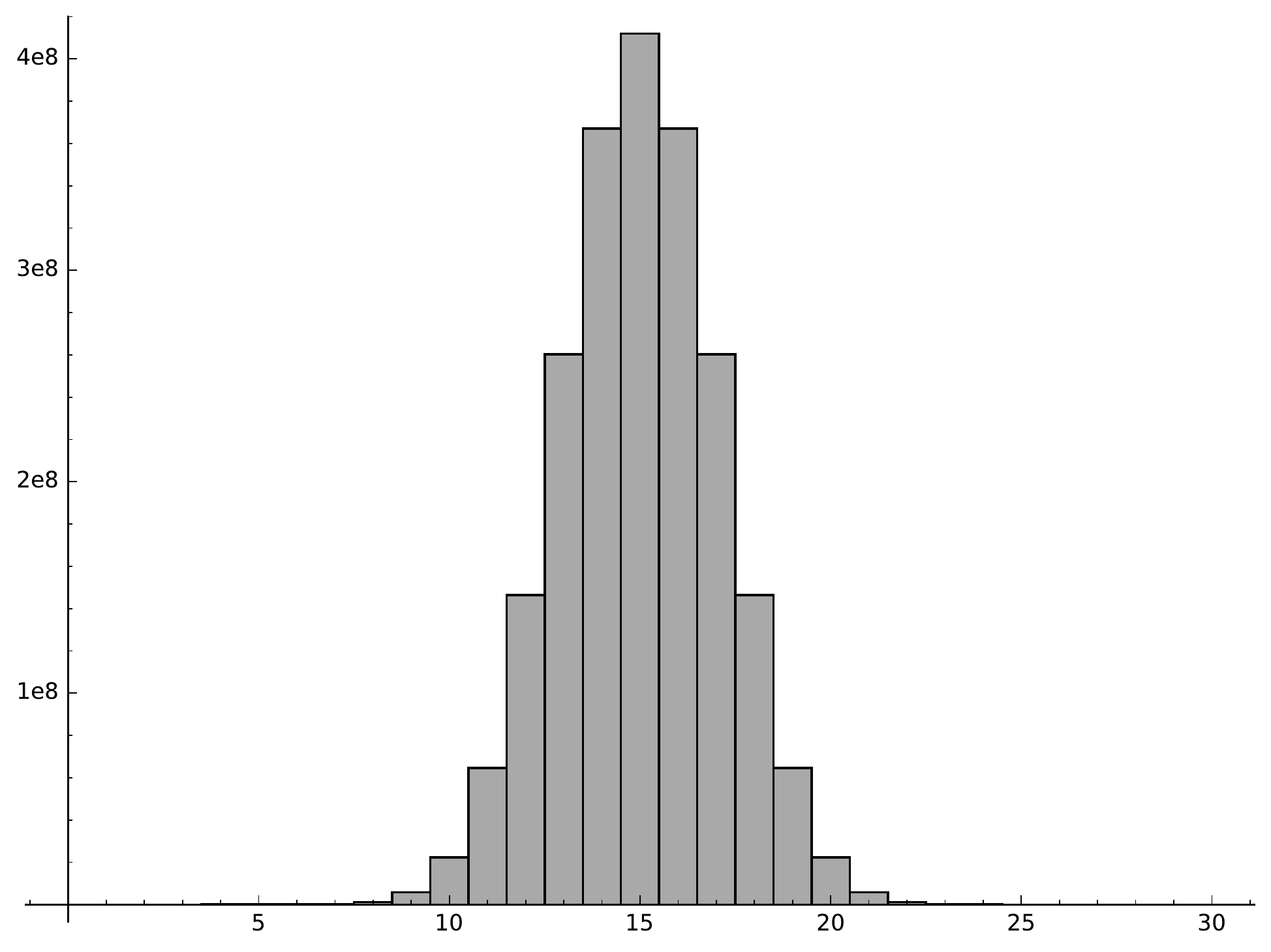}}
\subcaptionbox{$\widetilde{N}(g)$, $g \in \widetilde{\LL}_{30}$ \label{histLittlewoodB}}{\includegraphics[width=0.24\linewidth]{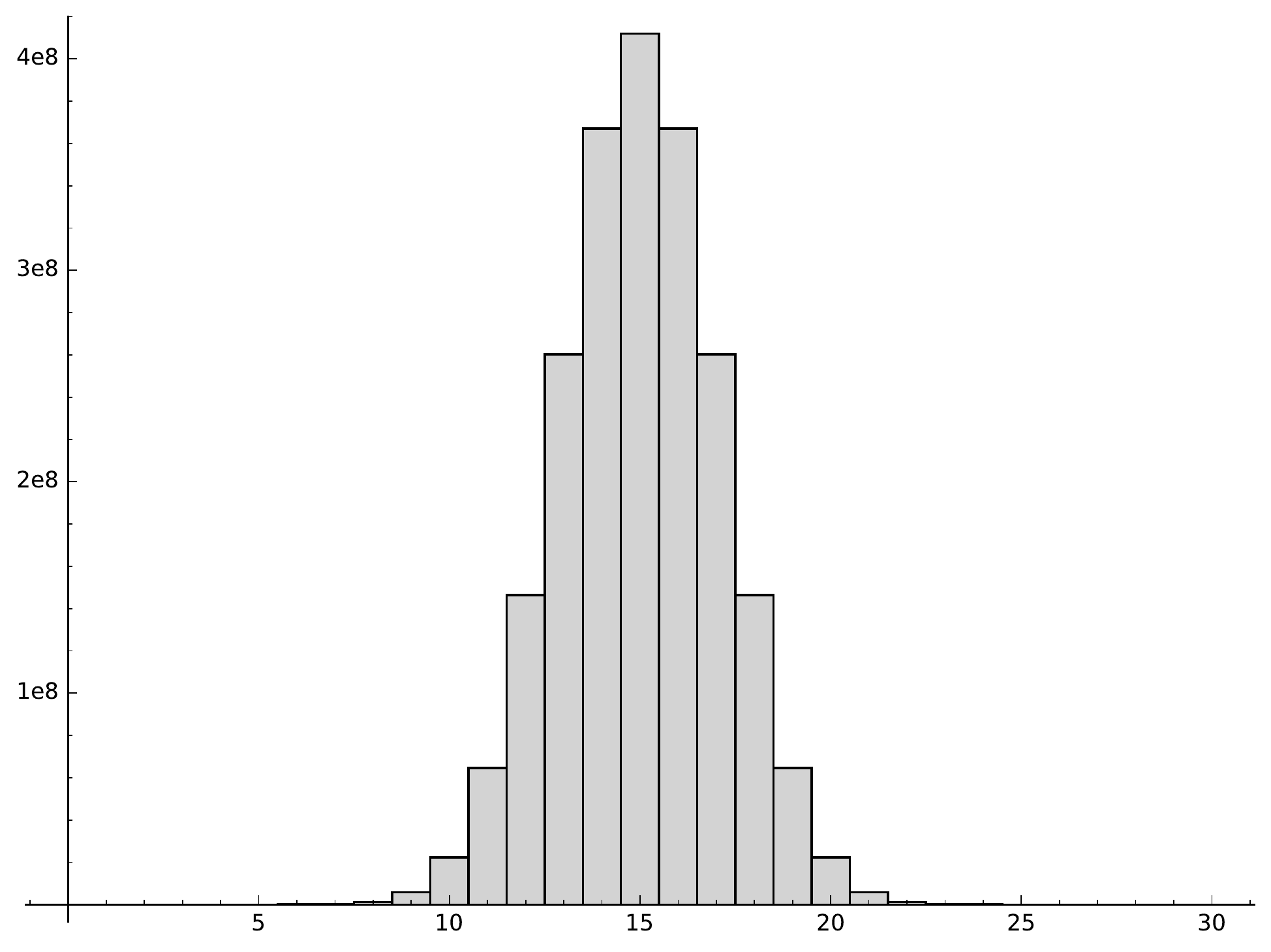}}
\subcaptionbox{$N(g)$, $g \in \LL_{31}$ \label{histLittlewoodC}}{\includegraphics[width=0.24\linewidth]{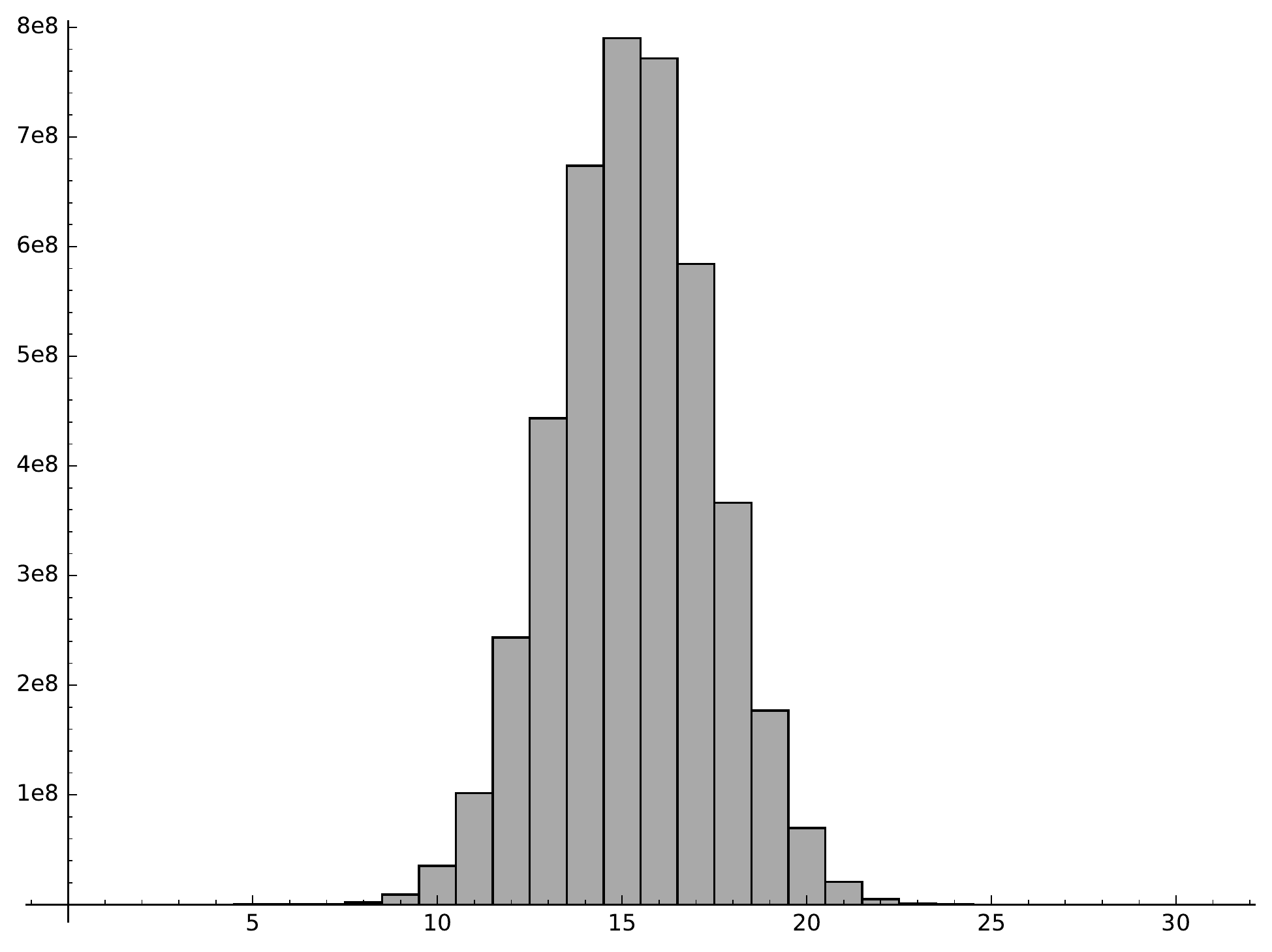}}
\subcaptionbox{$\widetilde{N}(g), g \in \widetilde{\LL}_{31}$ \label{histLittlewoodD}}{\includegraphics[width=0.24\linewidth]{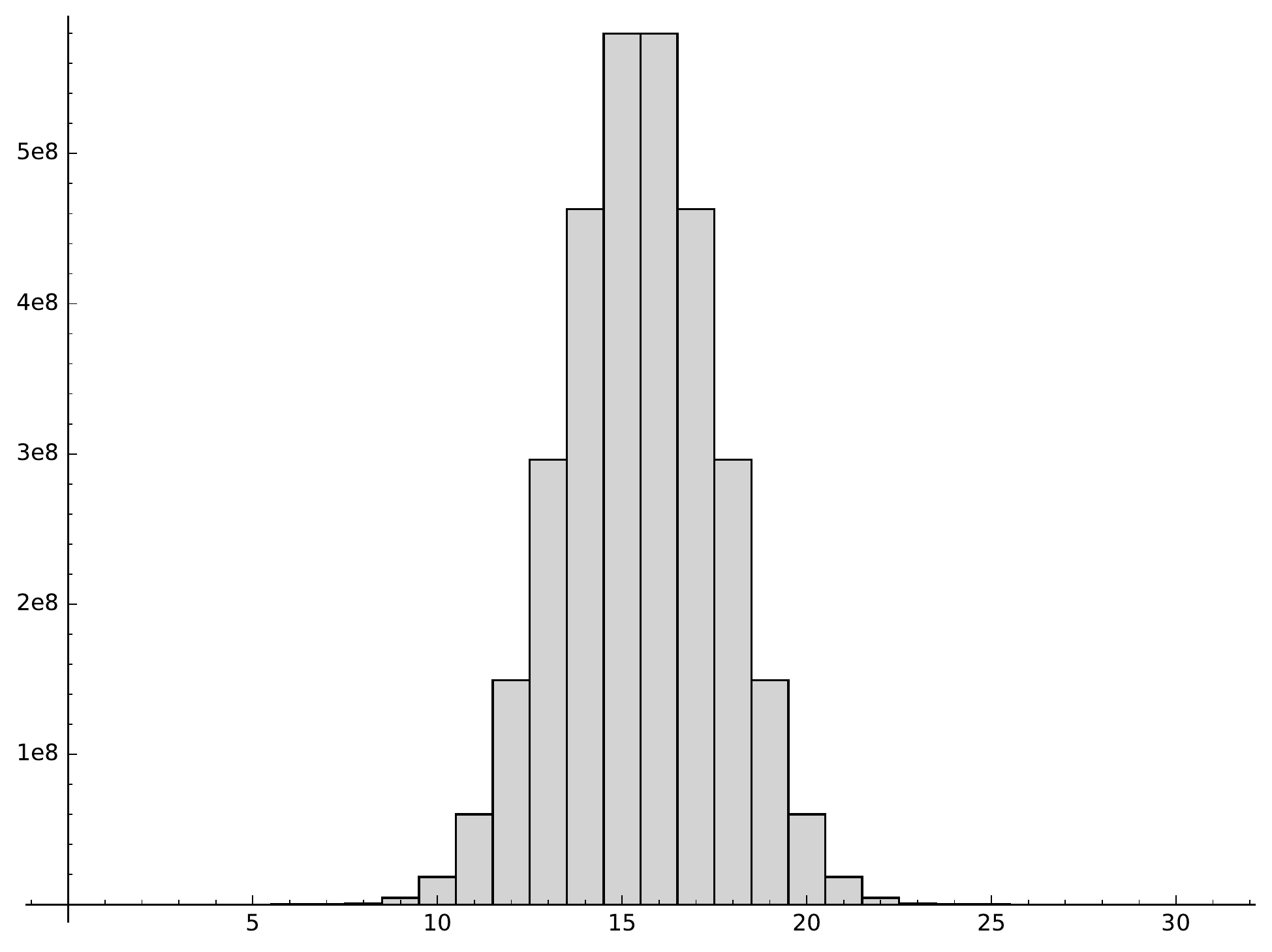}}
\caption{$N(g)$, $g \in \LL_n$ and $\widetilde{\LL}_n$, for $n=30, 31$}\label{histLittlewood}
\end{figure}
\begin{figure}
\centering
\subcaptionbox{$N(f)$, $f \in \NN_{35}$ \label{histZa}}{\includegraphics[width=0.45\linewidth]{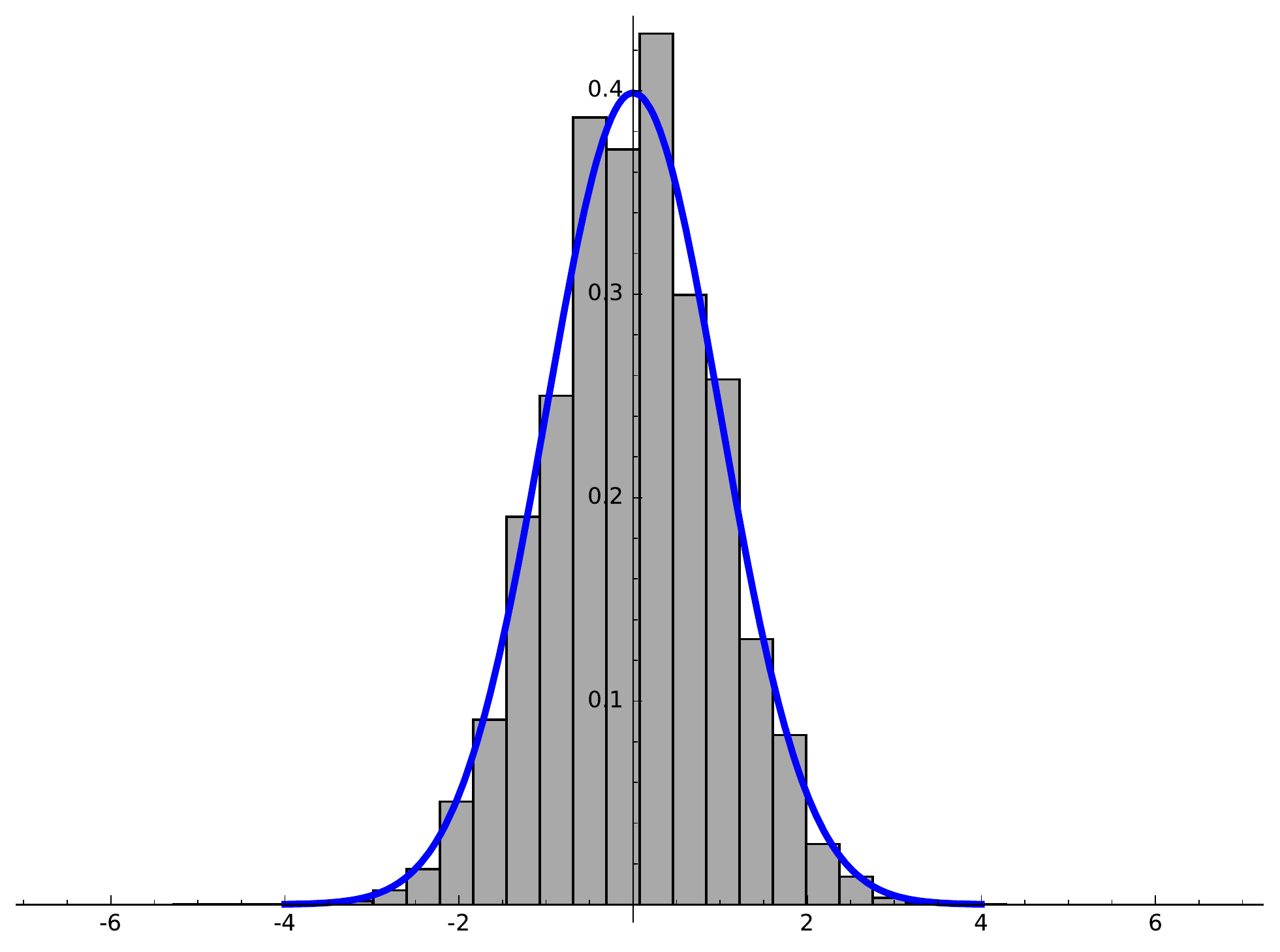}}
\subcaptionbox{$\widetilde{N}(f)$, $f \in \widetilde{\NN}_{35}$ \label{histZb}}{\includegraphics[width=0.45\linewidth]{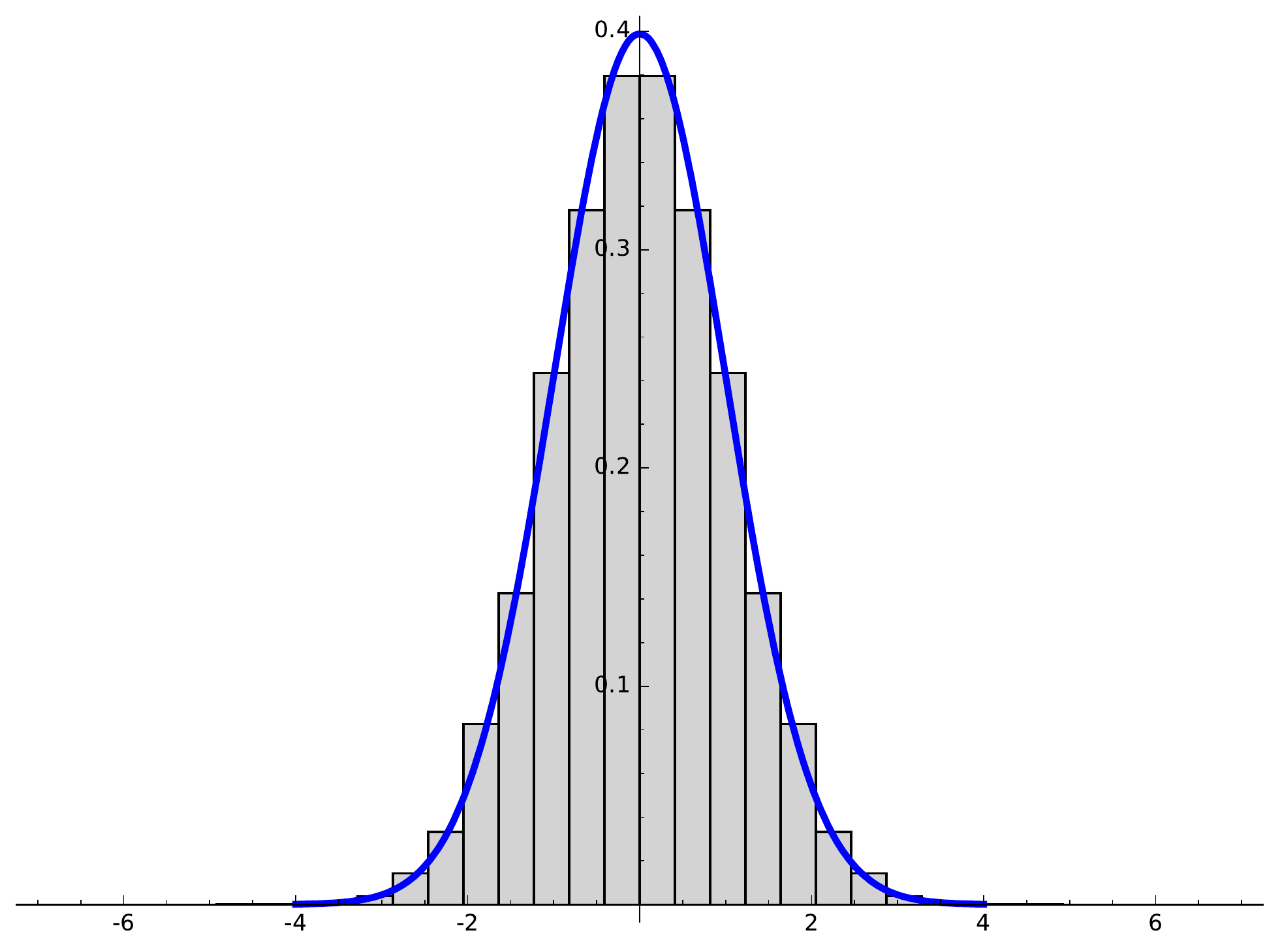}}
\vfill
\caption{Distribution of $z$-scores for $N(f)$, $f \in \NN_{35}$ and $f \in \widetilde{\NN}_{35}$, compared to density function of $Z(0, 1)$.}\label{histZapprox}
\end{figure}

We conclude this subsection with a natural conjecture illustrated by Figure \ref{histZapprox} in the odd degree Newman case.
\begin{conjecture}
For $f \in \NN_{2n+1}$ and for $g \in \LL_n$, 
\[
\frac{N(f)-EN(f)}{\sqrt{\variance{N(f)}}} \xrightarrow{\; \text{distr.} \; } Z(0, 1), \qquad \frac{\widetilde{N}(f)-E\widetilde{N}(f)}{\sqrt{\variance{\widetilde{N}(f)}}} \xrightarrow{\; \text{distr.} \; } Z(0, 1),
\]
\[
\frac{N(g)-EN(g)}{\sqrt{\variance{N(g)}}} \xrightarrow{\; \text{distr.} \; } Z(0, 1), \qquad \frac{\widetilde{N}(g)-E\widetilde{N}(g)}{\sqrt{\variance{\widetilde{N}(g)}}} \xrightarrow{\; \text{distr.} \; } Z(0, 1),
\]
as $n \to \infty$.
\end{conjecture}
\section{Other results}

We take a look at operations in $\NN$ and $\LL$ that are useful in building admissible pairs from smaller ones. On their own, these two methods are not powerful enough to produce the full sets of Newman or  Littlewood admissible pairs, however they are of interest for future applications.

\subsection{Products of spaced Newman polynomials}

The product construction was inspired by the method of Carroll, Eustice and Fiegel \cite{CEF} (also exploited in \cite{B6, Mer3}; see also Ch.15, pp. 126--127 of \cite{Bbook}). 

\begin{theorem}\label{MultThm} Suppose that Newman--admissible pairs $(a, c)$ and $(b, d) \in \N^2$, such that $ad - bc = 1$, are realized by $f(z), g(z) \in \NN$ with
\[
N(f)=a, \qquad N(g)=b, \qquad \deg{f}=c, \qquad  \deg{g}=d.
\]
Then, for every rational fraction $k/n$ that is contained in one of the intervals
\[
 \left[\frac{b}{d}, \frac{a+bc}{c+cd}\right], \left[ \frac{b+ad}{d+cd}, \frac{a}{c}\right] ,
\left[1-\frac{a}{c}, 1-\frac{b+ad}{d+cd}\right], \left[ 1-\frac{a+bc}{c+cd}, 1-\frac{b}{d}\right], 
 \]
the pair $(k,n) \in \N^2$ is Newman--admissible: it can be realized by $h(z) := f(z^l)g(z^m) \in \NN$ for some $l, m \in \N$, or its reciprocal $h^*(z)$.
\end{theorem}

\begin{proof}[Proof of Theorem \ref{MultThm}] For arbitrary $l$, $m \in \N$, let $h(z) := f(z^l)g(z^m)$. If $l > dm=\deg{g(z^m)}$ or $m > cl=\deg{f(z^l)}$, then $h(z) \in \NN$ as the terms in $h(z)$ do not collide. Also, $U(h)=0$ because $U(f)=U(g)=0$. Let $k=N(h)$, $n=\deg{h}$. Clearly, one has
\[
 \begin{cases}
 k = al + bm\\
 n =cl + dm
 \end{cases}, \text{ or, in matrix form, }
 \begin{pmatrix}
 k\\
 n
 \end{pmatrix} = \begin{pmatrix}
 a & b\\
 c & d
 \end{pmatrix} \cdot \begin{pmatrix}
 l\\
m
 \end{pmatrix}. 
 \]
 Since $ad - bc = 1$, the linear transformation is invertible over $\Z$:
 \begin{equation}\label{eq_lm}
 \begin{pmatrix}
 l\\
 m
 \end{pmatrix} = \begin{pmatrix}
 d & -b\\
 -c & a
 \end{pmatrix} \cdot \begin{pmatrix}
 k\\
n
 \end{pmatrix}. 
 \end{equation}
We want to find values $(k, n)$ that yield $(l, m) \in \mathcal{T}_1 \cup \mathcal{T}_2$ in \eqref{eq_lm}, where:
 \[
 \mathcal{T}_1 := \left\{(l, m): \in \Z^2: m > 0, l > dm
 \right\},
 \mathcal{T}_2 := \left\{(l, m) \in \Z^2: l > 0, m > cl
 \right\}.
 \] After solving for $(k, n)$ that correspond to parameters $(l, m)$, one finds
 \[
 \mathrm{Image}(\mathcal{T}_1) = \left\{ (k, n) \in \N^2: \frac{b+ad}{d+cd} < \frac{k}{n} < \frac{a}{c}\right\},
 \]
 \[
 \mathrm{Image}(\mathcal{T}_2) = \left\{ (k, n) \in \N^2: \frac{b}{d} < \frac{k}{n} < \frac{a+bc}{c+cd}\right\}.
 \]
 This is how the first two intervals in Theorem \ref{MultThm} appear. The last two intervals are obtained by reflection $(k, n) \mapsto (n-k, n)$ from the first two: they correspond to pairs realized by $h^*(z)$.
 \end{proof}
 \emph{Note:} It is possible that the pair $(k, n)$ can still be realized by the product polynomial $h(z) \in \NN$ even if $k/n$ lies outside of intervals described in Theorem \ref{MultThm}: this is because there might be still no collision of terms in the product even if $l \leq dm$ or $m \leq cl$. As it depends on the coefficient patterns of $f(z)$ and $g(z)$, such cases do not admit a simple description.
 
 \begin{example} Consider the pair $f(z) = 1+z^2+z^3$ and $g(z) = 1+z^3+z^5$. The determinant
 \[
\det{\begin{pmatrix}
N(f) & N(g)\\
 \deg{f} & \deg{g}
 \end{pmatrix}} = \det{\begin{pmatrix}
 2 & 3\\
 3 & 5
 \end{pmatrix}} = 1.
 \] Theorem \ref{MultThm} tells us that the pairs $(k,n) \in \N^2$ with
 \[
 k/n \in (1/3, 7/20) \cup (7/18, 2/5) \cup (3/5, 11/18) \cup (13/20, 2/3)
 \] certainly can be realized by $f(z^l)g(z^m)$, where $l, m \in \N$, or their reciprocals.
\end{example}

\subsection{Rotated Littlewood polynomials}

Rotated polynomials were considered in the constructions of sequences with large Merit factors and polynomials of small $L^4$--norm, see papers \cite{BoChoJe, Gola, JeKaSch, HohJen}. 

For $f(z) \in \C[z]$ of degree $n$, written as in equation \eqref{genForm} with $a_j \ne 0$, $0 \leq j \leq n$, the polynomial
\begin{equation}\label{defRot}
g(z) = a_n + a_0z + a_1 z^2+ \cdots a_{n-1}z^n
\end{equation}
produced by the cyclic shift of the coefficients of $f(z)$ is \emph{a forward rotation} of $f(z)$, which we denote by $g := \Rot{f}{}$. 
The polynomial $f(z)$ then is \emph{a backward rotation} of $g(z)$, which we denote as $f = \Rot{g}{-1}$. 
As $a_j \ne 0$, after $n+1$ rotations the original polynomial is restored. 
\begin{lemma}\label{RotLem}
Let $f(z)$ and $g(z) = \Rot{f}{}$ be as in equation \eqref{genForm} and \eqref{defRot}. Then
\begin{equation}\label{eqRotDif}
zf(z) - g(z) = a_n(z^{n+1}-1).
\end{equation}
If $|f(z)| + |g(z)| > |a_n||z^{n+1}-1|$ holds on $\partial \DD$, then $N(g) = N(f) + 1$.
\end{lemma} 
\begin{proof}
Multiply Eq. \eqref{genForm} by $z$ and subtract Eq. \eqref{defRot} to obtain Eq. \eqref{eqRotDif}. For the statement about the number of roots, apply the symmetric version of Rouch\'es theorem \cite{Este} to $g(z)$, $zf(z)$, $a_n(z^n-1)$ on $\partial \DD$ and use $N(zf)=1+N(f)$.
\end{proof}
If $f(z) \in \LL_n$, then $\Rot{f}{} \in \LL_n$, and vice versa. In particular, as long as $|f(z)| > |z^{n+1}-1|$ on $\partial\DD$, we can increase or decrease $N(f)$ by rotating it.
\begin{corollary}\label{rotL}
Let $f(z) \in \LL_n$ satisfy $N(f) = k_0$, $\min_{|z|=1}{\abs{f(z)}} = \mu$. Then the pairs $(k, n)$ with $k: \abs{k-k_0} < \lceil \mu/2\rceil -1$ are Littlewood admissible and can be realized by rotations of $f(z)$.
\end{corollary}
\begin{proof}
By equation \eqref{eqRotDif}, $\abs{g(z)} \geq \abs{f(z)} - 2$ and $\abs{f(z)} \geq \abs{g(z)} - 2$. Thus the minima of rotated polynomial is at most $\mu-2$. This makes possible to carry out at least $\lceil \mu/2\rceil -1$ rotations forward or backward while increasing or decreasing $N(f)$ by one through each rotation.
\end{proof}

Lemma \ref{RotLem} yields a somewhat unexpected corollary about the values of rotated polynomials.

\begin{corollary}\label{rotLow}
Assume that in equation \eqref{genForm} the coefficients $a_j$ of $f(z) \in \C[z]$ with $N(f) = k$, $\deg{f}=n$  are all non-zero. Then there exist $l, m \in \Z$, $0 \leq l \leq k$ and $0 \leq m \leq n-k $ and numbers $\zeta$, $\xi \in \partial \DD$,  such that
\[
\abs{\Rot{f}{-l-1}(\zeta)} + \abs{\Rot{f}{-l}(\zeta)} \leq \abs{a_{l}}\abs{\zeta^{n+1}-1},
\]
\[
\abs{\Rot{f}{m}(\xi)} + \abs{\Rot{f}{m+1}(\xi)} \leq \abs{a_{n-m}}\abs{\xi^{n+1}-1},
\]
\end{corollary}

\begin{proof}
As degrees of rotated polynomials stay same, one cannot decrease or increase $N(f)$ indefinitely. Hence, for some non--negative integers $0 \leq l \leq k$ and $0 \leq m \leq n-k$, $\Rot{f}{-l-1}$ or $\Rot{f}{m}$ will not satisfy the inequality of Lemma \ref{RotLem}, which guaranties the existence of the numbers $\zeta$ and $\xi$ with the prescribed property.
\end{proof}

In particular, for every $f \in \LL_n$, some rotation of $f$ will always attain absolute values $\leq 1$ on the unit circle.

\begin{example}
Consider the Barker polynomial of degree $12$ from (taken from Ch.14 of \cite{Bbook}):
\[
f(z) = 1+z+z^2+z^3+z^4-z^5-z^6+z^7+z^8-z^9+z^{10}-z^{11}+z^{12}.
\]
Since $f(z)$ is skew-symmetric and non vanishing on $\partial\DD$, one has $N(f)=6$.
Rotation yields polynomials with $k=5, 6, 7, 8$ (blue entries in Table \ref{rotBarker}). The conditions of Lemma 5.3 are fulfilled in the strong form $\abs{f_j(z)}+\abs{f_{j+1}(z)} > 2$ on $\partial\DD$ for $j=-5, -4, -1, 0$, and it still works for $j=1$.

\begin{table}[h]
\def\arraystretch{1.2}
\small{
\begin{tabular}{@{}lrrrrrrrrrrrrr@{}}
\toprule
$j$ 				        & $-6$ & $\textcolor{blue}{-5}$ &$\textcolor{blue}{-4}$ & $\textcolor{blue}{-3}$ & $-2$ & $\textcolor{blue}{-1}$ & $\textcolor{blue}{0}$ & $\textcolor{blue}{1}$ & $\textcolor{blue}{2}$ & $3$ & $4$ & $5$ & $6$\\
\midrule
$N\left(f_j\right)$  & $6$   & $5$   &$6$   & $7$  & $6$   & \textcolor{blue}{$5$}   & \textcolor{blue}{$6$} & \textcolor{blue}{$7$} & \textcolor{blue}{$8$} & $6$ & $5$ & $8$ & $7$\\
$\min$                          & \tiny{$0.25$}   & \textcolor{magenta}{\tiny{$1.23$}}   &\textcolor{violet}{\tiny{$1.80$}}   & \textcolor{magenta}{\tiny{$1.04$}}  & \tiny{$0.52$}   & \textcolor{magenta}{\tiny{$1.28$}}   & \textcolor{violet}{\tiny{$3.01$}} & \textcolor{magenta}{\tiny{$1.88$}} & {\textcolor{red}{\tiny{$0.10$}}} & \tiny{$0.92$} & \tiny{$0.31$} & \tiny{$0.06$} & \tiny{$0.08$}\\
\bottomrule
\end{tabular}
}
\caption{Rotated Barker polynomials $f_j = \Rot{f}{j}$}\label{rotBarker}
\end{table} 
\end{example}

This example is related to the famous conjecture on \emph{flat} Littlewood polynomials  (more precisely, a one side skew-symmetric version of this conjecture), see Ch.15 of \cite{Bbook} and the references therein.

\begin{conjecture}\label{flatConj}
For every even $n \in \N$ there exists a skew-symmetric polynomial $f(z) \in \LL_n$ such that $|f(z)|>c_1n^{1/2}$ on the unit circle, where $c_1 > 0$ is some constant that is independent of $n$.
\end{conjecture}

Assuming this conjecture, one can apply Corollary \ref{rotL} with $k_0 = n/2$ and $\mu = c_1n^{1/2}$.

\begin{proposition}\label{stdevProp}
Assume that Conjecture \ref{flatConj} is true. Then there exists a constant $c_2 > 0$, such that every pair $(k, n) \in \N^2$, where $n$ is even and $\abs{n/2-k} < 0.5c_2n^{1/2}$ is realized by a rotation of a skew--symmetric one-side flat Littlewood polynomials.
\end{proposition}

Very recently, a resolution of original Littlewood flat polynomial conjecture was announced by Balister, Bolob\'{a}s, Morris, Sahasrabudhe and Tiba \cite{BBMST}, although it is not yet clear whether the polynomials they prove to exist are skew-symmetric or not. 

In \cite{CEF} it was shown how to construct the sequence of skew-symmetric polynomials $f \in \LL_n$ such that $|f(z)|>n^{0.43}$ on $\partial \DD$. One takes $f_0(z)$ to be the aforementioned Barker polynomial of degree $12$, $n_0 := \deg{f_0}$ and defines
\begin{equation}\label{eqConstr}
f_{j+1}(z) := f_j(z^{n_j+1})f_j(z), \quad n_{j+1} := \deg{f_{j+1}} = (n_j+1)^2-1.
\end{equation}
Iteration equation \eqref{eqConstr} produces a sequence $\{f_j\}$ of Littlewood polynomials of degree $n_j = 13^{2^j}-1$. Applying Corollary \ref{rotL} to this sequence, one obtains the following unconditional result:

\begin{theorem}\label{almostflatTeor}
Pairs $(k, n) \in \N^2$, where $n=13^{2^j}-1$, $j \in \N_0$ and $\abs{k-n/2} \leq \lceil{n^{0.43}/2}\rceil - 1$ can be realized by rotations of Littlewood polynomial $f_j(z)$ in equation \eqref{eqConstr}.  
\end{theorem}

\section{Proofs of main Results}

\begin{proof}[Proof of Theorem \ref{mainNewmanThm}]
The smallest degree Newman polynomial $f(z)$ with the property $|f(z)| > 2$ on $\partial\DD$ of degree $38$ is given in Equation \eqref{superNewman}.  It has $N(f)=18$ zeros inside $\DD$. We choose this special polynomial $f(z)$ in Proposition \ref{Add2Prop} to obtain the polynomial  $h(z)$ with $N(h) = 18+r$ zeros in $\DD$ and $U(h) = 0$ zeros on $\partial \DD$. Note that for $r \geq 1$ and $n \geq 39+r$, $h(z) \in \NN_n$. By selecting $r \in \{1, 2, \dots, n-39\}$, one sees that $N(h)$ covers all $k$ in the interval $[19, n-21]$, for $n \geq 40$. By taking the reciprocal polynomials, one sees immediately that $N(h^*)$ covers all $k \in [21, n-19]$. The union of the two sets of integers is
\[
\left([19, n-21] \cup [21, n-19]\right) \cap \N =
\begin{cases}
	\{19, 21\} & \text{ for } n = 40,\\
	[19, n-19] \cap \N & \text{ for } n \geq 41.
\end{cases}					
\]
For $n=40$, the `gap' at $k=20$ can be covered by taking $f(z^{10})$, where $f(z)=1+z+z^4 \in \NN_4$ satisfies $N(f)=2$ and $U(f) = 0$. Therefore, all values $(k, n)$ for $k \in [19, n-19]$, $n \geq 40$ are Newman--admissible.

To prove that pairs $(k, n)$ with $ 3 \leq k \leq 18$ are Newman admissible for all sufficiently large $n$, we use Newman polynomials whose coefficients form a certain \emph{pattern} of digits $\{0, 1\}$: such patterns always start with \emph{a prefix} that is followed by the long repetition of a short sub--string called \emph{period} and end with \emph{a suffix}.

\emph{Cases $6 \leq k \leq 18$:} In Table \ref{largeNewman}, one finds a list of small degree polynomials $f_k(z) \in \NN$ that have $N(f_k)=k$ zeros inside $\DD$ and are always $>1$ in absolute value on the unit circle $\partial \DD$. Let $g_k(z):=f_k(z)+z^n$, where $n \geq \deg{f_k}+1$. By Proposition \ref{Add1Prop}, $N(g_k)=k$. Polynomials $f_k(z)$  themselves  cover cases where $n=\deg{f_k}$. Hence, pairs $(k, n)$, $6 \leq k \leq 18$ with $n \geq \deg{f_k}$ are Newman admissible. Intermediate lower bounds on degrees $n$ that are covered by this construction are listed in Table \ref{leastDegree}. Coefficients of Newman polynomials $g(z)$ constructed in this way have long repetitions of `$0$' and end with one-digit suffix $1$.

Next, we treat remaining values of $k$, in decreasing order. We use more complicated patterns with  longer periods from Table \ref{patternsNewman345}. The procedure to prove values $N(h)=k$, $U(h)=0$ for these $h(z)$ is outlined in Section \ref{sec_autoproof}.

\emph{Case $k=5$:} Consider odd degrees $n$ first. Take $f(z) = 1 + z^3+ z^7 +z^8+z^9$.
By direct computation, $N(f)=5$, $f(-1)=-1$ and $|f(z)| > 1$ for  $z \in \partial \DD \setminus\{-1\}$. By taking $f(z)$ itself, one sees that the pair $(k, n)=(5,9)$ is Newman-admissible. Larger degrees  $n=2m+1$, $m \in \N$, $m \geq 5$ are covered using polynomials $g(z)=f(z)+z^n$. These polynomials have the pattern of coefficients in $1001000111(0)^{2m-9}1$ with period of length $1$. By Proposition \ref{Add1Prop}, $g(z)$ has exactly $N(g)=5$ zeros in $\DD$ . Therefore, all pairs $(5, n)$, where $n \geq 9$ is odd, are Newman-admissible.  For $n$ even, $z=-1$ is a root of $g(z)$, hence our restriction to the odd $n$. Alternatively, for odd $n \geq 15$ one could also use the pattern no.~8  with period of length $2$ from Table \ref{patternsNewman345}.

For the even $n \geq 28$, $n \not \equiv 6 \pmod{8}$, we use Newman polynomials that arise from the pattern no.~7  from Table \ref{patternsNewman345} with the period of length $2$. Missing even degrees $n \equiv 6 \pmod{8}$, $n \geq 14$ are covered by using pattern no.~9  in Table \ref{patternsNewman345} with the period of length $8$.

\emph{Case $k=4$:}
For an even $n$, take $f^*(z) = 1 + z+ z^2 +z^6+z^9$. Here, $f^*(z)$ is reciprocal to $f(z)$ used in the $k=5$ case.   One has $N(f^*) = 4$, $|f^*(z)| > 1$ for $z \in \partial \DD \setminus\{-1\}$ and $f^*(-1)=1$. By Proposition \ref{Add1Prop}, the polynomial $g(z) = f^*(z)+z^n \in \NN_n$ of even degree  $n=2m$, $m \geq 5$ with pattern $1110001001(0)^{2m-10}1$ has $N(g)=4$ zeros in $\DD$.  Therefore, every pair $(4, n)$ with even $n \geq 10$ is Newman admissible. Alternatively, one could cover all even $n \geq 14$ by using pattern no.~4 with period of length $2$ from Table \ref{patternsNewman345}.

For the odd $n \geq 9$, use Newman polynomials of degree $n$ that arise from patterns no.~5  and no.~6  in Table \ref{patternsNewman345} with periods of length $4$ for each residue class $n \equiv 1,3 \pmod{4}$.

\emph{Case $k=3$:} All pairs $(3, n)$ of even degrees $n \geq 16$ are shown to be Newman--admissible by looking at polynomials $h(z)$ in Table \ref{patternsNewman345} with $N(h)=3$ that arise from pattern no.~1. Odd degrees $n \geq 9$ that are $n \not\equiv 3 \pmod{8}$ are covered by using pattern no.~2 with period of length $2$. Degrees $n \geq11$ that are $\equiv 3 \pmod{4}$ are covered by pattern no.~3 with the period of length $4$.

\begin{table}
{\small
\begin{tabular}{@{}lcccccccccccccccc@{}}
\toprule
$k =$ 			& $3$		&	$4$	&	$5$	&	$6$ 	& $7$		& $8$		& $9$		& $10$	 & $11$ & $12$ & $13$ & $14$ & $15$ & $16$ & $17$ & $18$\\
$n \geq$		& $16$		&	$9$	& 	$28$	&	$12$	& $15$		& $14$		& $17$		& $17$	 & $18$ & $18$ & $21$ & $21$ & $23$ & $23$ & $25$ & $25$\\
\bottomrule
\end{tabular}
}
\caption{Degrees of Newman polynomials obtained from patterns in Tables \ref{largeNewman} and \ref{patternsNewman345}.}\label{leastDegree}
\end{table}

Thus, for each $k \geq 3$, we now have $n_k \in \N$, such that the pair $(k, n)$ is Newman--admissible whenever $n \geq n_k$.  Lower bounds $n_k$ for $3 \leq k \leq 18$ are summarized in Table \ref{leastDegree}.

By using explicit examples of Newman polynomials of degree $n \in [k+3, n_k]$ from Table \ref{gapTable}, we can reduce intermediate lower bounds $n_k$ further: it turns out, $(k, n)$ with degrees $n \geq k + 3$ are all admissible for $k \in [4, 18]$. A tiny exception is $k=3$: from examples  in Table \ref{gapTable} and infinite sequences constructed above from patterns in Table \ref{patternsNewman345}, we see that pairs $(3, n)$ and, by reciprocation, $(n-3, n)$ are Newman--admissible for $n \geq 7$.

Thus, $(k, n)$ is Newman-admissible for every $n \geq 7$ and $k$ in the interval $I_n :=[3, \min\{18, n-3\}]$. For $7 \leq n \leq 21$, $I_n=[3, n-3]$ -- which are the full intervals that need to be covered. For $n \geq 22$, $I_n=[3, 18]$. By the reciprocation, pairs $(k, n)$ with $k$ in the interval $J_n := [n-18, n-3]$ for $n \geq 22$ are Newman admissible as well. Thus, admissible is every pair $(k, n)$, where $k$ belongs to the interval
\begin{equation}\label{eqIJ}
		I_n \cup J_n =\begin{cases}
			[3, n-3] & \text{ for } 7 \leq n \leq 37,\\
			[3, 18] \cup [19, 34] = [3, n-3] \; \setminus \; ]18, 19[ & \text{ for } n = 37,\\
			[3, 18] \cup [20, 35] = [3, n-3] \; \setminus \; ]18, 20[ & \text{ for } n = 38,\\
			[3, 18] \cup [21, 36] = [3, n-3] \; \setminus \; ]18, 21[ & \text{ for } n = 39,\\
			[3, n-3] \; \setminus \; ]18, n-18[ & \text{ for } n \geq 40.\\
		\end{cases}					
\end{equation}
Above, $]a, b[ \subset \R$ stands for the open interval between $a, b \in \R$, to avoid the confusion with the notation of a pair $(k, n)$.

From first part of the proof, we already know that `most' pairs $(k, n)$ with
$k \in [19, n-19] \cap \N$, $n \geq 40$, are Newman-admissible. For $n \geq 40$, by combining them with pairs where $k \in I_n \cup J_n$, we cover every $k \in [3, n-3]$.

For $3 \leq n \leq 39$, admissible pairs $(k, n)$ where $k \in I_n \cup J_n$ in \eqref{eqIJ} cover all $k \in [3, n-3]$, except for $3$ missing $k$ cases $(19, 38)$, $(19, 39)$, and $(20, 39)$ that were left out in \eqref{eqIJ}. First two missing pairs can be shown to be Newman--admissible by using $1+z+z^{9}+z^{35}+z^{38}$ and $1+z+z^{21}+z^{39}$, respectively, both of which have $k=19$ zeros in $\DD$ and no zero on $\partial\DD$. The last missing case $(20, 39)$ is then covered by reciprocation of $(19, 39)$.
\end{proof}

\begin{proof}[Proof of Theorem \ref{Thm_k2_Newman}]
It is enough to deal with the $k=2$ case, since the $k=n-2$ follows by reciprocation. For $k=2$, even degrees $n \geq 6$ are covered using the polynomials with the pattern no.~1 from Table \ref{patternsNewman2} pairs $(2, n)$ that have $k=2$ zeros inside $\DD$ and no zero on $\partial \DD$. For $k=2$, the degree $n=4$ can be covered by $f(z)=1+z+z^4$ that originates from Pattern no.~6 with $m=0$. Pattern no.~2 covers every degree $n \geq 5$, $n \equiv 2 \pmod{3}$. Patterns no.~6 and no.~7 cover every $n \geq 4$, $n \equiv 4 \pmod{5}$ and every $n \geq 3$, $n \equiv 3 \pmod{5}$, respectively. All together, these $4$ patterns cover all $n \geq 3$ with an exception of residue classes $n \equiv 1, 7, 15, 21, 25, 27 \pmod{30}$. One easily checks that the values of $n$, smaller or equal to $20$, that are not covered by one of the aforementioned patterns are: $n \in \{1, 2, 7, 15\}$. Cases $n=1$ and $n=2$ are inadmissible in a trivial way; $n=7$  by $1+z+z^2+z^4+z^7$ and $n=15$ by $1+z+z^2+z^4+z^7+z^9+z^{10}+z^{12}+z^{15}$.
\end{proof}

\begin{proof}[Proof of Theorem \ref{notNewmanAdmThm}] For $k \in \{3, n-3\}$, degrees $n \geq 7$, are Newman--admissible by Theorem \ref{mainNewmanThm}. For $n=5$, $f(z)=1+z^3+z^5$ has $k=3$ zeros inside $\DD$ and none on $\partial\DD$. An exhaustive search in $\NN_6$ shows that $(3, 6)$ is not--Newman admissible.  
For $k=2$, the proof again is computational, by complete examination of all Newman polynomials of degree $\leq 35$, see Section \ref{sec_exhaustive}.
\end{proof}

\begin{proof}[Proof of Theorem \ref{Thm_k2_Littlewood}]
By Table \ref{patternsLittlewood2}, for every $n \geq 3$ and $n \not\equiv 1 \pmod{6}$, there exists a regular $f \in \LL_n$ with $N(f)=2$ and $N(f^*)=n-2$. For $n$ in $3 \leq n \leq 12$, only $n = 7$ has remainder $1 \pmod{6}$.  To see that pairs $(2, 7)$ and $(5, 7)$ are Littlewood--admissible, consider sporadic $f(z)=1+z+z^2 -z^3 +z^4 - z^5 +z^6 -z^7$ with $N(f)=2$.
\end{proof}

\begin{proof}[Proof of Theorem \ref{Thm_k3to11_Littlewood}]
Let $k$ be $3 \leq k \leq 11$. By Table \ref{patternsLittlewood3to11}, there exists $f \in \LL_n$ with $N(f)=k$ for every $n \geq k+3$ if $k$ is odd and for every $n \geq k+4$ if $k$ is even. To cover the values $n=k+3$ for $k=4$, $6$, $8$, $10$, for each $k$ take a Littlewood polynomial $f(z)$ of degree $n=k+3$ with $N(f)=3$ that corresponds to the pattern no.~1 from Table \ref{patternsLittlewood3to11}: its reciprocal then has $\deg{f^*}=k+3$ and $N(f^*)=n-3=k$.
\end{proof}

\begin{proof}[Proof of Theorem \ref{mainLittlewoodThm}]
If the pair $(k, n)$ is realized by $f(z)$ with $N(f)=k$, $U(f)=0$, then the pair $(n-k, n)$ is realized by $f^*(z)$. The reflection $(k, n) \mapsto (n-k, n)$ between valid pairs preserves the admissibility.

Assume that $n$ is even. By Corollary \ref{colSPL}, every odd $k \in [3, n/2]$ is Littlewood admissible. By reflection, $(k', n)$ with $k' = n-k$ are also admissible and they constitute all odd integers $k' \in [n/2, n-3]$. Hence, every odd $k \in [3, n/2] \cup [n/2, n-3] = [3, n-3]$ for $n \geq 6$ is also admissible. In a similar way, every even $k \in [n/2, 2n/3]$ is admissible, and so is every even $k'=n-k \in [n/3, n/2]$. This means every even $k \in [n/3, n/2] \cup [n/2, 2n/3]=[n/3, 2n/3]$ is admissible.

Finally, assume that $n$ is odd. Every $(k, n)$ with odd $k \in [3, n/2]$ is admissible by Corollary \ref{colSPL}. Hence, every $(k', n)$ with even $k'=n-k \in [n/2, n-3]$ is admissible as well. It remains to note that for $n \geq 9$, intervals $[3, n/2]$ and $[n/2, n-3]$ absorb respective intervals $[n/3, n/2]$ and $[n/2, 2n/3]$ that can be obtained from  the $n$--odd, $k$--even case in Corollary \ref{colSPL} and its reflection. 
\end{proof}

\begin{proof}[Proof of Theorem \ref{thm_inadm_Littl}] 
The proof is computational, by complete examination of $\LL_n$ for $2 \leq n \leq 31$, see Section \ref{sec_exhaustive}.
\end{proof}

\begin{proof}[Proof of Theorem \ref{thmSPL}]
Formula \eqref{eqSPL} is obtained directly from the pattern $(+-+)^m+(-)^l$. It expands into
\[
h(z)=\frac{1 - z + z^2 - 2z^{2+3m}(1+z)  + z^{3m+l+1}(1+z+z^2)}{1-z^3}.
\]
Rewrite it as
\begin{equation}\label{eqForm1}
h(z) = \frac{f(z) + z^{3m+l+1}g(z)}{1-z},
\end{equation}
where
\begin{equation}\label{eqForm2}
f(z) :=\frac{1 - z + z^2 - 2z^{2+3m}(1+z)}{1+z+z^2}, \qquad g(z):=1.
\end{equation}

Denote the  numerator of Eq. \eqref{eqForm1} by $h_1(z):=(1-z)h(z)$ and the numerator polynomial in Eq. \eqref{eqForm2} by $f_1(z):=(1+z+z^2)f(z)$. By Lemma \ref{lemPrep1}, $\abs{f(z)} \geq \abs{g(z)}$ for $z \in \partial\DD$, with equality attained at points $z=1$, $e^{\pm 2\pi i/3}$ that are the only possible unimodular zeros of $h_1(z)$. Calculations show that $h_1(1) = 0$, $h_1'(1)=l-m-1$. Thus $z=1$ is a simple zero of $h_1(z)$ whenever $l \ne m+1$. A similar, but lengthier calculation shows that $h_1(e^{\pm 2\pi i/3}) \ne 0$ for $l, m \in \N$.

We now apply the formula of Boyd from Proposition \ref{propBoyd} to $h_1(z)$, $f(z)$ and $g(z)$ with $n=3m+l+1$ (here, $n$ denotes the power of $z$ in Proposition \ref{propBoyd}, not the degree of $h(z)$). We know that $z=1$ is the only possible exit point. By using the identity $zf'(z)/f(z) = z f_1'(z)/f_1(z) - z(1+2z)/(1+z+z^2)$, one evaluates $f'(1)/f(1)= (-9-12m)/(-3)-1 = 2+4m$. Since $g'(z)=0$, $n_0 = \Re\left.{\left(zf'(z)/f(z) - zg'(z)/g(z)\right)}\right|_{z=1} = 2+4m$ in Proposition \ref{propBoyd}. Therefore, $E(f, g) = 1$, for $n < n_0$ and $E(f, g) = 0$ for $n>n_0$. Since $n=3m+l+1$, $n_0=4m+2$, this is equivalent to $l < m+1$ and  $l > m+1$, respectively. From Lemma \ref{lemPrep2}, we know that $N(f)=2m+1$, $U(f)=0$. Boyd's formula yields $N(h)=N(h_1)=N(f) - E(f, g)=2m$, if $l < m+1$, and $N(h)=2m+1$, if $l>m+1$. This completes the proof.
\end{proof}

\begin{proof}[Proof of Corollary \ref{colSPL}]
Consider the even $k \geq 2$ case first. To obtain $f \in \LL_n$ with $N(f)=k$, $U(f)=0$ from Theorem \ref{thmSPL},  set $m = k/2$, $l=n-3m=n-3k/2$ in the `even' case of the formula therein. Then $k \geq 2$ yields $m \geq 1$, and $k\leq 2n/3$ results in $l \geq 0$. One must ensure that $l<m+1$. Since we are working with integers, this is equivalent to $l \leq m$, or $n-3m \leq m$, and this translates to $k \geq n/2$. Thus, every $(k, n)$ with even $k \in [n/2, 2n/3]$ is Littlewood--admissible.

For $k$ odd, pick $m = (k-1)/2$, $l=n-3m$ in the `odd' case of Theorem \ref{thmSPL}. For $k$ in range $3 \leq k \leq n/2$, $m \geq 1$ and $l = n - 3(k-1)/2 \geq (n+6)/4 \geq 2$. For integers $l$ and $m$, the inequality $l>m+1$ means $l \geq m+2$.  The later is equivalent to $n\geq 4m+2= 2k$, which is true. Thus, for every odd $k \in [3, n/2]$, $(k, n)$ is Littlewood--admissible.
\end{proof}

\begin{table}
\def\arraystretch{1.2}
\begin{tabular}{@{}cccl@{}}
\toprule
$N(f)$ & $\deg{f}$ & $\min{f}$ & Coefficients of $f(z)$, low to high\\
\midrule
$6$ & $12$ & $1.362$ & $1111100110101$\\
$7$ & $15$ & $1.183$ & $1100111100000101$\\
$8$ & $14$ & $1.025$ & $110011001011111$\\
$9$ & $17$ & $1.248$ & $100101000100011111$\\
$10$ & $17$ & $1.254$ & $100000101000110111$\\
$11$ & $18$ & $1.018$ & $1100011000010111011$\\
$12$ & $18$ & $1.029$ & $1000101001100101111$\\
$13$ & $21$ & $1.235$ & $1001001101010111001111$\\
$14$ & $21$ & $1.218$ & $1000100010001010010111$\\
$15$ & $23$ & $1.365$ & $100001010100111001011011$\\
$16$ & $23$ & $1.167$ & $100000101001101000110111$\\
$17$ & $25$ & $1.161$ & $10000100111001010111010011$\\
$18$ & $25$ & $1.143$ & $10101010101011010110011111$\\
\bottomrule
\end{tabular}
\caption{$f(z) \in \NN$ of smallest possible degree with largest minima greater than 1 for $6 \leq N(f) \leq 18$.}\label{largeNewman}
\end{table}

\begin{landscape}
\begin{table}
\begin{tabular}{@{}rcllccl@{}}
\toprule
no. & Degree					& Pattern 						& Polynomial $h(z)$ \\
\midrule
1. 	& $6+2m$, $m \geq 0$	& $110(01)^m0001$			& $\left(1+z-z^2-z^3+z^4 - z^{4+2m}(1-z^2+z^4)\right)/(1-z^2)$\\
2. 	& $2+3m$, $m \geq 1$	& $111(011)^m$				& $\left(1+z+z^2-z^3 - z^{4+3m}(1+z)\right)/(1-z^3)$\\
3.	& $2+4m$, $m \geq 1$	& $101(0001)^m$ 			& $\left(1+z^2-z^4 - z^{6+4m}\right)/(1-z^4)$\\
4. 	& $2+4m$, $m \geq 1$	& $101(1001)^m$			& $\left(1-z+2z^2-z^3 -z^{3+4m}(1-z+z^2)\right)/(1-z+z^2-z^3)$\\
5.	& \begin{tabular}{c}
				$4+4m$, $m \geq 1$,\\
				$m \not\equiv 0 \pmod{3}$
	 \end{tabular}  & $11011(01)^{2m}$					& $\left(1+z-z^2+z^4-z^5-z^{6+4m}\right)/(1-z^2)$\\
6. 	& $4+5m$, $m \geq 0$	& $11001(01001)^m$		& $\left(1+z+z^4-z^5 -z^{6+5m}(1+z^3)\right)/(1-z^5)$\\
7. 	& $3+5m$, $m \geq 0$	& $1011(00011)^m$			& $\left(1+z^2+z^3-z^5-z^{7+5m}(1+z)\right)/(1-z^5)$\\
8. 	& $4+6m$, $m \geq 0$	& $10111(000111)^m$		& $\left(1-z+z^2+z^3-z^4-z^{8+6m}\right)/(1-z+z^3-z^4)$\\
9. 	& $2+6m$, $m \geq 1$	& $111(001)^{2m}$			& $\left(1+z+z^2-z^3-z^4-z^{5+6m}\right)/(1-z^3)$	\\
10. & $4+8m$, $m \geq 0$	& $11101(00101101)^m$	& $\left(1-z^3+2z^4-z^5-z^{7+8m}(1-z+z^2)\right)/(1-z+z^4-z^5)$\\
\bottomrule
\end{tabular}
\caption{Regular $h(z) \in \NN$ with $N(h)=2, U(h)=0$.}\label{patternsNewman2} 
\end{table}

\begin{table}
\begin{tabular}{@{}llll@{}}
\toprule
no. & Degree 					& Pattern				& Polynomial $h(z)$ \\
\midrule
1.	& $2+2m$, $m \geq 1$	& $++(-)^{2m+1}$	& $(1 - 2z^2 + z^{3+2m})/(1-z)$ \\
2.  & $2+2m$, $m \geq 1$	& $++-(-+)^m$		& $(1 + 2z - 2z^3 +z^{3+2m})/(1+z)$ \\
3.  & $3m$, $m \geq 1$		& $+(++-)^m$		& $(1 + z + z^2 - 2z^3 - z^{1+3m}(1+z-z^2))/(1 - z^3)$ \\
4.  & $2+3m$, $m \geq 1$	& $+++(-++)^m$	& $(1 + z + z^2 - 2z^3 + z^{3+3m}(1 - z - z^2))/(1 - z^3)$ \\
\bottomrule
\end{tabular} 
\caption{Regular $h(z) \in \LL$ with $N(h)=2$ and $U(h)=0$, normalized by $a_0=a_1=1$.}\label{patternsLittlewood2} 
\end{table}

\begin{footnotesize}
\begin{table}
\begin{tabular}{@{}lclll@{}}
\toprule
no. & $N(h)$ & Degree  & Pattern & Polynomial $h(z)$ \\
\midrule
1 & $3$	& $12+2m$, $m \geq 2$					& $110010(01)^m0001001$					& \begin{tabular}{l}
																												$\left(1+z-z^2-z^3+z^4-z^6+z^7 - \dots \right.$\\
																												$\left. \dots - z^{7+2m}(1-z^2+z^4-z^5+z^7)\right)/(1-z^2)$
																											\end{tabular}\\
2 & $3$	& \begin{tabular}{c}
				$9+2m$, $m \geq 0$\\
				$m \not\equiv 1 \pmod{4}$
		    \end{tabular}								& $110011(01)^m0001$						& $\left(1+z-z^2-z^3+z^4+z^5-z^6-z^{7+2m}(1-z^2+z^4)\right)/(1-z^2)$\\

3 & $3$	& $3+4m$, $m \geq 2$					& $1011(0001)^m$							&	$\left(1+z^2+z^3-z^4-z^6 - z^{7+4m}\right)/(1-z^4)$\\

4 & $4$	& $14+2m$, $m \geq 0$					& $111001011(01)^m000101$ 				& $\left(1+z-z^3-z^4+z^5+z^8-z^9 - z^{10+2m}(1-z^2+z^6)\right)/(1-z^2)$\\

5 & $4$	& $9+4m$, $m \geq 0$					& $1011(1100)^m111011$					& \begin{tabular}{l}
																										$\left(1-z+2z^2-z^3+z^4-z^6 - \dots \right.$\\
																										$\left.  \dots -z^{6+4m}(1-z+z^2-z^3-z^6)\right)/(1-z+z^2-z^3)$
																							   			\end{tabular}\\

6 & $4$ 	& $7+4m$, $m \geq 1$					& $11101(0110)^m011$						& $\left(1+z^2-z^3+z^4-2z^5+2z^6-z^7-z^{10+4m}\right)/(1-z+z^2-z^3)$\\

7 & $5$ 	& \begin{tabular}{c}
				$22+2m$, $m \geq 3$\\
				$m \not\equiv 0 \pmod{4}$
			\end{tabular}							& $11001001001011(01)^m001001001$	&	\begin{tabular}{l}
																											$\left(1+z-z^2-z^3+z^4-z^6+z^7-z^9+z^{10}+z^{13}-z^{14} -\dots\right.$\\
																											$\left. \dots -z^{15+2m}(1-z+z^3-z^4+z^6-z^7+z^9)\right)/(1-z^2)$
																										\end{tabular}	\\

8 & $5$	& $15+2m$, $m \geq 0$					& $110011011011(01)^m0001$				& \begin{tabular}{l}
																											$\left(1+z-z^2-z^3+z^4+z^5-z^6+z^8-z^9+z^{11} -z^{12} - \dots \right.$ \\
																											$\left. \dots - z^{13+2m}(1-z^2+z^4)\right)/(1-z^2)$
																										\end{tabular}	\\

9 & $5$ 	& $6+8m$, $m \geq 1$					& $1001011(10001011)^m$					& $(1+z^3+z^5+z^6+z^7-z^8-z^{7+8m}(1+z^4+z^6+z^7))/(1-z^8)$\\ 
																	
\bottomrule
\end{tabular} 
\caption{Examples of $h(z) \in \NN$ with $N(h)=3, 4, 5$ and $U(h)=0$.}\label{patternsNewman345} 
\end{table}

\begin{table}
\begin{tabular}{@{}lclll@{}}
\toprule
no. & $N(h)$ & Degree  & Pattern & Polynomial $h(z)$ \\
\midrule
1 & $3$	& $3+m$, $m \geq 3$	& $+-++(-)^m$			& $(1 - 2z + 2z^2 - 2z^4 + z^{4+m})/(1-z)$\\
2 & $4$	& $5+m$, $m \geq 3$ 	& $+---+-(+)^m$		& $(1 - 2z + 2z^4 - 2z^5 + 2z^6 -z^{6+m})/(1 - z)$\\
3 & $5$	& $5+m$, $m \geq 3$	& $+-+-++(-)^m$		& $(1 - 2z + 2z^2 - 2z^3 + 2z^4 - 2z^6 - z^{6+m}/(1 - z)$\\
4 & $6$	& $7+m$, $m \geq 3$	& $+-+-+---(+)^m$ 		& $(1 - 2z + 2z^2 - 2z^3 + 2z^4 - 2z^5 + 2z^8 - z^{8+m})/(1 - z)$\\
5 & $7$	& $7+m$, $m \geq 3$	& $+-+-+-++(-)^m$	& $(1 - 2z + 2z^2 - 2z^3 + 2z^4 - 2z^5 + 2z^6 - 2z^8 - z^{8+m})/(1 - z)$\\
6 & $8$ 	& $9+m$, $m \geq 3$	& $++---+++-+(-)^m$	& $(1 - 2z^2 + 2z^5 - 2z^8 + 2z^9 - 2z^{10} + z^{10+m})/(1 - z)$\\
7 & $9$ 	& $9+m$, $m \geq 3$	& $+-+-+-+-++(-)^m$	&	\begin{tabular}{l}
																			$\left(1 - 2z + 2z^2 - 2z^3 + 2z^4 - 2z^5 + 2z^6 -\dots\right.$\\
																			$\left. \dots - 2z^7 + 2z^8 - 2z^{10} + z^{10+m}\right)/(1-z) $
																		\end{tabular}	\\
8 & $10$	& $11+m$, $m \geq 3$	& $+---+++---+-(+)^m$ & $(1 - 2z + 2z^4 - 2z^7 + 2z^{10} - 2z^{11} + 2z^{12} -z^{12+m})/(1 - z)$\\
9 & $11$ 	& $11+m$, $m \geq 3$	& $+-+-+-+-+-++(-)^m$	& \begin{tabular}{l}
																			$\left(1 - 2z + 2z^2 - 2z^3 + 2z^4 - 2z^5 + 2z^6 - 2z^7  +\dots\right.$\\
																			$\left. \dots + 2z^8 - 2z^9 + 2z^{10} - 2z^{12} + z^{12+m}\right)/(1-z) $
																		\end{tabular}	\\
\bottomrule
\end{tabular}
\caption{Examples of $h(z) \in \LL$ with $3 \leq N(h) \leq 11$ and $U(h)=0$.}\label{patternsLittlewood3to11} 
\end{table}
\end{footnotesize}
\end{landscape}

\begin{table}
\renewcommand{\arraystretch}{1.2}
{\small
\begin{tabular}{@{}lll@{}}
\toprule
 ${k=3}$										& ${k=6}$											& ${k=13}$ \\
$1 + z + z^{7}$								& $1 + z^{6} + z^{9}$ 							&	$1 + z^{7} + z^{9} + z^{15} + z^{16}$\\
$1 + z + z^{3} + z^{5} + z^{8}$			& $1 + z^{6} + z^{10}$						& 	$1 + z^{11} + z^{13} + z^{17}$\\
$1 + z + z^{9}$								& $1 + z^{6} + z^{11}$						& 	$1 + z^{5} + z^{11} + z^{15} + z^{18}$\\
$1 + z + z^{3} + z^{5} + z^{10}$		& ${k=7}$											& 	$1 + z^{5} + z^{17} + z^{19}$\\
$1 + z + z^{3} + z^{11}$					& $1 + z^{3} + z^{5} + z^{9} + z^{10}$	& 	$1 + z + z^{13} + z^{19} + z^{20}$\\
$1 + z + z^{5} + z^{7} + z^{12}$		& $1 + z^{9} + z^{11}$						& 	${k=14}$\\
$1 + z + z^{3} + z^{13}$					& $1 + z + z^{7} + z^{11} + z^{12}$		& 	$1 + z^{12} + z^{16} + z^{17}$\\
$1 + z + z^{7} + z^{9} + z^{14}$		& $1 + z^{7} + z^{13}$						& 	$1 + z^{15} + z^{16} + z^{18}$\\
$1 + z^{3} + z^{9} + z^{15}$			& $1 + z + z^{3} + z^{9} + z^{14}$			& 	$1 + z^{8} + z^{16} + z^{19}$\\
${k=4}$											& ${k=8}$ 											& 	$1 + z^2 + z^{19} + z^{20}$\\
$1 + z^{4} + z^{7}$						& $1 + z^{2} + z^{10} + z^{11}$			& 	${k=15}$\\
$1 + z^{2} + z^{8}$						& $1 + z^{8} + z^{12}$						& 	$1 + z^{5} + z^{11} + z^{13} + z^{17} + z^{18}$\\
${k=5}$											& $1 + z^{10} + z^{13}$						& 	$1 + z^{5} + z^{11} + z^{16} + z^{19}$\\
$1 + z + z^{5} + z^{7} + z^{8}$			& ${k=9}$ 											& 	$1 + z^{7} + z^{13} + z^{17} + z^{20}$\\
 $1 + z^{5} + z^{9}$						& $1 + z^{5} + z^{7} + z^{11} + z^{12}$	& 	$1 + z^{15} + z^{17} + z^{21}$\\
 $1 + z + z^{3} + z^{7} + z^{10}$		& $1 + z^{3} + z^{11} + z^{13}$ 			& 	$1 + z^{3} + z^{17} + z^{21} + z^{22}$\\
 $1 + z^{3} + z^{11}$						& $1 + z + z^{9} + z^{13} + z^{14}$ 		& 	${k=16}$\\
 $1 + z + z^{3} + z^{5} + z^{12}$		& $1 + z^{9} + z^{15}$ 						& 	$1 + z^{14} + z^{15} + z^{16} + z^{19}$\\
 $1 + z + z^{13}$							& $1 + z + z^{9} + z^{13} + z^{16}$		& 	$1 + z^{8} + z^{16} + z^{20}$\\
$1 + z + z^{3} + z^{5} + z^{14}$		& ${k=10}$ 											& 	$1 + z^{10} + z^{20} + z^{21}$\\
 $1 + z + z^{15}$							& $1 + z^{6} + z^{12} + z^{13}$			& 	$1 + z^{4} + z^{20} + z^{22}$\\
 $1 + z + z^{3} + z^{5} + z^{16}$		& $1 + z^{2} + z^{13} + z^{14}$			& 	${k=17}$\\
 $1 + z + z^{7} + z^{17}$					& $1 + z^{10} + z^{15}$						& 	$1 + z^{2} + z^{9} + z^{11} + z^{17} + z^{19} + z^{20}$\\
 $1 + z + z^{3} + z^{5} + z^{18}$		& $1 + z^{13} + z^{16}$						& 	$1 + z^{9} + z^{11} + z^{20} + z^{21}$\\
 $1 + z + z^{3} + z^{19}$					& ${k=11}$ 											& 	$1 + z^{11} + z^{17} + z^{19} + z^{22}$\\
 $1 + z + z^{3} + z^{7} + z^{20}$		& $1 + z^{5} + z^{7} + z^{13} + z^{14}$	& 	$1 + z^{13} + z^{21} + z^{23}$\\
 $1 + z + z^{3} + z^{21}$					& $1 + z^{9} + z^{13} + z^{15}$			& 	$1 + z + z^{21} + z^{23} + z^{24}$\\
 $1 + z + z^{3} + z^{7} + z^{22}$		& $1 + z + z^{13} + z^{15} + z^{16}$		& 	${k=18}$\\
 $1 + z + z^{3} + z^{23}$					& $1 + z^{15} + z^{17}$						& 	$1 + z^{6} + z^{12} + z^{18} + z^{21}$\\
 $1 + z + z^{5} + z^{7} + z^{24}$		& ${k=12}$ 											& 	$1 + z^{2} + z^{14} + z^{16} + z^{22}$\\
 $1 + z + z^{3} + z^{25}$					& $1 + z^{6} + z^{12} + z^{15}$			& 	$1 + z^{14} + z^{22} + z^{23}$\\
 $1 + z + z^{3} + z^{7} + z^{26}$		& $1 + z^{4} + z^{14} + z^{16}$			& 	$1 + z^{6} + z^{21} + z^{24}$\\
 $1 + z + z^{5} + z^{27}$					& $1 + z^{2} + z^{16} + z^{17}$			&	\\
\bottomrule
\end{tabular}
}
\caption{Newman polynomials $f(z) \in \NN$ with \newline $N(f)=k$, $3 \leq k \leq 18$ for missing low degrees.}\label{gapTable}
\end{table}

\section{Auxiliary propositions and lemmas}\label{sec_Lem}

Propositions \ref{Add1Prop}, \ref{Add2Prop}, \ref{propBoyd}, \ref{UnimodularProp} that are collected in Section \ref{sec_Lem} for the reference in our proofs are variations on a basic method due to Salem \cite{S1, S2}, and Pisot and Dufresnoy \cite{DP1, DP2}. See also \cite{PisotBook, SalemBook}. Subsequently, this technique was applied by numerous authors \cite{C1, C2, Cant, Grand, Smy2}. Boyd perfected it in \cite{B1}: his result is essential to our applications and is recreated as Proposition \ref{propBoyd}.

Criteria stated as Lemmas \ref{ResolventLem} and \ref{MultiplicityLem} were used in actual computations. The last two lemmas, \ref{lemPrep1} and \ref{lemPrep2}, concern auxiliary polynomials that arise in the proof of Theorem \ref{mainLittlewoodThm}.

In Theorem \ref{mainNewmanThm}, we will make use of two simple constructions that involve a polynomial which stays large on $\partial \DD$.

\begin{proposition}\label{Add1Prop}
Let $f(z) \in \C[z]$ and $m := \min_ {\abs{z}=1} \abs{f(z)}$. If $m>1$, then, for every $n \in \N$, the polynomial $g(z) := f(z) + z^n$ has $N(g) = N(f)$ zeros in $\DD$ and $U(g)=0$ zeros on $\partial\DD$. If $m=1$, then $N(g) = N(f)$,  $U(g)=0$ still holds provided that $g(z) \ne 0$ at points $z \in \partial\DD$ where $|f(z)|=1$.
\end{proposition}
\begin{proof}
Take a real number $t>1$. By Rouch\'e's Theorem,  for the polynomial  $g_t(z)=tf(z) + z^n$ one has $N(g) = N(f)$. By the continuity of roots with respect to $t$, $N(g_t) = N(g)$ as  $t \to 1$  if no root of $g_t(z)$ crosses the unit circle, that is, $g(z) \ne 0$ for $z \in \partial \DD$. If $m>1$, this is obvious. If $m=1$, this is true if $g(z) \ne 0$ at all points $z \in \partial\DD$ where this minima $|f(z)|=1$ is achieved.
\end{proof}

\begin{proposition}\label{Add2Prop}
Suppose that $f(z) \in \C[z]$ satisfies $|f(z)|>2$ for every $z \in \partial\DD$. Then, for every $n, r \in \N$, the polynomial
\[
h(z) = 1 + z^rf(z) + z^n
\]
has $N(h) = N(f)+r$ zeros in $\DD$ and $U(h)=0$ zeros on $\partial\DD$.
\end{proposition}
\begin{proof}
Apply Rouch\'e's Theorem to polynomials $z^rf(z)$ and $1+z^n$.
\end{proof}

In more complicated cases, we will need to determine zero numbers for polynomials of the form $f(z)+z^ng(z)$, where $\abs{f(z)} \geq \abs{g(z)}$ that constantly appear through the literature on Pisot and Salem numbers \cite{PisotBook, SalemBook}. Their reemergence is not a coincidence: Cantor, based on the earlier work of Pisot and Dufresnoy \cite{DP1, DP2},  has shown that the polynomials of this shape always arise from the limit functions in the Schur tree, see Theorem 5.7 in \cite{Cant}, and also papers \cite{B4, Grand}.

To handle such situations, we supply a variant of Boyd's formula from \cite{B1}. The original formula in \cite{B1} and a variant exploited in \cite{BCFJ} were specialized to the case $g(z)=f^*(z)$. The proof that is presented here is a slightly more explanatory than the original one in \cite{B1}. The main idea, however, remains the same: to count the number of \emph{exit points} by tracking the sign of certain derivative. 

\begin{proposition}[Boyd's exit point formula]\label{propBoyd}
Let $h(z)=f(z)+z^ng(z)$, $n \in \N$, where $f, g \in \C[z]\setminus\{0\}$ satisfy $\abs{f(z)} \geq \abs{g(z)}$ for $z \in \partial \DD$. If all unimodular zeros of $\zeta \in \partial \DD$ of $h(z)$ are simple, and $f(\zeta) \ne 0$, then
\[
N(h)=N(f)-E(f, g),
\] where
\begin{equation}\label{eqDefE}
E(f, g) := \#\left\{ \zeta \in \partial \DD: h(\zeta) = 0, \; \Re\left(\frac{\zeta f'(\zeta)}{f(\zeta)} -\frac{\zeta g'(\zeta)}{g(\zeta)}\right) > n\right\}
\end{equation}
denotes the number of exit points. In particular, for every $n > n_0$, where the quantity 
\begin{equation}\label{eq_n0}
n_0 := \max_{\substack{ h(\zeta)=0,\\ \zeta \in \partial \DD}} \Re\left(\frac{\zeta f'(\zeta)}{f(\zeta)} -\frac{\zeta g'(\zeta)}{g(\zeta)}\right),
\end{equation}
one has $E(f, g)=0$ and $N(h)=N(f)$.
\end{proposition}

\begin{proof}
Consider
\[
h(z, t) := tf(z)+z^ng(z) \in \C[z, t],
\] where $t$ is taken to be real and $\geq 1$. Let us denote by $z(t)$ the branch of algebraic function, defined by the equation $h(z(t), t)=0$, that satisfies $z(1) = \zeta \in \partial{\DD}$. By the complex version of Implicit Function Theorem (see Theorem 3.1.4 in \cite{Schei}), $z(t)$ is well--defined and locally differentiable for all $t$ where $\partial{h}(z, t)/\partial{z}\ne 0$, and, in that case,
\begin{equation}\label{eqDerivative}
\frac{1}{z}\frac{dz}{dt} = -\frac{1}{z}\frac{\partial{h}(z, t)/\partial{t}}{\partial{h}(z, t)/\partial{z}}= \frac{1}{t}\cdot\frac{1}{n + zg'(z)/g(z) - zf'(z)/f(z)}.
\end{equation}
The non--vanishing of the denominator for $t=1$ in equation \eqref{eqDerivative} is equivalent to $h'(\zeta) \ne 0$, which holds since $\zeta$ is a simple root of $h(z)$.  As $\abs{f(z)} \geq \abs{g(z)}$ for $z \in \partial\DD$, one has $N(h_t)=N(f)$ for $t > 1$. By the continuity of $z(t)$, $N(h) = N(f) - E(f, g)$, where $E(f, g)$ is the number of branches $z(t)$ that cross $\partial\DD$ at $t=1$ from inside of $\DD$. Due to $f(\zeta) h'(\zeta) \ne 0$, $z'(1) \ne 0$ by Eq. \eqref{eqDerivative}. This means that, for an exit point, $\abs{z(t)}^2$ is increasing as $t \downarrow 1$. Since $t$ is \emph{decreasing} to $1$ from above, the later is equivalent to $\partial{\abs{z(t)}^2}/\partial{t} < 0$: by differentiating and using $\overline{\zeta}=1/\zeta$, one finds that this is equivalent $\Re{\left((1/z)\partial{z}/\partial{t}\right)} <0$ at $t=1$, which, in turn, through Eq. \eqref{eqDerivative} yields Eq. \eqref{eqDefE}.  
\end{proof}

We will need a convenient formal criterion to check the inequality $\abs{f(z)} \geq \abs{g(z)}$ on the computer.

 \begin{lemma}\label{ResolventLem}
Let $d:=\max\{\deg{f}, \deg{g}\}$, where $f(z)$, $g(z) \in \C[z]$. Then $\abs{f(z)} \geq \abs{g(z)}$ on the unit circle $|z|=1$ if and  only the resolvent polynomial
$R(z) := f(z)f^{*}(z)z^{d - \deg{f}} - g(z)g^*(z)z^{d - \deg{g}}$ satisfies one of the two conditions:
\begin{enumerate}
\item either $R(z)$ is a zero polynomial in $\C[z]$, or
\label{ResolventLemi}
\item The middle coefficient $r_{d}$ of $z^d$ term in $R(z)$ must be $>0$ (so in particular, $\deg{R}$ is even), and every unimodular zero of $R(z)$ is of even multiplicity.
\label{ResolventLemii}
\end{enumerate}
\end{lemma}
\begin{proof}
As Case \ref{ResolventLem} is trivial, we can assume $\abs{f(z)} \ne \abs{g(z)}$ for at least one point $z$ on the unit circle. Set $z=e^{it}$, $t \in [0, 2\pi)$ and let $H(t) = \abs{f(e^{it})}^2 - \abs{g(e^{it})}^2$.
For $|z|=1$,
\[
\abs{f(z)}^2 - \abs{g(z)}^2 = f(z)\overline{f}(1/z)-g(z)\overline{g}(1/z)=z^{-d}R(z).
\]
Therefore, $H(t) = e^{-idt}R(e^{it})$. Notice that
\begin{equation}\label{eqAvg}
r_{d} = \frac{1}{2\pi}\int_{0}^{2\pi}e^{-idt}R(e^{it})\,dt.
\end{equation}
If $|f(z)| \geq |g(z)|$ and $R(z) \not\equiv 0$, then $H(t) \geq 0$ and $H(t)$ is not identically zero, which yields $r_d > 0$.
By repeated differentiation, one finds
\begin{equation}\label{eqDerH}
H^{(s)}(t) = i^s e^{it(s-d)} R^{(s)}(e^{it}) + \dots\text{ lower derivatives } R^{(j)}(e^{it}),
\end{equation}
for $s=0, 1, \dots, \deg{R(z)}$.

As $\abs{f(z)} \geq \abs{g(z)}$ is equivalent to $H(t) \geq 0$, any zero $t_0 \in [0, 2\pi)$ of $H(t)$ must be the point of local minimum. This means
\begin{equation}\label{derHeq}
H(t_0) = H'(t_0)= \dots = H^{2l-1}(t_0)=0, H^{2l}(t_0) \ne 0,
\end{equation}
for some $l \in \N$. By \eqref{eqDerH}, equations in \eqref{derHeq}, for $z_0 = e^{it_0}$, are equivalent to
\begin{equation}\label{derReq}
R(z_0) = R'(z_0) = \dots = R^{2l-1}(z_0) = 0, R^{2l}(z_0) \ne 0,
\end{equation} 
so $z_0$ a unimodular zero of $R$ of even multiplicity. 

Conversely, assume that all zeros $z_j=e^{it_j}$, $t_j \in [0, 2\pi)$ of $R(z)$ are of even multiplicity $2l_j$, $1 \leq j \leq m$. As equations \eqref{derReq} lead back to equations \eqref{derHeq}, $t_j$ must be also the points of local extrema (minima or maxima) of $H(t)$. It follows from this that $H(t)$ must have the same sign for every other $t \ne t_j$. If $r_d > 0$, by the average value formula \eqref{eqAvg},  $H(t) \geq 0$.
\end{proof}

Next, we formalize the detection of unimodular zeros among polynomials $f(z)+z^ng(z)$. The periodic recurrence of cyclotomic factors that is mentioned in Proposition \ref{UnimodularProp} was already exploited many times to determine all reducible members in polynomial sequences of this type by Schinzel \cite{schi}, Smyth \cite{Smy2}, Filaseta and others \cite{FiFoKo, FilMat}. 

\begin{proposition}\label{UnimodularProp}
Let $\zeta \in \partial \DD$ be a zero of one of polynomials $ f(z)+z^ng(z)$, $n \in \N \cup \{0\}$. Then $\zeta$ is also the zero of resolvent polynomial $R(z)$ defined in Lemma \ref{ResolventLem}. Additionally, if  such $\zeta$ is a root of unity and $f(\zeta) \ne 0$, then the sequence $f(\zeta)+\zeta^ng(\zeta)$ is periodic, with period $T=\ord{\zeta}$. Conversely,  any common zero $\zeta \ne 0$ of two polynomials in the sequence $f(z)+z^ng(z)$, that is not a zero of $f(z)$, must be the zero of $R(z)$ and also a root of unity. In such a case, the set of all $n \in \N$ such that $\zeta$ is a zero of $f(z)+z^ng(z)$ is an arithmetic progression with difference equal to $\ord{\zeta}$.
\end{proposition}
\begin{proof}
The equality $f(\zeta)=-\zeta^ng(\zeta)$ means $\abs{f(\zeta)}=\abs{g(\zeta)}$, so $R(\zeta)=0$. Now, if $f(\zeta)+\zeta^{n_1}g(\zeta)=0$ and $f(\zeta)+\zeta^{n_2}g(\zeta)=0$, $n_2 > n_1$, with $f(\zeta) \ne 0$, then $g(\zeta) \ne 0$ and $\zeta^{n_1}=\zeta^{n_2}$. Hence, $\zeta$ is a root of unity with $\ord{\zeta} \mid (n_2 - n_1)$. The periodicity becomes obvious.
\end{proof}

For the automatic detection of singular exponents $n \in \N$, for which $f(z)+z^ng(z)$ contain repeated unimodular zeros, and the calculation of the value $n_0$ in Boyd's formula \ref{propBoyd}, we use the following result.

\begin{lemma} \label{MultiplicityLem} Suppose that $f(z)$, $g(z)$ and $R(z)$ are as in Lemma \ref{ResolventLem} and $R(z) \not\equiv 0$ in $\C[z]$. Define the resultant $S(z):=S[f, g](z)$ by
\begin{equation}\label{eqS}
S(z) := \mathrm{Res}_{w}\left(R(w), w(f'(w)g(w) - f(w)g'(w)) - zf(w)g(w)\right),
\end{equation}
where the resultant is computed in $\C[z, w]$ with respect to $w$.

The critical exponent value $n = n_0$ in equation \ref{eq_n0} of Proposition \ref{propBoyd} satisfies $n_0 = \Re{(\xi)}$, where $\xi$ is one of the zeros of $S(z)$.

If $h(z)=f(z)+z^ng(z)$ has a repeated zero at $z=\zeta \in \partial \DD$, $f(\zeta) \ne 0$, for $n=n_0$, then the critical exponent $n_0 \in \N \cup \{0\}$ must be the zero of $S(z)$, and $n_0 = \left|\left(f(\zeta)/g(\zeta)\right) ' \right|$.
\end{lemma}

\begin{proof}
Observe that the zeros of $S(z)$ are $\xi=\zeta f'(\zeta)/f(\zeta) - \zeta g'(\zeta)/g(\zeta)$, where $\zeta$ runs over all zeros of $R(z)$. This, in particular, allows another way to compute $n_0$ in equation \eqref{eq_n0}.

Assume that $\zeta$ is a multiple root of $h(z)$. Then, from, $h'(\zeta)=0$,  one obtains
\begin{equation}\label{eqMult}
f'(\zeta)+\zeta^ng'(\zeta) + n\zeta^{n-1}g(\zeta)=0.
\end{equation}
After multiplying both sides by $\zeta g(\zeta) \ne 0$, it is possible to eliminate all occurrences of $\zeta^n$ by using $\zeta^ng(\zeta)=-f(\zeta)$:
\begin{equation}\label{eq_zeta}
\zeta f'(\zeta) g(\zeta) - \zeta f(\zeta)g'(\zeta) - nf(\zeta)g(\zeta)=0.
\end{equation}
Hence, if $\zeta$ is a repeated unimodular root of $h(z)$, then one of the roots of $S(z)$ is $\xi=n_0$ by equation \eqref{eq_zeta}. For the last statement, divide equation \eqref{eq_zeta} by $g(\zeta)^2$ 
\begin{equation}\label{eq_derivs}
\zeta\left(\frac{f(\zeta)}{g(\zeta)}\right)' - n\left(\frac{f(\zeta)}{g(\zeta)}\right)=0,
\end{equation}
and use $\abs{f(\zeta)}=\abs{g(\zeta)}$, since $\zeta$ is unimodular.
\end{proof}

In principle, Lemma \ref{MultiplicityLem} allows us to find (or estimate) singular exponents $n=n_0$ without using any floating point arithmetic. 

\begin{figure}
\subcaptionbox{$u(t) \geq v(t)$\label{graph_uv}}{\includegraphics[width=0.49\linewidth]{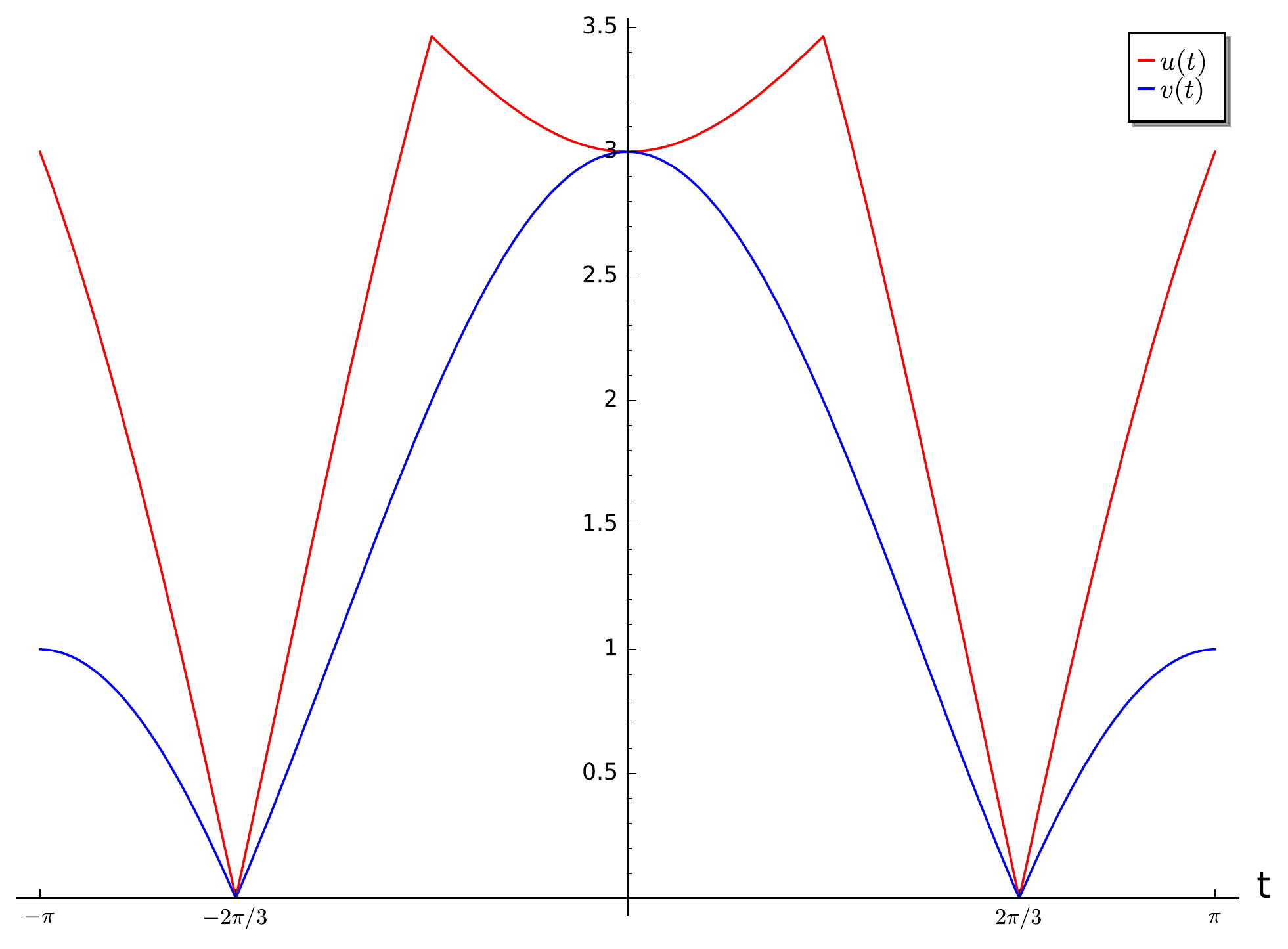}}
\subcaptionbox{$\abs{\cos{(3t/2)}}$ and $4\cos^2{(t/2)}$\label{graph_cos}}{\includegraphics[width=0.49\linewidth]{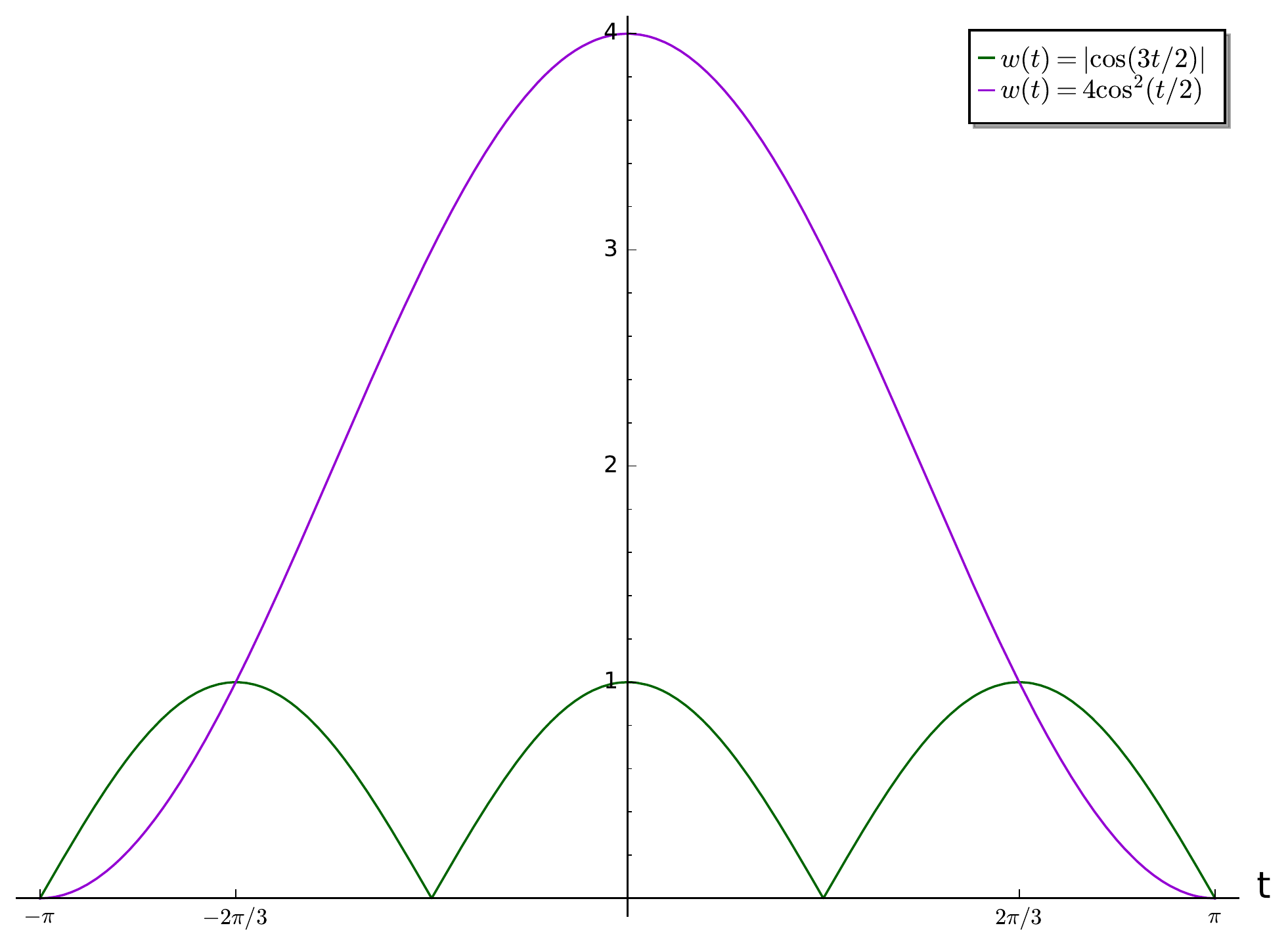}}
\vfill
\caption{Auxiliary functions used in the proofs of lemmas \ref{lemPrep1} and \ref{lemPrep2}}\label{fig_plots}
\end{figure}

We conclude this Section with two preparatory lemmas for Theorem \ref{thmSPL}.

\begin{lemma}\label{lemPrep1}
For $m, l \in \N$, let $f(z)=1 - z + z^2 - 2z^{2+3m}(1+z)$, $g(z)=1+z+z^2$. Then $\abs{f(z)} \geq \abs{g(z)}$ for $z \in \partial\DD$, with equality attained only at $z=1$, $z=e^{\pm 2\pi i/3}$. In particular, $e^{\pm 2\pi i/3}$ are the only simple unimodular zeros of $f(z)$.
\end{lemma}

\begin{proof}
For $z=e^{it}$, $t \in [-\pi, \pi)$, one has
\[
\abs{1 \pm z+z^2}=\abs{2\cos{t} \pm 1}, \qquad \abs{1+z}=2\cos{t/2}.
\]
Notice that
\[
\abs{f(e^{it})} \geq \abs{\abs{1-e^{it}+e^{2it}} - 2\abs{1+e^{it}}} = \abs{\abs{2\cos{t} -1}-4\cos{(t/2)}} := u(t).
\]
Set
\[
v(t) := \abs{g(e^{it})}=\abs{1+e^{it}+e^{2it}}=\abs{1+2\cos{t}}.
\]
By direct calculation, $u(t) \geq v(t)$, with equality at $t=0$, $\pm2\pi/3$, (see Figure \ref{graph_uv} for clarification). It follows that zeros of $g(z)$ are the only possibilities for $f(z)=0$ in $\partial\DD$. All it remains is to check that $f(\zeta)=0$, $f'(\zeta)=2 \mp \sqrt{3}(1+2m) \ne 0$, for $\zeta=e^{\pm 2\pi i/3}$.
\end{proof}

\begin{lemma}\label{lemPrep2}
Let $f(z)=1 - z + z^2 - 2z^{2+3m}(1+z)$, $m \in \N$. Then $U(f)=2$ and $N(f)=2m+1$.
\end{lemma}

\begin{proof}
Note that the formula $U(f)=2$ has been already established in Lemma \ref{lemPrep1}. Let
\[
q(z) := \frac{f(z)}{1+z+z^2} = \frac{1+z^3 - 2z^{2+3m}(1+z)^2}{(1+z)(1+z+z^2)}.
\]
Note that $q(z) \ne 0$ for $z \in \partial\DD$ by Lemma \ref{lemPrep1}. For $z=e^{it}$, $t \in \R$,
\[
q(e^{it}) = \frac{\cos(3t/2)- 4e^{(3/2+3m)it}\cos^2{(t/2)}}{\cos{(t/2)}(2\cos{t}+1)}.
\]
After separating real and imaginary parts, one obtains:
\[
\Re{q(e^{it})} = \frac{\cos(3t/2)- 4\cos^2{(t/2)}\cos{(3/2+3m)t}}{\cos{(t/2)}(2\cos{t}+1)},
\] and
\[
\Im{q(e^{it})} = \frac{- 4\cos^2{(t/2)}\sin{(3/2+3m)t}}{\cos{(t/2)}(2\cos{t}+1)}.
\]
Since $U(q)=0$, $\Re{q(e^{it})}$ and $\Re{q(e^{it})}$ cannot vanish simultaneously. Set
\begin{equation}\label{eq_var1}
\Phi(t) := \frac{\Im{\left(iq(e^{it})\right)}}{\Re{\left(iq(e^{it})\right)}}=\frac{\cos(3t/2)- 4\cos^2{(t/2)}\cos{(3/2+3m)t}}{4\cos^2{(t/2)}\sin{(3/2+3m)t}}.
\end{equation}

Restricting $t$ to the interval $t \in [-\pi, \pi]$, $\Phi(t)$ in equation \ref{eq_var1} has singularities at points
\[
t_j := \frac{2\pi j}{3+6m},  \qquad -1-3m \leq j \leq  1+3m,
\]
of order at most $1$ and singularities at the endpoints $t= \pm \pi$ of order at most $2$. For $t \in (t_j, t_{j+1})$ and also in the interval $(t_{3m+1}, \pi)$, the sign of the denominator of $\Phi(t)$ always coincides with the sign of
\[
\mathrm{sgn}\left(\sin{(1/2+3m)t}\right) =(-1)^j.
\]
At each endpoint  $t=t_j$, in the numerator of equation \eqref{eq_var1}, one has
\[
\cos{(3/2+3m)t_j}=(-1)^j.
\]
For $t \in (0, 2\pi/3)$, $\abs{\cos{(3t/2)}} < 4\cos^2{t/2}$, see Figure \ref{graph_cos}. Therefore, the numerator of $\Phi(t)$ has sign $(-1)^{j+1}$ at $t_j$, for $0 \leq j \leq 2m$.

The point $s \in \R$ is the singularity of  \emph{a first kind}, if $\Phi(t)$ jumps from $-\infty$ to $+\infty$ at $t=s$. Likewise, $s$ is a singularity of \emph{a second kind}, if $\Phi(t)$ jumps from $+\infty$ to $-\infty$ at $t=s$.

From the sign information one deduces that $\Phi(t)$ jumps from $+\infty$ to $-\infty$ at $t=t_j$, for $j=0, \dots, 2m$.

At $t_{2m+1}=2\pi/3$, laborious calculations yield
\[
\Phi(2\pi/3) = 0/0 = \lim_{t \to 2\pi/3} \Phi(t) =
-\left.\frac{d\Re{\left(q(e^{it})\right)/dt}}{d\Im{\left(q(e^{it})\right)/dt}}\right|_{t=2\pi/3}=  \frac{2\sqrt{3}}{3+6m} \ne 0,
\]
therefore $t=2\pi/3$ is a removable singularity that does not need to be accounted for.

For $t \in (2\pi/3, \pi)$, $\abs{\cos{(3t/2)}} > 4\cos^2{t/2}$, see Figure \ref{graph_cos}. This means that the sign of the numerator of Eq. \eqref{eq_var1} is equal to that of $\cos{(3t/2)}$, that is, negative, for $t > 2\pi/3$.

Hence, $\Phi(t) \to +\infty$ as $t \uparrow t_{2m+2}$. Further to the left, as the numerator stays negative for $2m+2 \leq j \leq 3m+1$, $\Phi(t)$ at $t=t_j$ jumps from $+\infty$ to $-\infty$ if $j$ is even, and from  $-\infty$ to $+\infty$ if $j$ odd.

At $t=\pi$, $\sin(3/2+3m)\pi=(-1)^{m+1}$, $\cos(3/2+3m)\pi=0$. This leads to
\[
\lim_{t \uparrow \pi} \Phi(t) = (-1)^{m+1} \lim_{t \uparrow \pi} \frac{\cos{(3t/2)}}{\cos^2{(t/2)}} = (-1)^{m+1} \lim_{t \uparrow \pi} \frac{3\sin{(3t/2)}}{2\sin{(t/2)}\cos{(t/2)}}
\]
\[
=(-1)^{m+2}\cdot \frac{3}{2}\lim_{t \uparrow \pi} \frac{1}{\cos{(t/2)}}=\begin{cases}
																					+\infty, & \text{if } m \text{ is even}\\
																					- \infty, & \text{if } m \text{ is odd}\\																				\end{cases}.
\]
A similar calculation yields
\[
\lim_{t \downarrow \pi} \Phi(t)  =\begin{cases}
										-\infty, & \text{if } m \text{ is even}\\
										+\infty, & \text{if } m \text{ is odd}\\
										\end{cases}.
\] 
Therefore, $\Phi(t)$ at $t=\pi$ jumps from $+\infty$ to $-\infty$ if $m$ is even, from $-\infty$ to $+\infty$ if $m$ is odd.

The number of singularities of the first kind in the interval $(0, \pi)$ is
\[
N_{0}^{\pi}(\Phi)^+  = \begin{cases}
										m/2,		& \text{ if } m \text{ is even},\\
										(m-1)/2,	& \text{ if } m \text{ is odd},\\
							\end{cases} 
\]
while the total number of singularities of the second kind is
\[
N_{0}^{\pi}(\Phi)^-  = \begin{cases}
										2m + m/2,	   & \text{ if } m \text{ is even},\\
										2m + (m+1)/2, & \text{ if } m \text{ is odd}.\\
							\end{cases}
\]
Since $\Phi(t)$ is an odd function in $[-\pi, \pi]$, the mapping $\Phi(t) \mapsto -\Phi(-t)$  is a bijection from singularities in $(-\pi, 0)$ to the singularities $(0, \pi)$ that preserves the kind of singularities. One must also account for a singularity at $t=0$ of the second kind, and a singularity at $t=\pi$ that is of the second kind if $m$ even, and of the first kind if $m$ odd. Choosing $\eps > 0$, such that $\Phi(-\pi+\eps)=\Phi(\pi + \eps) \not\in \{0, \pm \infty\}$, from $2\pi$-periodicity of $\Phi(t)$ one obtains:
\[
N_{-\pi+\eps}^{\pi+\eps}(\Phi)^+ = m, \qquad N_{-\pi+\eps}^{\pi+\eps}(\Phi)^- =2+5m.
\]
The Cauchy index of $\Phi$ in the interval $[-\pi+\eps, \pi+\eps]$ is
\[
I_{-\pi+\eps}^{\pi+\eps}(\Phi) = N_{-\pi+\eps}^{\pi+\eps}(\Phi)^+ - N_{-\pi+\eps}^{\pi+\eps}(\Phi)^- = -2 - 4m.
\]
By Cauchy Index Theorem (Theorem 11.1.3 on p.358 in \cite{RahSch}), the number of zeros of $iq(z)$ inside $\DD$ is
\[
N(iq(z)) = -I_{-\pi+\eps}^{\pi+\eps}(\Phi)/2 = 2m+1.
\]
Obviously, $N(f)=N(iq)=2m+1$.
\end{proof}

\section{Computations}\label{sec_comp}

\subsection{Large Newman polynomials}\label{sec_largeMin}

Several authors  looked at $\{0, 1\}$--polynomials that stay large (in absolute value) on $\partial\DD$. Smyth \cite{Smy1} proved that the largest minimum on the unit circle, taken over all Newman polynomials of a given degree, $\mu(n) := \max_{f \in \NN_n} \min_{z \in \partial\DD}\abs{f(z)}$, tends to $\infty$ with $n$. Boyd \cite{B6} strengthened this by proving $\mu(n) \gg n^{0.137}$ and tabulated $\mu(n)$ for $1 \leq n \leq 21$ thereby extending previous computations done by Campbell, Ferguson and Forcade \cite{CFF}. Boyd also found some examples of Newman polynomials of large degree and large minima: for instance, in Appendix A of \cite{B6} he listed $f \in \NN_{77}$ with $\abs{f(z)} > 1.47$. In \cite{B6}, Section $6$, Part $3$ he noted that it would be of interest to find the smallest integer $n_0$ for which there exists a $f(z) \in \NN_{n_0}$ with $\abs{f(z)} \geq 2$ and gave an estimate $n_0 \leq 272$. More recently Mercer posed variants of Boyd's question in \cite{Mer2, Mer3}. 

Our computations indicate that $n_0=38$; it arises from the `super--polynomial' $f(z)$ in equation \eqref{superNewman} that was used to prove Theorem \ref{mainNewmanThm}.  In order to find it, all possible $\{0, 1\}$ polynomials of degree at most $40$ were generated by the method described in Section \ref{sec_gen}, avoiding repetitions of reciprocal as well as the self--reciprocal polynomials (the later are known to vanish on $\partial\DD$, \cite{BEFL, KovMat, Mer}). Then, for each generated $f(z)$, the minimum of values $\abs{f(e^{2\pi i j/m})}$, $0 \leq j \leq m-1$ was computed at $m=128$ sample points in double precision using the fast Fourier transform library \texttt{FFTW} \cite{FriJoh} (v.3.3.8); polynomials $f(z)$ with values $<2$ were filtered out. The choice of the parameter $m=128$ can be explained as the smallest power of $2$ that is $ \approx 4\deg{f}$: it provides a good balance between the high speed of the computation and the small number of surviving `false positive' cases with true minima dipping below $2$ somewhere in between of the sampling points. The filtering procedure for $n \leq 40$ took about 94.6 hrs to complete on our machine; degree $n=38$ alone took about $11$ hrs. The survivor polynomials were then filtered one more time through FFT by increasing $m$ to $32678$, until first Newman super--polynomial became apparent at $n=38$. It was then verified by solving for the critical points of its derivative $\partial{dt}\abs{f(e^{it})}^2/\partial{t}$, $t \in [0, 2\pi)$ and evaluating $\abs{f(e^{it})}$ at these points in \texttt{Sage}.

\subsection{Heuristic pattern search}\label{sec_heuristic}

Patterns that produce infinite sequences of polynomials $f(z) \in \NN_n$ or $\LL_n$ with prescribed values $N(f)=k$, listed in Tables \ref{patternsNewman2}, \ref{patternsNewman345}, \ref{patternsLittlewood2}, \ref{patternsLittlewood3to11}, were found using the heuristic search procedure that attempts to find a suitable extension to already known good pattern of small length.

We start with \emph{a seed polynomial} $f_{\mathrm{seed}}(z) \in \NN$ or $\LL$ with $N(f_{\mathrm{seed}})=k$ zeros in $\DD$.  We look for seed polynomials by searching the sets $\NN_n$, $\LL_n$ of small degree, typically $6 \leq n \leq 12$.

Let $u = a_0 a_1 \dots a_n$ be the word on alphabet $\mathcal{A}=\{0, 1\}$ (in the Newman case) or $\mathcal{A} =\{-1, 1\}$ (in the Littlewood case) that represents the seed polynomial $f_{\mathrm{seed}}(z)$, written as in equation \eqref{genForm}. Together with the word $u$, denote by $W$ the collection of words $w=w_1 w_2 \dots w_l$ of length $l=\abs{w} \leq L$ on $\mathcal{A}$. In our searches, we typically looked at $L=1, 2, 3, 4, 5, 6$ -- length parameters were selected according to which degree progressions $n \in \N$ were to be covered.

For every $w \in W$ and for every possible factorization of $u$ into the prefix and suffix  parts $u=u'u''$, where one of the $u'$, $u''$ is also allowed to be empty word, we try to insert $w$ in between $u'$ and $u''$. In this way, we obtain a set of \emph{new seed words}
\[
\mathcal{S}_{\mathrm{new}}:=\mathcal{S}_{\mathrm{new}}(u, W) = \{u'wu'': u'u''=u, w \in W\}.
\]
For a word $u_{\mathrm{new}} \in \mathcal{S}_{\mathrm{new}}$, we compute the polynomial $f_{\mathrm{new}}(z) \in \NN$ or in $\LL$ whose coefficients are represented by the word $u_{\mathrm{new}}$ as in equation \eqref{genForm}. If $N(f_{\mathrm{new}}) \ne k$, then $u_{\mathrm{new}}$ is removed from $\mathcal{S}_{\mathrm{new}}$. If $N(f_{\mathrm{new}}) = k$, then we consider $f_{\mathrm{new}}(z)$ to be the new seed polynomial, and perform the search procedure on $f_{\mathrm{new}}(z)$ recursively, in a depth--first order. This is repeated for each $u_{\mathrm{new}} \in \mathcal{S}_{\mathrm{new}}$. If the same seed polynomial is encountered for the second time during the search, it is rejected.

The recursion is continued until the length of a seed word $u_{\mathrm{new}}$  exceeds certain pre-determined length $\abs{u_{\mathrm{new}}} \geq	 N$ (we used values $N \leq 60$). Then, one tries to factor the word $u_{\mathrm{new}}=p r^m s$ into prefix $p$, repeating part $r$ and suffix $s \in \mathcal{A}$, where $p$ or $s$ can be empty, but $r$ is always non--empty and is chosen in such a way that the number of repetitions $m \geq 2$ is maximized.  In the case of the alphabet $\mathcal{A}=\{-1, 1\}$, the factorization is also attempted for the word $\overline{u}_{\mathrm{new}}=a_0 \overline{a_1}a_2 \overline{a_3}\dots$, where $\overline{a}_j := -a_j$, that corresponds to the polynomial $f_{\mathrm{new}}(-z) \in \LL_n$. It should be noted that these factorizations, in general, are not unique: only one of them is reported.

\begin{example} We give a simple example for the $k=2$ case.

We start with $f_{\mathrm{seed}}(z)=1+z+z^2+z^4+z^5$, $N(f_{\mathrm{seed}})=2$ that corresponds to $u_{\mathrm{seed}}=111011$. We take $W = \{0, 011\}$ and set the maximal word length to $N=40$.

By inserting from $W$ into $u_{\mathrm{seed}}$, and rejecting words that produce polynomials with $N(f_{\mathrm{new}}) \ne 2$, we find that $u_{\mathrm{seed}}$ has $3$ possible children:
\[
\mathcal{S}_{\mathrm{new}} = \{ 1110{\underline 0}11, 111 \underline{011}011, 11101{\underline 0}1\}.
\] Here the inserted word is underlined. There were two different ways of inserting $0$ into $111011$ to get $1110011$, either as the first $0$ or as the second $0$. Since these insertions produce the same word, it is treated as a single descendant.

By depth first search, we find that $u_{\mathrm{new}}= 1110011 \in \mathcal{S}_{\mathrm{new}}$ has only one child $1{\underline 0}110011$. This word, in turn, also has only one child, $1011{\underline 0}0011$. This again has only one child, $10110 \underline{011}0011$. This last word has no children, and the search at this particular branch terminates.

The search on $u_{\mathrm{new}}= 111011011$ produces the sequence of children: 
\[
111011011, \qquad 111011011011, \qquad 111011011011011, \dots.
\]
When the length of the child reaches $\geq N= 40$, the recursion stops and the last child word in the branch is factored to find the longest run of repetitions. In our case,
\[
111011011011011011011011011011011011011011 = 111(011)^{13}.
\]  
The pattern $111(011)^m$ is reported and the search backtracks to the remaining child in $\mathcal{S}_{\mathrm{new}}$.

In a similar way, the described algorithm determines that $1110101$ has no descendants of length $\geq N$.

It should be noted that word $11101001101$ appears twice: as the child of $11101001$ ($111010\underline{011}01$) or the child of $1110101101$ ($111010\underline{0}1101$). This repeating word and its children are examined only once and the second occurrence of it in the tree is pruned.

The search tree is illustrated in Figure \ref{fig:tree}. For each child $u_{\mathrm{new}}$, the location of the insert $w$ is underlined. The nodes in a box have no children and terminate that branch. There is one branch that does not appear to terminate and produces a pattern with repetitions that is listed in Table  \ref{patternsNewman2}. It will be further examined and validated, as described in Section \ref{sec_autoproof}.
\end{example}

\begin{sidewaysfigure}
\vspace{11cm}
\begin{tikzpicture}
[ level 1/.style={sibling distance=-1.7cm}]
\Tree [.111011 
        [.1110\underline{0}11
          [.1\underline{0}110011 
            [.1011\underline{0}0011
              [.\node[draw]{10110\underline{011}0011}; ] 
            ]
          ]
        ]
        [.111\underline{011}011
          [.111\underline{011}011011
            [.111\underline{011}011011011
              [.111\underline{011}011011011011
                [.111\underline{011}011011011011011
                  [.111\underline{011}011011011011011011
                    [.111\underline{011}011011011011011011011
                      [.111\underline{011}011011011011011011011011
                        [.111\underline{011}011011011011011011011011011
                          [.111\underline{011}011011011011011011011011011011
                            [.111\underline{011}011011011011011011011011011011011 $\vdots$ ] 
                          ] 
                        ] 
                      ] 
                    ] 
                  ] 
                ] 
              ] 
            ] 
          ] 
        ]
        [. 11101\underline{0}1
          [.11101\underline{0}01
            [.\node[draw]{111\underline{0}01001}; ]
            [.\node[draw]{111\underline{011}01001}; ] 
            [.\node[draw]{111010\underline{011}01}; ] 
            [.\node[draw]{11101001\underline{011}}; ] 
          ]
          [.11101\underline{011}01
            [.\node[draw]{111\underline{011}0101101}; ] 
          ]
        ]
      ]
\end{tikzpicture}
\caption{Search tree for $k = 2$, $u_{\mathrm{seed}}=111011$ and $W = \{0, 011\}$}
\label{fig:tree}
\end{sidewaysfigure}

\subsection{Proving $N(h)=k$ and $U(h)=0$}\label{sec_autoproof}

A computer--aided procedure was carried out in \texttt{Sage} to prove the formula for $N(h)=k$ and ensure $U(h)=0$ for each polynomial $h(z)$ in Tables \ref{patternsNewman2}, \ref{patternsLittlewood2}, \ref{patternsNewman345} and \ref{patternsLittlewood3to11}.

First, from a given pattern of coefficients of the form $uv^mv$ on alphabet $\mathcal{A}=\{0, 1\}$ or $\mathcal{A}=\{-1, 1\}$ that corresponds to the polynomial $h(z)$, one builds the rational function representation
\[
h(z) = \frac{f(z)/d(z)+z^{s+tm}g(z)/d(z)}{(1-z^t)/d(z)};
\] here $t=|v|$ denotes the length of the period in the pattern, $m$ is the number of repetitions of the period $v$, $s=|u|$ is the prefix length, and $d(z)$ is the polynomial g.c.d. of $f(z)$, $g(z)$ and $1-z^t$ that needs to be factored out. Expressions for $h(z)$ are provided in aforementioned tables.

After that, one applies Lemma \ref{ResolventLem} in order to verify that $\abs{f(z)} \geq \abs{g(z)}$ on $\partial \DD$ and to find all possible unimodular zeros of the numerator polynomial: these are the only points $z \in \partial \DD$, where $\abs{f(z)}=\abs{g(z)}$ can hold. For every pattern listed in Tables \ref{patternsNewman2}, \ref{patternsLittlewood2}, \ref{patternsNewman345} and \ref{patternsLittlewood3to11}, we computed resolvents $R(z)$. In all cases, the unimodular zeros $\zeta$ of resolvent polynomials had even multiplicities and their middle terms were positive. We found that unimodular zeros of $h(z)$ were roots of unity of small orders $\ord{\zeta} \in \{1, 2, 3, 4, 5, 6, 8, 24\}$. We determined all arithmetic progressions $m \pmod{\ord{\zeta}}$ in $\N$, were $h(\zeta)=0$ might have such a zero: for this, it suffices to check the unimodular factors of $h(z)$ for each non-negative integer $0 \leq m \leq 24$ and then apply the periodicity property of Proposition \ref{UnimodularProp}. We calculated critical exponents $m_0$, such that all unimodular zeros of the numerator are be simple for every $m > m_0$ by the method described in Lemma \ref{MultiplicityLem}. We verified that unimodular factors in the numerator cancel out with denominator $1-z^t$, leaving $U(h) = 0$ for $m > m_0$; if they do not cancel, we exclude all bad arithmetic progressions $m \pmod{ \ord{\zeta}}$ where $U(h) \ne 0$ from the table.

By Boyd's formula \ref{propBoyd}, $N(h)=N(f)$ holds from certain threshold $m > m_1$; we found that this value $m_1$ always coincided with the critical exponent $m_0$,  at which polynomial $h(z)$ had double unimodular zeros. In our case, critical exponents were small: $m_0=0$, $1$, $2$, $3$, $4$: they are listed in aforementioned tables. By computing $N(f)=k$ on the computer, we completed proving the formulas $N(h)=k$ for $m \geq m_0$.

 \subsection{Gray codes and normalization}\label{sec_gen}
 
 For Newman and Littlewood polynomials, one has $\abs{f(z)}=\abs{f^*(z)}$ when $z \in \partial\DD$. Also, $U(f^*)=U(f)$, $N(f^*)=\deg{f} - N(f)-U(f)$. Thus, it suffices to test only one member of the pair $f(z)$, $f^*(z)$ to determine the zero counts or absolute minima of both. As degrees become large, avoiding computations on the reciprocal cuts down the computation time almost by a half.

For Newman polynomials, this can be accomplished as follows: the coefficients of every $f(z) \in \NN_{n}$, $n \geq 1$ can be written as a pattern $1uvw^*1$, where $u$, $v$ and $w$ denote (possibly empty) binary words on the alphabet $\{0, 1\}$ of length $l=|u|=|w|=\lfloor(n-1)/2\rfloor$, $\abs{v} \leq 1$, and $w^*$ stands for the reflection of $w$. Let $G: [1, 2^l] \mapsto [0, 2^l-1]$ be \emph{binary reflected Gray code} \cite{Gr} of the length $l$ of an integer $i \in [1, 2^l]$. Since the function $G$ is bijective, one can always find integers $i$ and $j$ such that $u=G(i)$, $w=G(j)$, $i, j \in [1, 2^l]$. By looking only at words $u$ and $w$ where $i \leq j$, one produces exactly one of the polynomials $f(z)$, $f^*(z)$, but not the other, when $f(z) \ne f^*(z)$. If self-reciprocal polynomials are not desirable as well, then one should consider only $i < j$.

For $f \in \LL_n$, it is also possible to generate only one member from the pair $f(z)$, $f^*(z)$. However, it is more complicated;  as we did not reach degrees $n$ as in the Newman case, this was not attempted. Instead, the normalization $a_0 = a_1 = 1$ was performed to compute only one of the four polynomials $f(z)$, $f(-z)$, $-f(z)$, $-f(-z)$. This was accomplished by the patterns of the form $11u$, where $u=\widetilde{G}(i)$ is the binary word on $\{-1, 1\}$ obtained from Gray code $G(i)$ of $i$ by replacing all `$0$'s with `$-1$'s.

Gray codes speed up the computation as no additional inner loop to reset all coefficients is required -- only one is flipped at a time.  For polynomials in $\NN$ and $\LL$, Gray codes were  already widely used in computations by M.~J.~Mossinghoff and his coauthors \cite{BoMo2, BoHaMo, Mo1, Mo2}.

\subsection{Exhaustive computations}\label{sec_exhaustive}

The exhaustive computation of numbers $N(f)$ and $U(f)$ for every $f \in \NN_n$, $2 \leq n \leq 35$ and every $f \in \LL_n$, $2 \leq n \leq 31$ was carried out in order to:
\begin{enumerate}
 \item check if heuristic searches did not miss any regular sequences of Newman and Littlewood polynomials with $N(f)=2$ besides those listed in Table \ref{patternsNewman2}, \ref{patternsNewman345}, \ref{patternsLittlewood2}, \ref{patternsLittlewood3to11};
 \item identify sporadic cases  that do not fit any of the regular pattern;
 \item gather the data on the frequency distribution of $N(f)$, described in Section \ref{sec_stats}.
 \end{enumerate}
 We implemented the Bistritz algorithm \cite{Bis} for the \texttt{fmpq\_poly\_t} type from \texttt{FLINT} library \cite{Hart} v.2.5.2 in \texttt{C}. A \texttt{C} generator program invoked the Bistritz rule on Newman and Littlewood polynomial produced as outlined in Section \ref{sec_gen}. We were quick to find out that it is highly inefficient to store the values of $N(f)$, $U(f)$  and coefficients of generated polynomials in text: as $n$ approached $30$, the required disk space exceeded 100 GB. As a remedy, we implemented buffered binary data files, where the values of $N(f)$ and $U(f)$ are packed as two byte--size records in one large array; the position of the record corresponds to the order in which the polynomial $f(z)$ was generated sequence. After finishing the computations of $N(f)$ and $U(f)$ for each degree, our \texttt{C} program also calculated the statistics and stored it in a separate text file. To extract specific polynomials for post-processing, we wrote another \texttt{C} program that loops over the computed binary data file (in the same order as they were generated) and outputs polynomials with required $N(f)$ and $U(f)$ numbers to a separate text file.  In the Newman case, the combined size of the binary data files for degrees $2 \leq n \leq 35$ was $32$ GB ($16$ GB file for degree $n=35$ alone); in the Littlewood case, $4$ GB of data files for $2 \leq n \leq 31$ ($2$ GB for the last degree $n=31$) were generated. The total computation time in the Newman case was $389.8$ CPU hours ($202.5$ hrs for the last degree $n=35$); the corresponding time in the Littlewood case for all degrees was $66.4$ hrs, for $n=31$ it took $45$ hrs.
 
\subsection{Software and hardware}

Small degree computations (depending on the task, for degrees $n \leq 15$ or $\leq 18$) and the post--processing (statistics calculations in Section \ref{sec_stats} or the verification procedure in Section \ref{sec_autoproof}) were done in \texttt{Sage} v.8.1. \cite{Sage} and \texttt{Maple}. Heuristic pattern searches were accomplished on \texttt{Maple}\texttrademark \cite{WMaple}\footnote{Maple is a trademark of Waterloo Maple Inc.}. Larger degree computations were coded in \texttt{C}, compiled and built by \texttt{gcc} v.8.3.0. on \texttt{MinGW-w64} \cite{mingw64} toolchain provided by \texttt{MSYS2} \cite{msys2}. The first named author worked on his 2.2GHz \texttt{Intel Core i7} 8GB 1600Mhz DDR3 laptop and the 3.8GHz \texttt{Intel Core i5} 16 GB 2400 MHz DDR4 desktop machine. The second named author performed his computations on his personal 2.2Ghz \texttt{Intel i5-5200} $4$--core Windows 10 laptop with 1600MHz 8GB DDR3 RAM.

\section{Acknowledgments}

We would like to thank M.~J.~Mossinghoff for references to the earlier work \cite{BLS, G1, G2} related to our subject. Michael pointed to the polynomial that can be used to prove the $\{0, 1\}$--admissibility of pairs $(5,  2n+1), n \in\N$ and provided few other random small degree examples in the Newman case.

The authors have benefited from discussions during 2015 \emph{Algebraic numbers and polynomials} event at Banf International Research Station (BIRS) and \emph{Numeration} conferences at Universit\'e  Paris Diderot -- Paris 7 in 2018 and Erwin Schr\"{o}dinger Institute in Vienna in 2019.

\end{document}